\documentclass[twocolumn]{autart} 

\usepackage{braket,amsfonts}
\usepackage{graphicx}
\usepackage{amsmath}
\usepackage{amssymb}
\usepackage{amsfonts}
\usepackage{hyperref}
\usepackage[caption=false]{subfig}
\usepackage{pgfplots}
\usepackage{graphics}
\usepackage{color}
\usepackage{epsfig}
\usepackage{cite}
\usepackage{float}
\usepackage{wrapfig}
\usepackage{bbold}
\usepackage{appendix}
\usepackage{hyperref}
\usepackage{textcomp}
\usepackage{graphicx,color}
\usepackage{mathrsfs}
\usepackage{url}
\usepackage{color}
\usepackage{dsfont}
\usepackage{bbm}
\usepackage{verbatim}
\usepackage{tikz}

\newtheorem{theorem}{Theorem}
\newtheorem{lemma}{Lemma}
\newtheorem{proposition}{Proposition}
\newtheorem{corollary}{Corollary}
\newtheorem{definition}{Definition}
\newtheorem{remark}{Remark}
\newtheorem{assumption}{Assumption}

\newenvironment{proof}[1][Proof]{\begin{trivlist}
\item[\hskip \labelsep {\bfseries #1}]}{\end{trivlist}}

\newcommand{\supscr}[2]{#1^{\textup{#2}}}
\newcommand{\until}[1]{\{1,\dots,#1\}}

\newcommand{\real}{{\mathbb{R}}}

\newcommand{\realnonnegative}{{\mathbb{R}}_{\ge 0}}

\newcommand{\dvol}{\operatorname{dvol}}
\newcommand{\vol}{\operatorname{vol}}
\newcommand{\diff}{\operatorname{d}}

\newcommand{\supp}{\operatorname{supp}}

\newcommand{\id}{\operatorname{id}}
\newcommand{\Lip}{\operatorname{Lip}}
\newcommand{\lip}{\operatorname{lip}}

\newcommand\oprocendsymbol{\hbox{$\square$}}
\newcommand\oprocend{\relax\ifmmode\else\unskip\hfill\fi\oprocendsymbol} 

\makeatletter
\renewcommand*{\@opargbegintheorem}[3]{\trivlist
      \item[\hskip \labelsep{\bfseries #1\ #2}] \textbf{(#3)}\ \itshape}
\makeatother

\begin{document}

\begin{frontmatter}

\title{A Multiscale Analysis of Multi-Agent \\ Coverage Control Algorithms}

\thanks[footnoteinfo]{This material is based upon work supported  by grant AFOSR FA-9550-18-1-0158.}

\author[krishnan1,krishnan2]{Vishaal Krishnan}\ead{vkrishnan@seas.harvard.edu},     
\author[martinez]{Sonia Mart{\'i}nez}\ead{soniamd@ucsd.edu},               

\address[krishnan1]{School of Engineering and Applied Sciences, Harvard University, Cambridge MA 02138 USA.}                               
\address[krishnan2]{Lund Vision Group, Department of Biology, Lund University, 223 62 Lund, Sweden.}                                            
\address[martinez]{Department of Mechanical and Aerospace Engineering, University of California, San Diego, La Jolla CA 92093 USA.}


\begin{abstract}    
This paper presents a theoretical framework for the design and 
analysis of gradient descent-based algorithms for coverage control 
tasks involving robot swarms. We adopt a multiscale 
approach to analysis and design to ensure consistency of
the algorithms in the large-scale limit.
First, we represent the macroscopic configuration of the swarm 
as a probability measure and formulate the macroscopic coverage 
task as the minimization of a convex objective function over 
probability measures. We then construct a macroscopic 
dynamics for swarm coverage, which takes the form of a 
proximal descent scheme in the $L^2$-Wasserstein space.
Our analysis exploits the generalized geodesic convexity of 
the coverage objective function, proving convergence in the 
$L^2$-Wasserstein sense to the target probability measure.
We then obtain a consistent gradient descent algorithm in the 
Euclidean space that is implementable by a finite collection of 
agents, via a ``variational'' discretization of
the macroscopic coverage objective function.
We establish the convergence properties of the gradient descent 
and its behavior in the continuous-time and large-scale limits.
Furthermore, we establish a connection with well-known 
Lloyd-based algorithms, seen as a particular class of algorithms   
within our framework, and demonstrate our results
via numerical experiments.                       
\end{abstract}

\end{frontmatter}

\section{Introduction}
Multi-agent systems are groups of autonomous agents with sensing,
communication, and computational capabilities.  It is often necessary
to achieve a desired coverage of a spatial region before these systems
can be deployed for specific purposes. This has spurred intense
research activity on the design of multi-agent coverage control
algorithms~\cite{JC-SM-TK-FB:02-tra, JC:08-tac}.   
In spatial coverage control problems involving
large-scale multi-agent systems, it is often more
appropriate and convenient to specify the task objective at the
macroscopic scale for the distribution of agents over the spatial region.
However, actuation still rests at the microscopic scale at the level
of the individual agents, and faces a multitude of constraints imposed
by the multi-agent setting.  These include information constraints
from limitations on sensing, communication and localization, and
physical constraints such as collision and obstacle avoidance.  This
separation of scales poses a problem for the analysis and design of algorithms with
performance guarantees. While \textit{mechanistic} models relying on
theoretical tools from infinite-dimensional analysis are often more
appropriate for macro scales, an \textit{algorithmic} approach that
relies on tools from finite dimensional analysis is more effective in
addressing the above microscopic constraints.  This underscores the
need for a formal theory bridging the two scales. Such a bridge theory
is crucial for integrating the mechanistic and algorithmic paradigms
and in understanding how macroscopic coverage objectives translate to
the microscopic level of individual agents and conversely, how the microscopic
algorithms shape macroscopic behavior. 

\textbf{Related work.} Multi-agent coverage control algorithms have
been widely studied over the past two decades and have a rich
literature. For an (inexhaustive) overview of the literature, we adopt the
classification into \textit{mechanistic} vs \textit{algorithmic}
models, as introduced earlier.
The algorithmic perspective is predominantly based on tools from
distributed optimization.  Initial works combined distributed
optimization with ideas from computational geometry and dynamic
systems, applying the well-known Lloyd algorithm~\cite{SL:82} 
for quantization of signals to the multi-agent setting
\cite{JC-SM-TK-FB:02-tra, MZ-CGC:11, AB-MS-JCM-RS-DR:10}.
In this sense, coverage algorithms can be understood as obtaining 
a quantization of the underlying spatial domain.
Furthermore, the problem setting in coverage control has been 
extended to include sensing, energy, obstacle
and collision avoidance constraints encountered in
the multi-agent scenarios~\cite{JC-SM-FB:05-esaim, YR-SM:13, SB-NM-VK:13}.
Interest in the mechanistic perspective, fueled by efforts to scale
up the size of these systems, emphasized the need for tools of
macroscopic analysis. Specifying the configuration of the multi-agent system 
as a probability measure or probability density naturally led to the application 
of mathematical tools from probability, stochastic processes and partial differential
equations (PDE). One approach involves the use of PDE-based models,
applying ideas of diffusion/heat flow to coverage control~\cite{VK-SM:18-sicon, TZ-QH-HL:20}.
Tools from parameter tuning and boundary control of PDEs~\cite{PF-MK:11,FZ-AB-KE-SB:18} 
have been used in this context. Statistical physics-based approaches, including the application of
mean-field theory, have also been recently explored~\cite{KE-SB:19, KE-ZK-AS-SB:20}.
Another approach involves the design of coverage by synthesis of Markov transition
matrices~\cite{SB-SJC-FYH:13, ND-UE-BA:15, SB-SJC-FH:17, MEC-YY-BA-MO:18}.
Some works at the intersection of the microscopic and macroscopic
perspectives include~\cite{FZ-AB-KE-SB:18}, where the authors obtain
performance bounds for spatial coverage by multi-agent
swarms, characterizing coverage performance as a function of
the number of robots and robot sensing radius.

Energy considerations in multi-agent transport have more recently resulted in the adoption of
tools from optimal transport theory~\cite{JDB-YB:00,NP-GP-EO:14,YC-TTG-MP:21,YC-TTG-MP:15-i,YC-TTG-MP:15-ii,YC-TTG-MP:18}.  
 More fundamentally, the need for ideas from optimal transport theory~\cite{CV:08} arises 
 from the utility of optimal transport metrics for
 analysis in the space of probability measures, mostly importantly the
 Wasserstein distance.  This allows for establishing a connection
 between the gradient descent-based algorithmic approaches to
 multi-agent coverage and gradient flows in the space of probability
 measures. For a detailed treatment of the theory of gradient flows
 in the space of probability measures, we refer the reader to
 \cite{LA-NG-GS:08}.
 Some well-known transport PDEs can
be formulated as gradient flows on functionals in the space of
probability measures~\cite{LA-NG-GS:08}.  Furthermore, from a
computational perspective, gradient flows in the space of probability
measures are often discretized into particle gradient flows.  The
gradient flow structure underlying these PDEs allows for their
discretization by formulating proximal gradient descent schemes in the
space of probability measures.  For instance, in~\cite{RJ-DK-FO:98}
the authors discretize the well-known Fokker-Planck equation by a
proximal recursion. In~\cite{AS-AK-GL:20} the authors present a non-asymptotic 
analysis of proximal recursions in the $L^2$-Wasserstein space.
 In \cite{LC-FB:18}, the authors investigate the
convergence of such particle gradient flows to global minima in the
limit $N \rightarrow \infty$.
In \cite{KC-AH:19-acc}, the authors apply
proximal descent schemes to study uncertainty propagation in
stochastic systems.
 
Optimal transport theory also has underlying connections to 
the problem of quantization~\cite{QD-VF-MG:99, DB-SR:15}, 
which as described earlier has well-known connections to the
coverage control problem.
This application of ideas from optimal transport to multi-agent 
coverage control remains an active area of 
research~\cite{VK-SM:18-cdc, GF-SF-TAW:14,
SF-GF-PZ-TAW:16, SB-SJC-FYH:14, MHB-UE-BA-MM:18}.
The various applications of optimal transport 
have motivated a search for efficient
computational methods for the optimal transport 
problem, and we refer the reader to~\cite{GP-MC:17} for a
comprehensive account.
Entropic regularization of the Kantorovich formulation
has been an efficient tool for approximate 
computation of the optimal transport cost 
using the Sinkhorn algorithms~\cite{MC:13}, \cite{MC-GP:18}.
Data-driven approaches to the computation of the optimal transport
cost between two distributions from their samples have been 
investigated in \cite{ET-GT:14, MK-ET:17}, and
with an eye towards large-scale problems in \cite{AG-MC-GP-FB:16},
\cite{VS-BD-RF-NC-AR-MB:17}, \cite{QM:11}. 
A related problem of computation of Wasserstein barycenters was 
addressed in~\cite{MC-AD:14}.
While computational approaches to optimal transport
often work with the static, Monge or Kantorovich formulations
of the problem, investigations involving dynamical formulations was initiated 
by~\cite{JDB-YB:00}, where the authors recast the~$L^2$ 
Monge-Kantorovich mass transfer problem in a fluid mechanics 
framework. The problem of optimal transport was also 
explored from a stochastic control perspective in~\cite{TM-MT:08} 
and~\cite{YC-TG-MP:16}, where the latter further explored connections to Schrodinger bridges.

\textbf{Contributions.}
%
%
  This paper contributes a multi-scale analysis of gradient 
  descent-based coverage algorithms for multi-agent systems, 
  with three main goals in mind: 
  (i) the formalization of coverage objectives for large-scale 
  multi-agent systems via meaningful macroscopic metrics, 
  (ii) the systematic design of provable correct algorithms that are 
  consistent across the macroscopic and microscopic scales, and 
  (iii) to gain a fundamental understanding of widely studied coverage 
  algorithms for large-scale multi-agent systems
  and shed new light on their behavior as the number of agents 
  $N \rightarrow \infty$. 
  A suitable theoretical framework for the above is largely missing 
  in the literature and this work addresses the gap.

  We formulate the coverage task as a minimization in the space of
  probability measures and define a proximal gradient descent on the
  aggregate objective function.  The multi-agent configuration is
  specified by discretizing the underlying probability measure and we
  obtain implementable coverage algorithms as a proximal gradient
  descent on the discretized aggregate objective function w.r.t. agent
  positions. This leads to a new class of ``variational'' gradient
  algorithms, and we show that this class of algorithms subsumes
  previously defined coverage algorithms based on distortion
  metrics. This allows us to establish a connection between the
  macroscopic and microscopic perspectives and present a unified
  theory of multi-agent coverage algorithms.
  
  \textbf{Paper outline.} The rest of the paper is organized as
  follows. Section~\ref{sec:coverage_optimization} contains a
  description of the coverage optimization problem setting.
  In Section~\ref{sec:preliminaries}, we present the mathematical
  preliminaries that underlie the main results in the paper.
   In Section~\ref{sec:transport_models}, we present an iterative descent
  scheme in the space of probability measures and establish
  convergence results for such a scheme. Building on these results, we
  propose multi-agent coverage algorithms in
  Section~\ref{sec:multi_agent_coverage_control} as the discretization
  of the iterative descent scheme from
  Section~\ref{sec:transport_models}, establish convergence results
  and study their behavior in the continuous-time and $N \rightarrow
  \infty$ limits. Section~\ref{sec:case_study_coverage_control}
  contains a case study of the well-known Lloyd's algorithm within the
  theoretical framework developed in the prior sections and results
  from numerical experiments. 

  \textbf{Notation.}  We let~$\|\cdot \| : \real^d \rightarrow
  \realnonnegative$ denote the Euclidean norm on~$\real^d$ and~$|
  \cdot | : \real \rightarrow \realnonnegative$ the absolute value
  function.  The gradient operator in $\real^d$ is represented as
  $\nabla = \left(
    \partial/\partial x_1, \ldots \partial/\partial x_n \right)$,
  where, as a shorthand, we let $\partial/\partial z \equiv \partial_z
  $ be the partial derivative w.r.t.~a variable~$z$ and
  $\frac{\partial}{\partial x_i} \equiv \partial_i$.
  Consider a set $\Omega \subseteq \real^d$. In what
  follows,~$\partial \Omega \subseteq \real^d$ denotes its boundary,
  $\bar{\Omega} = \Omega \cup \partial \Omega$ its closure, and
  $\mathring{\Omega} = \Omega \setminus \partial \Omega$ its interior
  with respect to the standard Euclidean topology.  For~$M \subseteq
  \Omega$, we define the distance~$d(x,M)$ of a point~$x \in \Omega$
  to~$M$ as~$d(x,M) = \inf_{y \in M} \| x - y \|$.  Given any~$x \in
  \Omega \subset \real^d$, the set~$B_r(x)$ is the closed $d$-ball of
  radius~$r>0$, centered at~$x$.
  The  identity map is denoted by~$\id$ and the identity matrix of
  dimension~$d$ by~$I_d$.
  The indicator function on~$\Omega$ for the subset~$M$ is denoted
  as~$\mathsf{1}_{M}:\Omega \rightarrow \{0,1\}$.
  We use~$\left \langle f, g \right \rangle$ to represent the inner
  product of functions~$f,g : \Omega \rightarrow \real$ w.r.t.~the
  Lebesgue measure~$\vol$, given by~$\left \langle f,g \right \rangle =
  \int_{\Omega} fg \dvol$. The set $\Lip(\Omega)$ is the space of
  Lipschitz continuous functions on $\Omega$.
  A function $p: \Omega \rightarrow \real$ is called $l$-smooth
(or Lipschitz differentiable) if for any~$x, y \in \Omega$, we
have~$\| \nabla p(y) - \nabla p(x) \| \leq l \| y-x \|$.
  It can be shown that for an $l$-smooth function 
  $p: \Omega \rightarrow \real$ and
  any~$x,y \in \Omega$, we have $| p(y) - p(x) - \langle \nabla
  p(x), y-x \rangle | \leq \frac{l}{2} \| y-x \|^2$.
  We denote by $\mathcal{P}(\Omega)$ the space of probability measures
  over $\Omega$ and by $\mathcal{P}^{\rm a}(\Omega)$ the space of 
  atomless probability measures over $\Omega$. For a measurable mapping~$\mathcal{T} : \Omega
  \rightarrow \Theta$, where $\Omega$ and $\Theta$ are measurable, we
  denote by $\mathcal{T}_{\#} \mu \in \mathcal{P}(\Theta)$ the
  pushforward measure of~$\mu \in \mathcal{P}(\Omega)$ and we have
  $\mathcal{T}_{\#} \mu (B) = \mu (\mathcal{T}^{-1}(B))$,  for all
  measurable $B \subseteq \Theta$.

\section{Coverage optimization problem} \label{sec:coverage_optimization} 
In this section, we
formulate the multi-agent coverage problem as an optimization of a
macroscopic coverage objective, which forms the focus of our analysis
and algorithm design in the subsequent sections.  We begin by
specifying the problem setting.
Let $\Omega \subset \real^d$ be compact and convex,  
and $\mathbf{x} = \left(x_1, \ldots, x_N \right)$ (with $x_i \in
\Omega$ for $i \in \mathcal{I} = \lbrace 1, \ldots, N \rbrace$ being
the agent positions) denote the microscopic state of the multi-agent
system.  In specifying the macroscopic configuration, we look for a
representation that satisfies two key properties, (i)
\textit{Permutation-invariance:} Assuming that the agents are
identical, i.e., the multi-agent system is homogeneous, 
we note that every microscopic configuration $\mathbf{x}
\in \Omega^N$ is equivalent to $\left( P \otimes I_{d} \right)
\mathbf{x}$ for any permutation $P \in \real^{N\times N}$.  The
representation must be invariant under such permutations, and (ii)
\textit{Consistency in the $N \rightarrow \infty$ limit}: The space of
representations must contain the ``representation limit" as $N
\rightarrow \infty$, to enable the study of large-scale properties of
coverage algorithms.  This leads us to specifying the macroscopic
configuration of the multi-agent system by probability measures over
the underlying space~$\Omega$.  For the microscopic configuration
$\mathbf{x} = \left(x_1, \ldots, x_N \right)$, we specify the
corresponding macroscopic configuration by the probability measure
$\widehat{\mu}^N_{\mathbf{x}} = \frac{1}{N} \sum_{i=1}^N
\delta_{x_i}$.  We note that $\widehat{\mu}^N_{\mathbf{x}}$ is
invariant under permutations of agent positions.  Furthermore, if the
positions~$x_i$ are independently and identically distributed
according to an (absolutely continuous) probability measure $\mu \in
\mathcal{P}(\Omega)$, it follows from the Glivenko-Cantelli
theorem~\cite{PB:13} that as $N \rightarrow \infty$, the discrete
probability measure $\widehat{\mu}^N_{\mathbf{x}}$ converges
uniformly, and almost surely, to~$\mu$.  In this way, probability
measures over~$\Omega$ are a suitable space of macroscopic
representations that combine the desired properties of
permutation-invariance and consistency in the $N \rightarrow \infty$
limit.

With the microscopic and macroscopic representations of the
multi-agent system in place, as illustrated in
Figure~\ref{fig:multiscale}, we now move to the specification of the
coverage task as the minimization of a macroscopic coverage objective
function $F:\mathcal{P}(\Omega) \rightarrow \real$. We let~$F$ be
$l$-smooth and strictly geodesically convex (the notion is introduced in
Section~\ref{sec:preliminaries}), 
  with a unique minimizer $\mu^* \in
\mathcal{P}(\Omega)$.  The coverage problem can then be described as
follows: Given an initial macroscopic configuration $\mu_0 \in
\mathcal{P}(\Omega)$ of the multi-agent system (with $\mu_0$ being an
absolutely continuous probability measure), specify a descent scheme
in $\mathcal{P}(\Omega)$ that minimizes the coverage objective
function~$F$, generating a sequence $\{\mu_k\}_{k \in \mathbb{N}}$
that converges weakly to $\mu^*$ as $k \rightarrow \infty$.  In
Section~\ref{sec:transport_models}, we propose a proximal descent
scheme that exploits the (generalized) convexity of $F$ to solve the
coverage task.  Furthermore, in
Section~\ref{sec:multi_agent_coverage_control} we obtain an
implementable multi-agent coverage algorithm that updates agent
positions in $\Omega$ and performs consistently (in the $N \rightarrow
\infty$ limit) with the macroscopic descent scheme.  That is, we
design a provably-correct, discrete-time, agent-based algorithm that
generates microscopic sequences $\{\mathbf{x}_k\}_{k \in \mathbb{N}}
\subseteq \Omega^N$ such that $\lim_{k,N \rightarrow \infty}
\widehat{\mu}_{\mathbf{x}_k}^N = \mu^\star$.  We address this question
in Section~\ref{sec:multi_agent_coverage_control} by tying the
macroscopic descent scheme with the microscopic coverage algorithm by
means of a variational approach.
 
\textit{Example coverage objective functions.}  We introduce a class of
coverage objective functions, whose convexity properties will be
analyzed in Section~\ref{sec:case_study_coverage_control}. 
Furthermore, in Section~\ref{sec:case_study_coverage_control} we also
establish a relationship between the macroscopic descent scheme
corresponding to these objective functions and the well-known Lloyd's
algorithm~\cite{JC-SM-TK-FB:02-tra}.
Let $f : \real \rightarrow \real$ be a strictly convex,
non-decreasing and $l$-smooth function with $f(0) = 0$, and let:
\begin{align}
  C_f(\mu, \nu) = \inf_{\substack{T: \Omega \rightarrow \Omega \\
      T_{\#}\mu = \nu }} \int_{\Omega} f(|x - T(x)|)~d\mu(x),
	\label{eq:OT_cost_f}
\end{align}
be defined for two probability measures~$\mu$ and~$\nu$.  In the
quadratic case $f(x) = x^2$, we get $C_f \equiv W_2^2$, the so-called
$L^2$-Wasserstein distance, which is a metric over
$\mathcal{P}(\Omega)$.  Conversely, this suggests the design of a
coverage objective function given a target macroscopic configuration
$\mu^\star$, as $F(\mu) = W_2^2(\mu,\mu^\star)$, which quantifies how
far $\mu$ is from the target~$\mu^\star$.

\begin{figure}[!h]
        \includegraphics[width=0.48\textwidth]{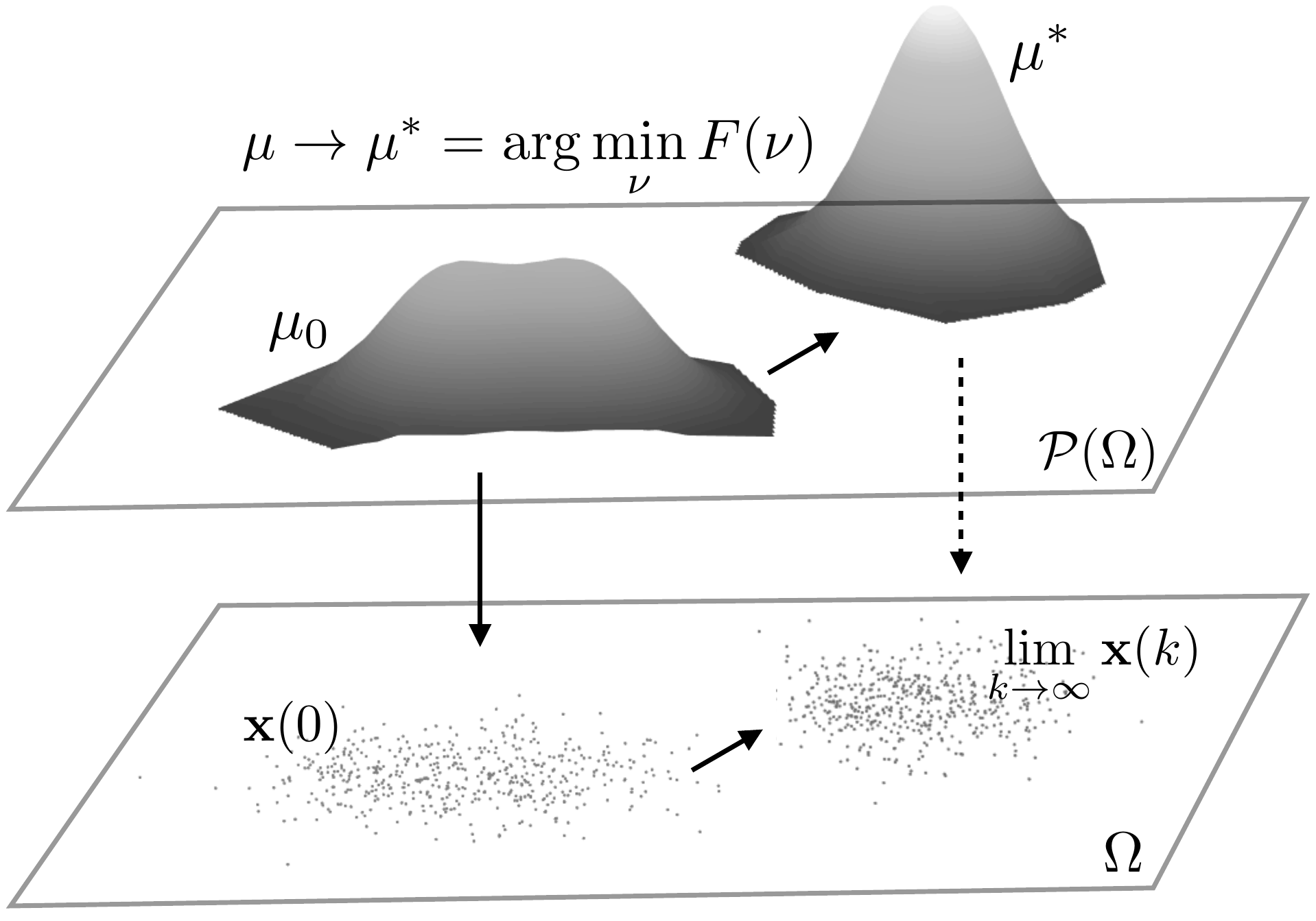}
        \caption{The figure illustrates scale separation in coverage
        control, with the agent positions in the domain~$\Omega$
        and the corresponding probability distributions over $\Omega$ 
        constituting the microscopic and macroscopic
        representations, respectively.}
		\label{fig:multiscale}
\end{figure}

\section{Mathematical preliminaries}
\label{sec:preliminaries}
Here, we summarize mathematical notions that are needed as
a background to the main results of the paper. For more
details, we refer the reader to~\cite{PB:13,FS:15} for more
information as well as to the Appendix.

\subsection{The Wasserstein space of probability measures}

Let $\Omega$ be a compact subset of $\real^d$, let
$\mathcal{P}(\Omega)$ denote the space of probability measures over
$\Omega$ and $\supscr{\mathcal{P}}{a}(\Omega)$ the set of
atomless\footnote{A measure $\mu$ is atomless if
    there are no atoms for the measure. An atom is a measurable set $A
    \subseteq \Omega$, with $\mu(A) >0$ and s.t.~for any $B \subseteq
    A$ with $\mu(B) <\mu(A)$ implies $\mu(B) = 0$. In particular, any
    absolutely continuous measure (wrt the Lebesgue measure) is
    atomless.}  probability measures over $\Omega$. The
$L^2$-Wasserstein distance between $\mu, \nu \in \mathcal{P}(\Omega)$
is given by:
\begin{align}
  W_2^2(\mu, \nu) = \min_{\pi \in \Pi(\mu, \nu)} \int_{\Omega \times
    \Omega} |x-y|^2 ~d\pi(x,y),
	\label{eq:Kantorovich_Formulation_W2}
\end{align}
where $\Pi(\mu, \nu)$ is the space of joint probability measures over
$\Omega \times \Omega$ with marginals $\mu$ and $\nu$. The definition
of $L^2$-Wasserstein distance in~\eqref{eq:Kantorovich_Formulation_W2}
follows from the so-called Kantorovich formulation of optimal
transport. An alternative formulation of this problem, called the
Monge formulation of optimal transport, is given below:
\begin{align}
  W_2^2(\mu, \nu) = \min_{\substack{T: \Omega \rightarrow \Omega \\
      T_{\#}\mu = \nu }} \int_{\Omega} |x - T(x)|^2 ~d\mu(x).
  \label{eq:Monge_Formulation_W2}
\end{align}
In the Monge formulation~\eqref{eq:Monge_Formulation_W2}, the
minimization is carried out over the space of maps $T: \Omega
\rightarrow \Omega$ for which the probability measure $\nu$ is
obtained as the pushforward of $\mu$; i.e.~$\mathcal{T}_{\#} \mu =
\nu$.  This can be viewed as a deterministic formulation of optimal
transport, where the transport is carried out by a map, whereas the
Kantorovich formulation \eqref{eq:Kantorovich_Formulation_W2} can be
seen as a problem relaxation, where the transport plan is described by
a joint probability measure $\pi$ over $\Omega \times \Omega$, with
$\mu$ and $\nu$ as its marginals.  It is to be noted that the Monge
formulation does not always admit a solution, while the Kantorovich
problem does. Roughly speaking, the Kantorovich formulation is the
``minimal'' extension of the Monge formulation, as both problems
attain the same infimum. The result~\cite[Theorem 1.17]{FS:15}
guarantees the existence and uniqueness of minimizers for both
problems, when $\mu$ is atomless.

It holds that~\cite{PB:13} that $(\mathcal{P}({\Omega}), W_2)$ is a
complete metric space. 
In addition, $W_2$ metrizes the convergence wrt the so-called weak
topology (cf.~\cite[Theorem~6.9]{CV:08}). Due to this, we will
indistinctively refer to convergence of measures in the weak topology
as convergence wrt $W_2$.  Finally, it can be shown that
$\mathcal{P}({\Omega})$ is compact wrt the weak topology;
see Appendix.

\subsection{Regularity of functionals on the Wasserstein space}
Results in convex analysis can be appropriately generalized to
functionals on the $L^2$-Wasserstein space $(\mathcal{P}(\Omega),
W_2)$, see~\cite{LA-NG-GS:08} for a detailed treatment. Before we can
define any notion of convexity on $(\mathcal{P}(\Omega), W_2)$, we
need to introduce an appropriate notion of interpolation, which plays
a similar role to the notion of segments in Euclidean space.

\begin{definition}[\bf \emph{Generalized displacement interpolation}]
\label{defn:generalized_displacement_interpolation}
Let $\Omega$ be a compact subset of $\real^d$, $\mu, \nu \in
\mathcal{P}(\Omega)$, and $\theta \in
\supscr{\mathcal{P}}{a}(\Omega)$ be an atomless probability
measure. Let $T_{\theta \rightarrow \mu}: \Omega \rightarrow \Omega$
and $T_{\theta \rightarrow \nu}: \Omega \rightarrow \Omega$ be the
optimal transport maps from $\theta$ to $\mu$, and $\theta$ to $\nu$
resp.~in the $L^2$-Wasserstein space over $\Omega$. A
\emph{(generalized) displacement interpolant} of $\mu$ and $\nu$
\emph{w.r.t. $\theta$} is given by $\gamma_t = \left( (1-t) T_{\theta
    \rightarrow \mu} + t T_{\theta \rightarrow \nu} \right)_{\#}
\theta$, for $t \in [0,1]$.
\end{definition}
We refer the reader to~\cite[Chapter~9.2]{LA-NG-GS:08} for a
detailed discussion of the above notion and its
motivation. 
It can be shown that, for a compact and convex
$\Omega \subset \real^d$, the space $\mathcal{P}(\Omega)$ is
geodesically convex w.r.t. the notion of generalized displacement
interpolation for any absolutely continuous reference measure~$\theta
\in \supscr{\mathcal{P}}{a}(\Omega)$. 
Now, we introduce the following standard definition on the
(generalized) geodesic convexity of functionals on the
$L^2$-Wasserstein space~$(\mathcal{P}(\Omega), W_2)$.  We note that
this is a particular case of a more general definition of convexity
in~\cite[Chapter~9]{LA-NG-GS:08} (Definition 9.2.4 for $\lambda = 0$.) 
\begin{definition}[\bf \emph{(Generalized) Geodesic convexity}]
\label{defn:generalized_geodesic_convexity}
Let $\Omega \subseteq \real^d$ be a compact and convex set, and let
$\mu, \nu, \theta \in \supscr{\mathcal{P}}{a}(\Omega)$ be atomless probability
measures, for which there exist unique $T_{\theta \rightarrow \mu}:
\Omega \rightarrow \Omega$ and $T_{\theta \rightarrow \nu}: \Omega
\rightarrow \Omega$ optimal transport maps from $\theta$ to $\mu$ and
from $\theta$ to $\nu$ respectively, in $(\mathcal{P}(\Omega), W_2)$.
A functional $F: \mathcal{P}(\Omega) \rightarrow \real$ is
(generalized) geodesically convex w.r.t.~$\theta \in \mathcal{P}(\Omega)$
(resp.~(generalized) strictly geodesically convex w.r.t.~$\theta$) 
  if the following holds for every $t \in [0,1]$:
\begin{align*}
  F(\left((1-t) T_{\theta \rightarrow \mu} + t T_{\theta \rightarrow
      \nu}\right)_{\#} \theta) \leq (1-t) F(\mu) + t F(\nu).
         \end{align*}
         (resp.~the previous inequality holds with strict inequality).
 Furthermore, a functional $F$ is geodesically convex
 (resp.~strictly geodesically convex) 
 if it is generalized geodesically convex 
(resp.strictly generalized geodesically convex) 
 with respect to any atomless probability measure
 $\theta \in \supscr{\mathcal{P}}{a}(\Omega)$. 
\end{definition}

As in the finite-dimensional case, it can be shown that a geodesically
convex functional is also continuous, thus, it admits a
maximizer/minimizer over the compact $(\mathcal{P}(\Omega),W_2)$.  It
is possible to obtain a first-order characterization of convexity for
Fr\'echet differentiable functionals on atomless measures. The notion
of Fr\'echet differential in Wasserstein space~\cite{LA-NG-GS:08}
is an extension of the standard notion for
functions defined on $\real^d$. For the sake of completeness, we provide
a statement of the first-order convexity characterization: 
\begin{lemma}[\bf \emph{First-order convexity condition}] 
  \label{lemma:first_order_convexity} Let $\Omega
  \subseteq \real^d$ be compact and convex, $\mu, \nu, \theta \in
  \supscr{\mathcal{P}}{a}(\Omega)$ be atomless measures.  Let $F:
  \mathcal{P}(\Omega) \rightarrow \real$ be a Fr{\'e}chet
  differentiable and (generalized) geodesically convex functional. 
  Then, we have:
\begin{align}
  F(\nu) \geq F(\mu) + \int_{\Omega} \left \langle \xi_{\mu} (T_{\theta \rightarrow \mu}) ,
    T_{\theta \rightarrow \nu} - T_{\theta \rightarrow \mu} \right
  \rangle d\theta,
	\label{eq:first_order_convexity}
\end{align}
where~$\xi_{\mu}$ is the Fr{\'e}chet derivative of~$F$ at~$\mu$,
and $T_{\theta \rightarrow \mu}: \Omega \rightarrow \Omega$ and $T_{\theta
  \rightarrow \nu}: \Omega \rightarrow \Omega$ are optimal transport maps
from $\theta$ to $\mu$ and from $\theta$ to $\nu$ respectively.
\end{lemma}
Finally, we define the notion of strong geodesic convexity of
Fr\'echet-differentiable functionals on~$(\mathcal{P}(\Omega), W_2)$:
  \begin{definition}[\bf \emph{Strong geodesic convexity of a
      functional on~$(\mathcal{P}(\Omega), W_2)$}]
\label{defn:strong_geodesic_convexity}
Let $\Omega \subseteq \real^d$ be compact and convex and $\mu, \nu,
\theta \in \supscr{\mathcal{P}}{a}(\Omega)$ be atomless probability
measures.  Let $F:\mathcal{P}(\Omega) \rightarrow \real$ be
Frech\'et-differentiable. Let~$\xi_{\mu}$ and~$\xi_{\nu}$
be the Fr{\'e}chet derivatives of $F$ evaluated at measures $\mu$
and~$\nu$, respectively. Then,~$F$ is \emph{strongly (geodesically)
  convex} w.r.t.~$\theta$ as reference measure 
  if there exists an~$m > 0$ such that:
\begin{align*}
 	F(\nu) &\geq  F(\mu) + \int_{\Omega} \left \langle \xi_{\mu} (T_{\theta \rightarrow \mu}) , T_{\theta \rightarrow \nu} - T_{\theta \rightarrow \mu} \right \rangle d\theta  \\
 											&\qquad + 	\frac{m}{2} \int_{\Omega} \left| T_{\theta \rightarrow \nu} - T_{\theta \rightarrow \mu} \right|^2 d\theta,
\end{align*}
where $T_{\theta \rightarrow \mu}$, $T_{\theta \rightarrow \nu}$
are the optimal transport maps from~$\theta$ to~$\mu$,~$\nu$
respectively.
\end{definition}
It can be shown that strong geodesic convexity
of functionals implies strict geodesic convexity and, thus, geodesic
convexity of functionals over $\supscr{\mathcal{P}}{a}(\Omega)$. It
can also be proven that a Fr{\'e}chet differentiable, strongly geodesically 
convex function admits a unique minimizer over
$\mathcal{P}(\Omega)$.  

Furthermore, we introduce the notion of $l$-smoothness that will be useful for
the development of gradient descent-based transport schemes  in the paper.  
\begin{definition}[\bf \emph{$l$-smoothness of functionals on $(\mathcal{P}(\Omega),W_2)$}]
\label{defn:l_smooth_functional}
  A functional $F : \mathcal{P}(\Omega) \rightarrow \real$ is called
  \emph{$l$-smooth} w.r.t. a base measure~$\theta \in \mathcal{P}^{\rm a}(\Omega)$ 
  if for any~$\mu, \nu \in \mathcal{P}^{\rm a}(\Omega)$, we have:
\begin{align*}
 	&\left| F(\nu) - F(\mu) - \int_{\Omega} \left \langle \xi_{\mu} (T_{\theta \rightarrow \mu}) , T_{\theta \rightarrow \nu} - T_{\theta \rightarrow \mu} \right \rangle d\theta \right|  \\
 	&\leq 	\frac{l}{2} \int_{\Omega} \left| T_{\theta \rightarrow \nu} - T_{\theta \rightarrow \mu} \right|^2 d\theta,
\end{align*}
where $\xi_{\mu}$ is the Fr{\'e}chet derivative of~$F$ 
evaluated at~$\mu$, and
$T_{\theta \rightarrow \mu}$, $T_{\theta \rightarrow \nu}$
are the optimal transport maps from~$\theta$ to~$\mu$,~$\nu$
respectively.
\end{definition}

\section{Macroscopic and particle descent schemes}
\label{sec:transport_models}
In this section, we present a (macroscopic) iterative descent scheme
in the space of probability measures $\mathcal{P}(\Omega)$ and
establish weak convergence to the minimizer under certain conditions.
Furthermore, we derive an equivalent (microscopic) characterization of
the descent scheme in $\Omega$.
More specifically, Theorem~\ref{thm:conv_prox_recursion_wasserstein} first establishes 
convergence of the proximal descent scheme~\eqref{eq:prox_recursion_wasserstein}
in~$\mathcal{P}(\Omega)$ under appropriate conditions.
In Theorem~\ref{thm:target_dynamics_euclidean} we establish the 
microscopic scheme in~$\Omega$ corresponding to~\eqref{eq:prox_recursion_wasserstein}.
We then note the hurdle to implementation of the above scheme
and circumvent it by an appropriate modification of the scheme
and establish its convergence in Theorem~\ref{thm:proximal_scheme_convergence_target}.
We refer to the Appendix for additional definitions and supporting results.

We consider the following proximal recursion in $\mathcal{P}(\Omega)$
starting from any absolutely continuous $\mu_0 \in
\mathcal{P}(\Omega)$:
\begin{align}
  \mu_{k+1} \in \arg \min_{\nu \in \mathcal{P}(\Omega)}
  \frac{1}{2\tau} W_2^2(\mu_k, \nu) + F(\nu).
  \label{eq:prox_recursion_wasserstein}
\end{align} In the remainder of this section, we will
  assume that $F$ is sufficiently regular, which ensures that the
  previous scheme is well defined and results in an atomless sequence
  of measures with negligible mass on $\partial
  \Omega$. \begin{assumption}[\bf \emph{Regularity conditions on
      $F$}] \label{ass:regularity}
   The functional~$F: \mathcal{P}(\Omega) \rightarrow \real$ 
    is $l$-smooth (in the sense of Definition~\ref{defn:l_smooth_functional} w.r.t. any base measure~$\theta \in \mathcal{P}^{\rm a}(\Omega)$)
    and strictly geodesically convex on $(\mathcal{P}(\Omega), W_2)$
    (in the sense of Definition~\ref{defn:generalized_geodesic_convexity} w.r.t. any base measure~$\theta \in \mathcal{P}^{\rm a}(\Omega)$).
    The Fr{\'e}chet derivative $\xi_\mu$ of~$F$
    is Lipschitz continuous on~$\Omega$ (with Lipschitz constant~$\lambda$)
    at any~$\mu \in \mathcal{P}^{\rm a}(\Omega)$.
    Furthermore, it satisfies the boundary condition 
    $\xi_{\mu} \cdot \mathbf{n} \geq 0$ on $\partial
    \Omega$ (where $\mathbf{n}$ is the outward normal to $\partial
    \Omega$) for any~$\mu \in \mathcal{P}^{\rm a}(\Omega)$.
  \end{assumption}

\begin{assumption}[\bf \emph{Atomless proximal descent  sequence}]
\label{ass:atomless_sequence_prox_wasserstein}
We assume that the sequence $\lbrace \mu_k \rbrace_{k \in \mathbb{N}}$
generated by \eqref{eq:prox_recursion_wasserstein} is such that $\mu_k
\in \supscr{\mathcal{P}}{a}(\Omega)$ for all $k \in
\mathbb{N}$. 
\end{assumption}
We remark here that sufficient regularity of the functional~$F$ 
and the atomlessness of~$\mu_0$ should guarantee validity of
Assumption~\ref{ass:atomless_sequence_prox_wasserstein}. 
Since we do not offer a characterization of the regularity
of~$F$ to this end, we retain Assumption~\ref{ass:atomless_sequence_prox_wasserstein} 
in establishing the following theorem:
\begin{theorem}[\bf \emph{Convergence of proximal recursion~\eqref{eq:prox_recursion_wasserstein}}]
  Let $\Omega \subseteq \real^d$ be a compact, convex set, and let $F:
  \mathcal{P}(\Omega) \rightarrow \real$ satisfy the
    regularity conditions of Assumption~\ref{ass:regularity}. 
  Let $\mu_0$ be an absolutely continuous measure. Under
  Assumption~\ref{ass:atomless_sequence_prox_wasserstein} on the
  generation of a proximal descent atomless sequence, and for $0 <
  \tau < 1/l$, the sequence $\lbrace \mu_k \rbrace_{k \in \mathbb{N}}$, 
  generated by the proximal recursion~\eqref{eq:prox_recursion_wasserstein}, 
  converges weakly to~$\mu^\star = \arg \min_{\nu \in \mathcal{P}(\Omega)}F(\nu)$ 
  as $k \rightarrow \infty$.
\label{thm:conv_prox_recursion_wasserstein}
\end{theorem}
\begin{proof}
  It follows that:
\begin{align*}
  &\frac{1}{2\tau} W_2^2(\mu_k, \mu_{k+1}) + F(\mu_{k+1}) \leq F(\mu_k) \\
  &\Longleftrightarrow F(\mu_{k+1}) \leq F(\mu_k) - \frac{1}{2\tau}
  W_2^2(\mu_k, \mu_{k+1}).
\end{align*}
This implies that for~$\mu_{k} \neq \mu_{k+1}$, we have~$F(\mu_{k+1})
< F(\mu_k)$ and the sequence $\lbrace F(\mu_k) \rbrace_{k\in
  \mathbb{N}}$ is monotonically strictly decreasing. In addition,
$\lbrace \mu_k \rbrace_{k \in \mathbb{N}}$ is contained in the
sublevel set $\mathcal{S}(\mu_0)$ of $F(\mu_0)$.  

From Lemma~\ref{lemma:compactness_sublevel_F} 
in the Appendix, $\mathcal{S}(\mu_0)$ is geodesically
convex and compact in the $L^2$-Wasserstein space
$(\mathcal{P}(\Omega), W_2)$. Thus, there is a weakly convergent
subsequence $\{\mu_{k_\ell}\}\rightarrow_\ell \overline{\mu} \in
\mathcal{S}(\mu_0)$.  Consider the functional $G_{\mu}$
from~\eqref{eq:prox_recursion_wasserstein}, for $\mu \in
\mathcal{P}(\Omega)$, such that $G_{\mu}(\nu) = \frac{1}{2\tau}
W_2^2(\mu,\nu) + F(\nu)$.  First, note that
\begin{align*}
  &| G_{\mu_{k_\ell}}(\nu) - G_{\overline{\mu}}(\nu)| =
  \frac{1}{2\tau}| W_2^2(\overline{\mu},\nu) -
  W_2^2(\mu_{k_\ell},\nu)|
  \nonumber \\
  & = \frac{1}{2\tau} (W_2(\overline{\mu},\nu) +
  W_2(\mu_{k_\ell},\nu)) |W_2(\overline{\mu},\nu) -
  W_2(\mu_{k_\ell},\nu)|, 
\end{align*}
for all $\ell$. Due to the triangular inequality, for all $\nu$,
$|W_2(\overline{\mu},\nu) - W_2(\mu_{k_\ell},\nu)| \le
W_2(\mu_{k_\ell},\overline{\mu}) $. Therefore,

{\small
\begin{align*}
  & | G_{\mu_{k_\ell}}(\nu) - G_{\overline{\mu}}(\nu)| \le
  \frac{1}{2\tau} (W_2(\overline{\mu},\nu) + W_2(\mu_{k_\ell},\nu))
  W_2(\overline{\mu},\mu_{k_\ell}).
\end{align*}
}
In addition, $\mathcal{S}(\mu_0)$ is a compact set
and $W_2$ is a continuous functional, then there is a constant $M$ such that
$W_2(\overline{\mu},\nu) + W_2(\mu_{k_\ell},\nu) \leq M$ and we have:

{\small
\begin{align*}
  & | G_{\mu_{k_\ell}}(\nu) - G_{\overline{\mu}}(\nu)| \le \frac{M}{2\tau}
  W_2(\overline{\mu},\mu_{k_\ell}),
\end{align*}
}for all $\nu$. Since $\mu_{k_\ell} \rightarrow_\ell \overline{\mu}$,
this implies the uniform convergence of the functionals
$G_{\mu_{k_\ell}}(\nu)$ to $G_{\overline{\mu}}(\nu)$.  In particular,
this implies that for all $\epsilon >0$, there is an $\ell_0$ such
that for all $\ell \ge \ell_0$, we have 
\begin{align*}
  | G_{\mu_{k_\ell}}(\nu) - G_{\overline{\mu}}(\nu)|< \epsilon,
\end{align*}
for all $\nu$. Let $\overline{\mu}^+ =\arg \min_\nu
G_{\overline{\mu}}(\nu)$, and recall that $\mu_{k_\ell +1} = \arg
\min_\nu G_{\mu_{k_\ell}}(\nu)$. Then, by  the $\min$ properties:
\begin{align*}
  & G_{\mu_{k_\ell}}(\mu_{k_\ell +1}) \le G_{\mu_{k_\ell}}(\nu) <
  G_{\overline{\mu}}(\nu) + \epsilon\\
  & \quad \qquad \Longrightarrow \, G_{\mu_{k_\ell}}(\mu_{k_\ell +1}) \le
  G_{\overline{\mu}}(\overline{\mu}^+) + \epsilon,\\
  & G_{\overline{\mu}}(\overline{\mu}^+) - \epsilon <
  G_{\overline{\mu}}(\nu) - \epsilon \le G_{\mu_{k_\ell}}(\nu)\\
   & \quad \qquad \Longrightarrow \,
   G_{\overline{\mu}}(\overline{\mu}^+) - \epsilon  \le G_{\mu_{k_\ell}}(\mu_{k_\ell +1}).
\end{align*}
That is, we have $|G_{\mu_{k_\ell}}(\mu_{k_\ell +1})
-G_{\overline{\mu}}(\overline{\mu}^+) |\le \epsilon$ for all $\ell \ge
\ell_0$.  The fact that  $\overline{\mu}$ is a fixed point for
$G_{\overline{\mu}}(\nu)$ now follows from the set of inequalities:
\begin{align*}
  G_{\overline{\mu}}(\overline{\mu}^+) \le 
  G_{\overline{\mu}}(\overline{\mu}) = F(\overline{\mu}) \le
  G_{\mu_{k_\ell}}(\mu_{k_{\ell} +1})< F(\mu_{k_\ell})
\end{align*}
The gap $G_{\mu_{k_\ell}}(\mu_{k_\ell +1})
-G_{\overline{\mu}}(\overline{\mu}^+)$ can be made arbitrarily small
by increasing $\ell$, so it must be that
$G_{\overline{\mu}}(\overline{\mu}) = F(\overline{\mu}) =
G_{\overline{\mu}}(\overline{\mu}^+)$, which implies $\overline{\mu}^+
= \overline{\mu}$ is the solution to the minimization problem of
$G_{\overline{\mu}}$ and satisfies $\nabla \left(\frac{\delta
    G}{\delta \nu} \right)_{\overline{\mu}} = 0$. The equation $\nabla
\left(\frac{\delta G}{\delta \nu} \right)_{\overline{\mu}} = 0$ is
equivalent to $\frac{1}{\tau}\nabla \phi_{\overline{\mu} \rightarrow
  \overline{\mu}} + \nabla \left(\frac{\delta F}{\delta \nu}
\right)_{\overline{\mu}} = 0$. Since $\nabla\phi_{\overline{\mu}
  \rightarrow \overline{\mu}} =0$, then $\overline{\mu}$ is a
minimizer of $F$, and from the strict geodesic convexity of
$F$ we get that the minimizer is unique and $\overline{\mu} = \mu^\star$.
Note that we can apply this reasoning to all the
accumulation points $\tilde{\mu}$ of the sequence $\{\mu_{k}\}$. 
Since all the convergent subsequences of $\{\mu_{k}\}$ have the 
same limit $\mu^\star$ and $\{\mu_{k}\}$ is contained in 
$\mathcal{S}(\mu_0)$ which is compact, 
we conclude that the whole sequence 
$\{\mu_{k}\}$ converges to $\mu^\star$ in $W_2$, i.e., weakly
as $k \rightarrow \infty$. \oprocend
\end{proof} 
\begin{remark}[\bf \emph{Squared-Wasserstein distance as objective functional}]
We now consider the case where the $L^2$-Wasserstein distance from the target measure~$\mu^*$ 
is chosen as the objective functional, i.e., $F(\nu) = \frac{1}{2} W_2^2(\nu, \mu^*)$.
We note that~$F$ is strictly (generalized) geodesically convex only w.r.t.~$\mu^*$
as the reference measure. This violates the regularity assumption~\ref{ass:regularity} 
(where we let~$F$ to be strictly (generalized) geodesically convex w.r.t. any 
(atomless) reference measure) and presentes a hurdle to
the application of Theorem~\ref{thm:conv_prox_recursion_wasserstein}. 
However, this hurdle can be mitigated as follows.
Let $G_{\mu_k}(\nu) = \frac{1}{2\tau} W_2^2(\mu_k, \nu) + F(\nu)$.  
The Fr{\'e}chet derivative of $G_{\mu_k}$ is given by $\nabla \left(
  \left. \frac{\delta G_{\mu_k}(\nu)}{\delta \nu} \right|_{\nu}
\right) = \frac{1}{\tau} \nabla \phi_{\nu \rightarrow \mu_k} + \nabla
\phi_{\nu \rightarrow \mu^*}$.  Moreover, at the critical point
$\mu_{k+1}$ of $G_{\mu_k}$ we have $\frac{1}{\tau} \nabla
\phi_{\mu_{k+1} \rightarrow \mu_k} + \nabla \phi_{\mu_{k+1}
  \rightarrow \mu^*} = \frac{1}{\tau} \left( \id - T_{\mu_{k+1}
    \rightarrow \mu_k} \right) +\left( \id - T_{\mu_{k+1} \rightarrow
    \mu^*} \right) = 0$, which implies that $\left( T_{\mu_{k+1}
    \rightarrow \mu_k} - \id \right) = \tau \left( \id - T_{\mu_{k+1}
    \rightarrow \mu^*} \right)$.  We then have $W_2(\mu_k, \mu_{k+1})
= \tau W_2(\mu_{k+1}, \mu^*)$.
For any (and only) $\nu$ on the geodesic between $\mu_k$ and $\mu^*$,
we have $W_2(\mu_k, \mu^*) = W_2(\mu_k, \nu) + W_2(\nu, \mu^*)$
(wherein the triangle inequality is an equality), and this is the case
if and only if $\int_{\Omega} \left \langle \id - T_{\nu \rightarrow
    \mu_k}, T_{\nu \rightarrow \mu^*} - \id \right \rangle d\nu = 
W_2(\mu_k, \nu) W_2(\nu, \mu^*)$. We see that this is
  indeed the case for $\nu = \mu_{k+1}$,  from
which we infer that $\mu_{k+1}$ lies on the geodesic between $\mu_k$
and $\mu^*$.  We therefore get that $\lbrace \mu_k \rbrace_{k \in
  \mathbb{N}}$ lies on the geodesic connecting $\mu_0$ and $\mu^*$.
Consequently, need only be concerned with the geodesic convexity of~$F$
along the geodesic connecting~$\mu_0$ and~$\mu^*$.
Now, from Proposition~\ref{prop:strictly_convex_C_f} in Appendix~\ref{app:agg_obj_func} 
it follows that $W_2^2(\cdot, \mu_k)$ is generalized geodesically convex
with reference measure $\mu_k$, and similarly $W_2^2(\cdot, \mu^*)$ is
generalized geodesically convex with reference measure $\mu^*$,
and the two measures $\mu_k$ and $\mu^*$ are interchangeable as reference
measures along the geodesic between them.  It then follows that $\mu_k$
can be chosen as the reference measure and the arguments in the
proof of Theorem~\ref{thm:conv_prox_recursion_wasserstein} apply.
\end{remark}
The implementation of~\eqref{eq:prox_recursion_wasserstein} can be
challenging because involves the solution of an infinite-dimensional
optimization problem. To address this, we determine the stochastic process in
$\Omega$ that equivalently describes the 
recursion~\eqref{eq:prox_recursion_wasserstein}.
More precisely, consider a proximal recursion in~$\Omega$ from an
initial condition $x_0 \in \Omega$:
\begin{align}
  x_{k+1} \in \arg \min_{z \in \Omega} \frac{1}{2\tau} | x_k - z |^2 +
   f_k(z),
  \label{eq:prox_recursion_euclidean}
\end{align}
where~$\lbrace f_k \rbrace_{k \in \mathbb{N}}$ is a sequence of
functions on~$\Omega$. Suppose that the initial condition $x_0$ is in
fact a random variable distributed according to $\mu_0$ (denoted $x_0 \sim \mu_0$). 
We are interested in defining the process
in~$\Omega$, through an appropriate choice of $\lbrace f_k \rbrace_{k
  \in \mathbb{N}}$, which results in a consistent transport of the
initial measure~$\mu_0$ according to the
recursion~\eqref{eq:prox_recursion_wasserstein}.
\begin{theorem}[\bf \emph{Target dynamics in $\Omega$}]
  \label{thm:target_dynamics_euclidean} Let $\Omega \subseteq \real^d$
  be a compact, convex set, and let $F:\Omega \longrightarrow \real$
  satisfy the  regularity conditions of
    Assumption~\ref{ass:regularity}.
  Under Assumption~\ref{ass:atomless_sequence_prox_wasserstein}, the
  proximal recursion~\eqref{eq:prox_recursion_wasserstein}, for $0<
  \tau < 1/l$, starting from $\mu_0 \in \mathcal{P}^a(\Omega)$ is
  obtained as the transport of $\mu_0$
  by~\eqref{eq:prox_recursion_euclidean} with $x_0 \sim \mu_0$ and
  $f_k = \left. \frac{\delta F}{\delta \nu}\right|_{\mu_{k+1}}$, for
  all $k \in \mathbb{N}$.
\end{theorem}
\begin{proof}
  We rewrite the single-step update
  in~\eqref{eq:prox_recursion_wasserstein} from an absolutely continuous 
  probability measure $\mu \in \mathcal{P}(\Omega)$ as follows:
\begin{align}
  \mu^{+} = \arg \min_{\nu \in \mathcal{P}(\Omega)} \frac{1}{2\tau}
  W_2^2(\mu, \nu) + F(\nu).
	\label{eq:PGD_wasserstein}
\end{align}
From Lemma~\ref{lemma:strongly_convex_PGD} the minimizer~$\mu^{+}$ in
\eqref{eq:PGD_wasserstein} is unique. 
Let $\{\mathbf{v}_\epsilon \}$ be a smooth
  one-parameter family of Lipschitz continuous vector fields such that
  $\mathbf{v}_0 = \mathbf{v}$, where $\mathbf{v}$ is
  any Lipschitz vector field on $\Omega$. Now, define a one-parameter family of
absolutely continuous probability measures $\lbrace \nu_{\epsilon}
\rbrace_{\epsilon \in \real}$ by means of $\partial_{\epsilon}
\nu_{\epsilon} + \nabla \cdot (\nu_{\epsilon} \mathbf{v}_{\epsilon}) =
0$, subject to $\mathbf{v}_{\epsilon} \cdot \mathbf{n} = 0$, and such
that $\nu_{0} = \mu^{+}$. Since~$\mu^{+}$ is a critical point of the
objective function in \eqref{eq:PGD_wasserstein}, from~\cite[Theorem~5.24]{FS:15}
we have:
\begin{align*}
 0 &= \left. \frac{d}{d\epsilon} \left( \frac{1}{2\tau} W_2^2(\mu,
      \nu_{\epsilon}) + F(\nu_{\epsilon}) \right) \right|_{\epsilon = 0}\\
    & = \frac{1}{\tau} \int_{\Omega} \left\langle \nabla
    		\phi_{\mu^{+} \rightarrow \mu}, \mathbf{v} \right\rangle d\mu^{+}
  			+ \int_{\Omega} \left\langle \xi , \mathbf{v} \right\rangle d\mu^{+} \\
  	&= \int_{\Omega} \left\langle \frac{1}{\tau} \nabla \phi_{\mu^{+}
      \rightarrow \mu} + \xi, \mathbf{v} \right\rangle d\mu^{+},
\end{align*}
where $\xi = \left. \nabla \left( \frac{\delta F}{\delta \nu}\right)
\right|_{\nu = \mu^{+}}$ and $\nabla \phi_{\mu^{+} \rightarrow \mu} =
\id - T_{\mu^{+} \rightarrow \mu}$, with $T_{\mu^{+} \rightarrow \mu}
: \Omega \rightarrow \Omega$ being the optimal transport map from
$\mu^{+}$ to $\mu$.  Since $\int_{\Omega} \left\langle \frac{1}{\tau}
  \nabla \phi_{\mu^{+} \rightarrow \mu} + \xi, \mathbf{v}
\right\rangle d\mu^{+} = 0$  for all $\mathbf{v}$,
 it implies that $\frac{1}{\tau} \nabla \phi_{\mu^{+}
  \rightarrow \mu} + \xi = 0$ ($\mu^{+}$ a.e. in $\Omega$), and we obtain:
\begin{align*}
  \frac{1}{\tau} \nabla \phi_{\mu^{+} \rightarrow \mu} + \xi =
  \frac{1}{\tau} \left( \id - T_{\mu^{+} \rightarrow \mu} \right) + \xi
  = 0,
\end{align*}
which implies that:
\begin{align}
  T_{\mu^{+}\rightarrow \mu} = \id + \tau \xi.
  \label{eq:transport_map}
\end{align}
Let $\varphi = \left. \left(\frac{\delta F}{\delta
      \nu}\right)\right|_{\nu = \mu^{+}}$. For any $y \in \Omega$ and
$\tau < 1/l$, consider: 
\begin{align}
  y^{+} = \arg \min_{z \in \Omega} \underbrace{\frac{1}{2\tau} |y-z|^2 + \varphi(z)}_{\triangleq g_y(z)}.
	\label{eq:PGD_euclidean}
\end{align}
The uniqueness of the minimizer above follows from 
the strong convexity of~$g_y$ for $\tau < 1/l$ (this
can be verified by following a similar procedure
as in the proof of Lemma~\ref{lemma:strongly_convex_PGD}, 
but now in the Euclidean space).
 If $y^{+} \in \mathring{\Omega}$ is a critical point of $g_y$ 
 in~\eqref{eq:PGD_euclidean}, then it
satisfies $y^{+} = y - \tau \nabla \varphi(y^{+})$. Since $\xi =
\nabla \varphi$, we can equivalently write $y^{+} = \left( \id + \tau
  \xi \right)^{-1}(y)$. That is, when the image of $y \in \Omega$
under the $\arg \min$ map in~\eqref{eq:PGD_euclidean} is a critical
point in the interior of $\Omega$, then it is also the
inverse image of $y$ under the optimal transport map $T_{\mu^{+}
  \rightarrow \mu}$.

Now, for a $y \in \mathring{\Omega}$, 
the inner product of the gradient of $g_y$ at any point $z
\in \partial \Omega$ on the boundary of $\Omega$ with the outward
normal $\mathbf{n}$ to $\partial \Omega$ at $z$ is given by $\nabla
g_y \cdot \mathbf{n} = \left( \frac{1}{\tau} (z-y) + \nabla \varphi
  (z) \right) \cdot \mathbf{n} = \frac{1}{\tau} (z-y) \cdot \mathbf{n}
> 0$, since $\nabla \varphi \cdot \mathbf{n} = 0$ and $z-y$ points
outward to $\Omega$ (as $z \in \partial \Omega$ and $y \in
\mathring{\Omega}$ and $\Omega$ is convex).  This implies that there
exists a point $\tilde{z}$ in the interior of $\Omega$ in a
neighborhood of $z$ such that $g_y(\tilde{z}) < g_y(z)$, which implies
that $z$ cannot be the minimizer.  Thus, for any $y \in
\mathring{\Omega}$, the minimizer of $g_y(z) = \frac{1}{2\tau} |y-z|^2
+ \varphi(z)$ cannot lie on the boundary $\partial \Omega$, and must
therefore lie in the interior of $\Omega$ and be a critical point of
the objective function $g_y$.  Now, when $y \in \partial \Omega$, if
$y^{+} \notin \mathring{\Omega}$, it must be that $y^{+} = y$
(otherwise we obtain a contradiction for the same reason as above, the
inner product of $\nabla g_y$ with the outward normal would be
strictly positive) and the $\mathrm{argmin}$ map (and the optimal transport map)
coincides with the identity map in this case.

It then follows that for any $y \in \Omega$, its image $y^{+}$ under the 
$\mathrm{argmin}$ map is
exactly its inverse image under the optimal transport map $T_{\mu^{+}
  \rightarrow \mu}$.  That is, the map in \eqref{eq:PGD_euclidean} is
the inverse of the optimal transport map $T_{\mu^{+} \rightarrow
  \mu}$.  Thus, we have that the map $T_{\mu^{+}\rightarrow \mu} = \id
+ \tau \xi$ is well-defined and so is its inverse,
it holds that $\left( T_{\mu^{+}\rightarrow \mu} \right)^{-1}_{\#} \mu = \left(
  \id + \tau \xi \right)^{-1}_{\#} \mu = \mu^{+}$,
and~\eqref{eq:PGD_wasserstein} is the lift to the space of probability
measures of \eqref{eq:PGD_euclidean}.

We therefore conclude that the proximal recursion
\eqref{eq:prox_recursion_wasserstein} starting from $\mu_0$ is the
transport of $\mu_0$ by \eqref{eq:prox_recursion_euclidean} with $x_0
\sim \mu_0$. \oprocend
\end{proof}
From a computational perspective,
Theorem~\ref{thm:target_dynamics_euclidean} still requires the
evaluation of the first variation $\frac{\delta F}{\delta \nu}$ at
$\mu_{k+1}$, the transported measure at the future time
instant~$k+1$. To circumvent this problem, we can alternatively consider
the dynamics~\eqref{eq:prox_recursion_euclidean} with the choice
of~$\tilde{f}_k = \left. \frac{\delta F}{\delta
    \nu}\right|_{\mu_{k}}$, which only requires the evaluation, at
time instant~$k$, of the first variation $\frac{\delta F}{\delta \nu}$
at $\mu_k$.
Consider the $l$-smooth, geodesically-convex (linear)
$\widetilde{F}(\nu) = \mathbb{E}_{\nu} \left[ \left. \frac{\delta F}{\delta \mu}
  \right|_{\mu_k} \right]$, for $\nu \in \Omega$,
which satisfies $\frac{\delta \widetilde{F}}{\delta \nu} = \left. \frac{\delta
    F}{\delta \mu}\right|_{\mu_{k}}$. It follows from
Theorem~\ref{thm:target_dynamics_euclidean} that the descent in
$\mathcal{P}(\Omega)$ corresponding
to~\eqref{eq:prox_recursion_euclidean} with $\tilde{f}_k = \left. \frac{\delta
    \widetilde{F}}{\delta \nu} \right|_{\mu_k}$~is given by:
\begin{align}
  \mu_{k+1} \in \arg \min_{\nu \in \mathcal{P}(\Omega)} \frac{1}{2\tau}
  W_2^2(\mu_k, \nu) + \mathbb{E}_{\nu} \left[ \left. \frac{\delta
        F}{\delta \mu} \right|_{\mu_k} \right].
	\label{eq:equiv_descent_wasserstein}
\end{align}

The convergence of~\eqref{eq:equiv_descent_wasserstein} can also be
established as follows:
\begin{theorem}[\bf \emph{Convergence of
    recursion~\eqref{eq:equiv_descent_wasserstein}}] Let $F: \Omega
  \rightarrow \real$ satisfy the
    regularity conditions of Assumption~\ref{ass:regularity}
    and let $\alpha = l + \lambda$. 
  The sequence
  $\lbrace \mu_k \rbrace_{k \in \mathbb{N}}$ obtained as the transport
  of measure~$\mu_0 \in \mathcal{P}^a(\Omega)$
  by~\eqref{eq:equiv_descent_wasserstein} with $\tau < 1/\alpha$, $x_0 \sim
  \mu_0$ and the choice $\tilde{f}_k = \left. \frac{\delta
      \widetilde{F}}{\delta \nu}\right|_{\mu_{k}}$, converges weakly
  to $\mu^\star = \arg\min_{\nu \in \mathcal{P}(\Omega)} F(\nu)$ as $k
  \rightarrow \infty$.
\label{thm:proximal_scheme_convergence_target}
\end{theorem}
\begin{proof}
  From the $l$-smoothness of~$F$ and Lemma~\ref{lemma:l_smooth_functional} (with $\mu_{k+1}$
  as the reference measure), we get:
\begin{align*}
  &\int_{\Omega} \left \langle \xi_k(T_{\mu_{k+1}
      \rightarrow \mu_k}) - \xi_{k+1} , T_{\mu_{k+1}
      \rightarrow \mu_k} - \id \right \rangle d\mu_{k+1}  \\
      & \leq l W_2^2(\mu_k, \mu_{k+1}).
\end{align*}
Since $\mathbb{E}_{\nu} \left[ \left. \frac{\delta F}{\delta \mu} \right|_{\mu_k} \right]$ 
is linear in~$\nu$ for a given~$\mu_k$, and the Fr{\'e}chet derivative of~$F$
is Lipschitz continuous at any atomless probability measure (with Lipschitz constant~$\lambda$), 
it follows from Lemma~\ref{lemma:strongly_convex_PGD} that the 
objective functional in~\eqref{eq:equiv_descent_wasserstein} is strongly convex
w.r.t. reference measure~$\mu_k$
(note that $\tau < 1/\alpha = 1/(l + \lambda) < 1/\lambda$),
and therefore has a unique minimizer.
Following similar steps as in the proof of
  Theorem~\ref{thm:conv_prox_recursion_wasserstein} to characterize
the critical point of~\eqref{eq:equiv_descent_wasserstein}, we get
that $T_{\mu_{k+1} \rightarrow \mu_k} = \id + \tau \xi_k$, and by
substitution in the above, we obtain:
\begin{align*}
 & \int_{\Omega} \left \langle \xi_{k+1} , T_{\mu_{k+1}
      \rightarrow \mu_k} - \id \right \rangle d\mu_{k+1} \\
 & \geq  \int_{\Omega} \left \langle \xi_k(T_{\mu_{k+1}
      \rightarrow \mu_k}),  T_{\mu_{k+1} \rightarrow \mu_k} - \id \right \rangle d\mu_{k+1}  \\
  &\qquad \qquad - l W_2^2(\mu_k, \mu_{k+1}) \\
 &= \int_{\Omega} \left \langle \xi_k(T_{\mu_{k+1}
      \rightarrow \mu_k}) - \xi_k ,  T_{\mu_{k+1} \rightarrow \mu_k} - \id \right \rangle d\mu_{k+1}  \\
  &\qquad + \int_{\Omega} \left \langle \xi_k ,  T_{\mu_{k+1} \rightarrow \mu_k} - \id \right \rangle d\mu_{k+1}  \\
&\qquad \qquad - l W_2^2(\mu_k, \mu_{k+1}) \\
 &\geq - \lambda  W_2^2(\mu_k, \mu_{k+1}) + \frac{1}{\tau} W_2^2(\mu_k, \mu_{k+1}) \\
 &\qquad \qquad - l W_2^2(\mu_k, \mu_{k+1}) \\
&= \left( \frac{1}{\tau} - \alpha \right) W_2^2(\mu_k, \mu_{k+1}).
\end{align*}
Moreover, from the convexity of $F$ and
Lemma~\ref{lemma:first_order_convexity} 
(with $\mu_{k+1}$ as the reference measure) 
we have: 
\begin{align*}
  F(\mu_k) \geq F(\mu_{k+1}) + \int_{\Omega} \left \langle \xi_{k+1} , T_{\mu_{k+1} \rightarrow \mu_k} - \id \right \rangle
  d\mu_{k+1} .
\end{align*}
Substituting in the latest inequality, we obtain:
\begin{align*}
  F(\mu_k) \geq F(\mu_{k+1}) + \left( \frac{1}{\tau} - \alpha \right)
  W_2^2(\mu_k, \mu_{k+1}).
\end{align*}
From this inequality, we deduce that $\mu_{k+1}$ belongs to the
$F$-sublevel set of $\mu_k$, and consequently that the sequence
$\lbrace \mu_k \rbrace_{k \in \mathbb{N}}$ is contained in
$\mathcal{S}(\mu_0)$, the $F$-sublevel set of $\mu_0$. From here,
following similar steps as in the proof of
Theorem~\ref{thm:conv_prox_recursion_wasserstein}, we conclude that
the sequence $\lbrace \mu_k \rbrace_{k \in \mathbb{N}}$ is convergent
and $\lim_{K \rightarrow \infty} W_2^2(\mu_K, \bar{\mu}) = 0$ for some
$\bar{\mu} \in \mathcal{S}(\mu_0)$.
As the sequence $\lbrace \mu_k \rbrace_{k \in \mathbb{N}}$ is
generated by \eqref{eq:equiv_descent_wasserstein}, the limit
$\bar{\mu}$ must be one of its fixed points, again following
similar reasoning as in Theorem~\ref{thm:conv_prox_recursion_wasserstein}.  
Since $F$ is strictly convex, we get that the only fixed point
of \eqref{eq:equiv_descent_wasserstein} is $\mu^\star$. We therefore have
$\bar{\mu} = \mu^\star$.
 \oprocend
\end{proof}
Now Theorem~\ref{thm:target_dynamics_euclidean} allows us to consider
the transport in $\mathcal{P}(\Omega)$ given by the following proximal
scheme in $\Omega$:
\begin{align}
	x^{+} = \arg \min_{z \in \Omega} \frac{1}{2\tau} | x -z |^2 + f(z),
	\label{eq:target_dyn_euclidean}
\end{align}
where $x \sim \mu$ and $f = \left. \frac{\delta F}{\delta \nu}
\right|_{\mu}$. This scheme is convergent according to
Theorem~\ref{thm:proximal_scheme_convergence_target}.
%
%

\section{Multi-agent proximal descent
  algorithms} \label{sec:multi_agent_coverage_control} In this
section, we bring the sample-based, proximal descent schemes of the
previous section to a form that is closer to the more familiar
multi-agent cooperative control algorithms. We achieve this by a
direct discretization of the functional.  By doing so, we are able to
retain some convergence properties of the algorithms, as shown in this
section.  We then show that, in the limit of space and time
discretizations, the corresponding algorithm recovers the lost
properties.

We start by describing the multi-agent system by an appropriate
probability distribution. Recall that the configuration of the
collective is given by $\mathbf{x} = (x_1, \ldots, x_N)$, with $x_i
\in \Omega$ for $i \in \lbrace 1, \ldots, N \rbrace$.  Let
$\widehat{\mu}^{N}_{\mathbf{x}} = \frac{1}{N} \sum_{i=1}^N
\delta_{x_i}$ be the discrete measure in $\mathcal{P}(\Omega)$
corresponding to the configuration~$\mathbf{x}$, where
$\delta_{x_i}$ is the Dirac measure supported at~$x_i$.
For a macroscopic description of the transport, we first let the
macroscopic configuration be specified by an absolutely continuous
probability measure, and since $\widehat{\mu}^N_{\mathbf{x}}$ is is
not absolutely continuous, we consider an alternative absolutely
continuous probability measure $\widehat{\mu}^{h,N}_{\mathbf{x}}$
through its density function using a smooth kernel, as follows:
\begin{align}
  \widehat{\mu}^{h,N}_{\mathbf{x}}(x) = \frac{1}{N} \sum_{i=1}^N
  K_h(x-x_i),
  \label{eq:density_kernel}
\end{align}
where~$h>0$ is the bandwidth of the kernel. With a slight abuse of
notation, we allow $\widehat{\mu}^{h,N}_{\mathbf{x}}$ to denote both
the absolutely continuous measure and its corresponding density
function. We also denote, for $x \in \Omega$,
$\widehat{\mu}^{h,1}_{x}$ simply by $\widehat{\mu}^{h}_{x}$. Thus, we
have $\widehat{\mu}^{h,N}_{\mathbf{x}} = \frac{1}{N} \sum_{i=1}^N
\widehat{\mu}^h_{x_i}$, for $\mathbf{x} \in \Omega^N$.
\begin{assumption}[\bf \emph{Properties of kernel and kernel-based
    measures}]
\label{ass:kernel_properties}
For $h>0$ and $z \in \Omega$ and a kernel-based probability
measure $\widehat{\mu}^h_z$ defined as
in~\eqref{eq:density_kernel} for $N=1$, the following hold:\\
\textbf{(i) Smoothness:} The kernel $K_h$ is smooth,
$K_h \in
C^{\infty}(\Omega)$, for every $h > 0$. \\
\textbf{(ii) Monotonicity of support:} For any $z \in \Omega$ and $h_1
< h_2$, we let $\supp \left( \widehat{\mu}^{h_1}_{z} \right) \subset
\supp \left( \widehat{\mu}^{h_2}_{z} \right)$. \\
\textbf{(iii) Containment:} For every $h>0$, there exists a set
$\tilde{\Omega}_h \subset \Omega$ (relatively) open, such that for $z
\in \tilde{\Omega}_h$, the support of the measure
$\widehat{\mu}^{h}_{z}$ satisfies $\supp(\widehat{\mu}^{h}_{z})
\subset \Omega$.  Moreover, $\lim_{h \rightarrow 0} \tilde{\Omega}_h =
\Omega$ in Hausdorff
distance. \\
\textbf{(iv) Total variation convergence:} Let $\mathcal{M}$ be the
space of all measureable functions over $\Omega$. It holds that
$\lim_{h \rightarrow 0} \sup_{f \in \mathcal{M}} \left \lbrace
  \int_{\Omega} f(z) K_h(x-z) \dvol(z) - f(x) \right \rbrace = 0$,
that is, the kernel-based measure converges uniformly to the Dirac
measure as $h \rightarrow 0$.
\end{assumption}
An example kernel for~\eqref{eq:density_kernel} that satisfies
Assumption~\ref{ass:kernel_properties} is the truncated Gaussian
kernel restricted to an open ball $B_{h}(x_i)$ of radius~$h$ centered
at~$x_i$, given by $K_h(x - x_i) = \frac{1}{C}\text{exp}
\left(\frac{- \left|x - x_i \right|^2}{2h^2} \right) \mathbf{1}_{B_{h}(x_i)}(x)$,
where $C = \int_{B_{h}(x_i)} \text{exp} \left(\frac{- \left|x -
    x_i \right|^2}{2h^2} \right) \dvol(x)$ is the normalizing constant.
\subsection{Discretization of functional~$F$ and its properties}
We define an aggregate objective function $F^{h,N}$ for the
multi-agent system as the discretization of the functional $F$, for
$h>0$, as follows:
\begin{align}
  F^{h,N}(\mathbf{x}) = F(\widehat{\mu}^{h,N}_{\mathbf{x}}),
  \label{eq:F^h,N_definition}
\end{align} 
and, subsequently, analyze its properties. 
First note that $F^{h,N}$ is invariant under permutations, that is,
for~$\mathbf{x} \in \tilde{\Omega}^N_h$ and~$P \in \real^{N \times N}$
a permutation, we have~$F^{h,N}(\mathbf{x}) = F^{h,N}((P \otimes
I_d)~\mathbf{x})$.
The following lemma establishes the almost sure convergence of the
$F^{h,N}$ to~$F$ as $h \rightarrow 0, N \rightarrow \infty$:
\begin{lemma}[\bf \emph{Convergence as $h \rightarrow 0,~N \rightarrow \infty$}] \label{lemma:conv_h,N}
    Let Assumption~\ref{ass:kernel_properties} and the Fr{\'e}chet
    differentiability of the functional~$F$ hold, and let
    $x_{i} \sim \mu$ for $i \in \until{N}$, independent and identically distributed.
    Then, we have $\lim_{h \rightarrow 0} \lim_{N \rightarrow \infty}
    F^{h,N}(x_1, \ldots, x_N) = F(\mu)$, $\mu$-almost surely.
\end{lemma}
The following lemma relates the derivative of the function $F^{h,N}$
to the Fr\'echet derivative of the functional $F$:
\begin{lemma}[\bf \emph{Derivative of $F^{h,N}$}] 
  \label{lemma:der_F^h,N_vs_F}
  Let Assumption~\ref{ass:kernel_properties} and the Fr{\'e}chet
  differentiability of the functional~$F$ hold, and let $h>0$ with set
  $\tilde{\Omega}_h$ as in
  Assumption~\ref{ass:kernel_properties}-(iii).  For $\mathbf{x} = (z,
  \eta) \in \tilde{\Omega}_h\times \tilde{\Omega}_h^{N-1}$, we have
  that the derivative of the function $F^{h,N}$ satisfies:
\begin{align*}
  \partial_1 F^{h,N}(z, \eta) = \frac{1}{N} \int_{\supp \left(
      \widehat{\mu}^h_{z} \right)} \nabla
  \varphi^{h,N}_{\mathbf{x}}~d\widehat{\mu}^h_{z},
\end{align*}
where $d\widehat{\mu}^h_{z} = \rho^{h}_z~\dvol$ with $\rho^{h}_z (x) =
K(x-z,h)$, $\varphi^{h,N}_{\mathbf{x}} = \frac{\delta F}{\delta \nu}
\left. \right|_{\widehat{\mu}^{h,N}_{ \mathbf{x}}}$ and $\partial_1$
denotes the derivative w.r.t the first argument.
\end{lemma}
From the invariance of $F^{h,N}$ under permutations, 
the expression in Lemma~\ref{lemma:der_F^h,N_vs_F} 
holds for the partial derivative of $F^{h,N}$ w.r.t 
every component of $\mathbf{x}$.
%
%
In what follows, we will make the following assumption characterize the
behavior of the discretization $F^{h,N}$ along the boundary through
the following assumption:
\begin{assumption}[\bf \emph{Boundary conditions}]
\label{ass:F^h,N_boundary_condns}
The function $F^{h,N}$ is Fr{\'e}chet differentiable
and its derivative satisfies the boundary 
condition $\partial_1 F^{h,N}(z, \xi) \cdot \mathbf{n}(z) =
0$ for $z \in \partial \tilde{\Omega}_h$ and all $\xi \in
\tilde{\Omega}^{N-1}_h$.
\end{assumption}
In general, note that $F^{h,N} : \Omega^N \rightarrow \real$ is
nonconvex in spite of being the discretization of a strictly
geodesically convex functional $F : \mathcal{P}(\Omega) \rightarrow
\real$.  This is because the notion of convexity of functions over
$\Omega^N$, which is the domain of the function $F^{h,N}$, is not
implied by the notion of geodesic convexity over the space of
probability measures over $\Omega$. In this way, for $\mathbf{x},
\mathbf{y} \in \Omega^N$ with $\sum_{i=1}^N \frac{1}{N} \delta_{x_i},
\sum_{i=1}^N \frac{1}{N} \delta_{y_i} \in \mathcal{P}(\Omega)$ being
the corresponding discrete measures, the supports of the geodesics
(when they exist) between $\sum_{i=1}^N \frac{1}{N} \delta_{x_i}$ and
$\sum_{i=1}^N \frac{1}{N} \delta_{y_i}$ in $\mathcal{P}(\Omega)$ do
not necessarily correspond to the straight line segment between
$\mathbf{x}$ and $\mathbf{y}$ in $\Omega^N$. In what follows, we
identify a condition that can guarantee convexity of the discretized
functional.  We note that this condition is employed later to prove
the convergence of the discrete algorithms to local minimizers.
\begin{definition}[\bf \emph{Cyclical monotonicity}]
  A set $\Gamma \subset \Omega \times \Omega$ is cyclically monotone
  if any sequence $\lbrace (x_i, y_i) \rbrace_{i=1}^N$, with $(x_i,
  y_i) \in \Gamma$, satisfies:
\begin{align*}
	\sum_{i=1}^N |x_i - y_i|^2 \leq \sum_{i=1}^N |x_i - y_{\sigma(i)}|^2,
\end{align*}
where $\sigma$ is any permutation.
\end{definition}
\noindent
We note that the notion of cyclical monotonicity is a geometric property (Chapter~5, \cite{CV:08})
that indicates that the assignment $\lbrace (x_i, y_i) \rbrace_{i=1}^N$ (as specified by the pairings)
of points $\lbrace x_i \rbrace_{i=1}^N$ to $\lbrace y_i \rbrace_{i=1}^N$ 
is optimal w.r.t. the transport cost.
Now for $\delta > 0$, we define a subset $\Delta_{\delta}
\subset
\Omega^N$ as follows:
\begin{align*}
  \Delta_{\delta} = \left \lbrace \left. \mathbf{z} = (z_1, \ldots, z_N) \in
      \mathring{\Omega}^N \right| |z_i - z_j| > \delta ,~ \forall ~ i
    \neq j \right \rbrace.
\end{align*}
For every $\mathbf{x} \in \Delta_{\delta}$, we now define a set
$\Gamma_{\mathbf{x}} \subset \Omega^N$ such that for all $\mathbf{y}
\in \Gamma_{\mathbf{x}}$, we have:
\begin{align*}
	\sum_{i=1}^N |x_i - y_i|^2 \leq \sum_{i=1}^N |x_i - y_{\sigma(i)}|^2,
\end{align*}
for any permutation $\sigma$. In other words, $\Gamma_{\mathbf{x}}$ is
the subset of $\Omega^N$ such that for any $\mathbf{y} \in
\Gamma_{\mathbf{x}}$,  $\lbrace (x_i, y_i) \rbrace_{i=1}^N$ is cyclically monotone.
We now establish through the following lemma that the set
$\Gamma_{\mathbf{x}}$ contains an open neighborhood of $\mathbf{x}$:
\begin{lemma}[\bf \emph{$\Gamma_{\mathbf{x}}$ contains an open
    neighborhood of $\mathbf{x}$}]
\label{lemma:neighborhood_cyclically_monotone}
For any $\delta > 0$ and $\mathbf{x} \in \Delta_{\delta}$, there
exists an open neighborhood $\mathcal{N}(\mathbf{x}) \subset \Omega^N$
of $\mathbf{x}$ such that $\mathcal{N}(\mathbf{x}) \subset
\Gamma_{\mathbf{x}}$.
\end{lemma}
From Lemma~\ref{lemma:neighborhood_cyclically_monotone},
  we get that for $\mathbf{x} \in \Delta_{\delta}$ with a given
$\delta > 0$, there is a $\bar{h}_\delta$ such that for all $0<h
<\bar{h}_\delta$, the supports of the components
$\widehat{\mu}^{h}_{x_i}$ of the measure
$\widehat{\mu}^{h,N}_{\mathbf{x}}$ can be made disjoint.
\begin{lemma}[\bf \emph{Relaxation to atomless measures}]
\label{lemma:relaxation_atomless_measures}
For any $\delta > 0$ and $\mathbf{x} \in \Delta_{\delta}$ and $\mathbf{y} \in
\Gamma_{\mathbf{x}}$, there is $\bar{h}_{\delta} > 0$ such that for $0
\leq h \leq \bar{h}_{\delta}$ and the measures
$\widehat{\mu}^{h,N}_{\mathbf{x}}, \widehat{\mu}^{h,N}_{\mathbf{y}}$
defined in \eqref{eq:density_kernel}, the optimal transport map
$T_{\widehat{\mu}^{h,N}_{\mathbf{x}} \rightarrow
  \widehat{\mu}^{h,N}_{\mathbf{y}}}$ from
$\widehat{\mu}^{h,N}_{\mathbf{x}}$ to
$\widehat{\mu}^{h,N}_{\mathbf{y}}$ satisfies:
\begin{align*}
  \left( T_{\widehat{\mu}^{h,N}_{\mathbf{x}} \rightarrow
      \widehat{\mu}^{h,N}_{\mathbf{y}}} - \id \right) (z) = y_i - x_i,
  ~~~ \forall\; z \in \supp\left( \widehat{\mu}^{h}_{x_i} \right).
\end{align*}
\end{lemma}
 We note that
  Lemma~\ref{lemma:relaxation_atomless_measures} is an extension of
  existing results for Dirac measures (Chapter~5, \cite{CV:08}) to
  kernel-based measures.
In this way, Lemma~\ref{lemma:relaxation_atomless_measures}
essentially establishes that for $\mathbf{x} \in \Delta_{\delta}$ and
any $\mathbf{y} \in \Gamma_{\mathbf{x}}$, the optimal transport from
$\widehat{\mu}^{h,N}_{\mathbf{x}}$ to
$\widehat{\mu}^{h,N}_{\mathbf{y}}$ is simply achieved by the
translation of components $\widehat{\mu}^{h}_{x_i}$ along the rays
$y_i - x_i$ to $\widehat{\mu}^{h}_{y_i}$ for each $i \in \lbrace 1,
\ldots, N \rbrace$.
\begin{corollary}[\bf \emph{$L^2$-Wasserstein distance}]
  For any $\delta > 0$ and $\mathbf{x} \in \Delta_{\delta}$ and $\mathbf{y} \in
  \Gamma_{\mathbf{x}}$, there is a $\bar{h}_\delta>0$
    such that for any $0 < h \leq \bar{h}_{\delta}$:
\begin{align*}
  W_2^2 \left( \widehat{\mu}^{h,N}_{\mathbf{x}},
    \widehat{\mu}^{h,N}_{\mathbf{y}} \right) = \frac{1}{N}
  \sum_{i=1}^N | x_i - y_i |^2.
\end{align*}
\end{corollary}
%
With the above results we now establish the following:
\begin{lemma}[\bf \emph{$l/N$-smoothness of $F^{h,N}$}]
  \label{lemma:alpha_smoothness_F^{h,N}} 
  Let $F$ satisfy the regularity conditions of Assumption~\ref{ass:regularity}
  and let Assumption~\ref{ass:kernel_properties} hold.
  For any $\delta > 0$, $\mathbf{x} \in \Delta_{\delta}$, $h \in (0,\bar{h}_{\delta}]$ and
  $\mathbf{y} \in \Gamma_{\mathbf{x}}$:
	\begin{align*}
		\left| \left \langle \nabla F^{h,N}(\mathbf{y}) - \nabla F^{h,N}(\mathbf{x}), \mathbf{y} - \mathbf{x} \right \rangle \right| \leq \frac{l}{N} \| \mathbf{y} - \mathbf{x} \|^2.
	\end{align*}
\end{lemma}
\begin{lemma}[\bf \emph{Comparison lemma for $F^{h,N}$ on cyclically
    monotone sets}]
  \label{lemma:comparison_cyclically_monotone} Let $F$ be a
  Fr{\'e}chet differentiable and geodesically convex functional (in
  the sense of Definition~\ref{defn:generalized_geodesic_convexity}).
  For any $\delta > 0$, $\mathbf{x} \in \Delta_{\delta}$, $h \in (0,\bar{h}_{\delta}]$ and
  $\mathbf{y} \in \Gamma_{\mathbf{x}}$:
\begin{align*}
  F^{h,N}(\mathbf{y}) \geq F^{h,N}(\mathbf{x}) + \left \langle \nabla
    F^{h,N}(\mathbf{x}) , \mathbf{y} - \mathbf{x} \right \rangle.
\end{align*}	
\end{lemma}
We remark here that $F^{h,N}$ is convex in the limited sense
established by the comparison result in Lemma
\ref{lemma:comparison_cyclically_monotone}, and this does not
necessarily generalize to the entire domain $\Omega^N$, due to which
the function $F^{h,N}$ can be non-convex in general.
%
%
\subsection{Multi-agent proximal descent algorithms}
\label{subsec:mas-proximal-descent}
We formulate the proximal descent algorithm on the function $F^{h,N}$
as follows:
\begin{align}
  \mathbf{x}^{+} \in \arg \min_{\mathbf{z} \in \tilde{\Omega}^N_h}
  \frac{1}{2\tau} \| \mathbf{x} - \mathbf{z} \|^2 +
  F^{h,N}(\mathbf{z}).
	\label{eq:prox_grad_descent_omega_N}
\end{align}
Even though $F^{h,N}$ is in general nonconvex, we can establish strong
convexity of the proximal descent objective function
in~\eqref{eq:prox_grad_descent_omega_N} under some conditions through
the following lemma:
\begin{lemma}[\bf \emph{Strong convexity of objective function}]
\label{lemma:strong_convexity_prox_grad_omega_N}
For~$\alpha = l/N$, the function
$G^{h,N}_{\mathbf{x}}(\mathbf{z}) = \frac{1}{2\tau} \| \mathbf{x} -
\mathbf{z} \|^2 + F^{h,N}(\mathbf{z})$ is $\left(\frac{1}{\tau} -
  \alpha \right)$-strongly convex for $0 < \tau < \frac{1}{\alpha}$.
\end{lemma}
It follows from Lemma~\ref{lemma:strong_convexity_prox_grad_omega_N}
that the minimizer in~\eqref{eq:prox_grad_descent_omega_N} is unique
for $\alpha$-smooth $F^{h,N}$ and sufficiently small $\tau$.
Now, with $\mathbf{x}_{-i} = (x_1, \ldots, x_{i-1}, x_{i+1}, \ldots,
x_N) \in \tilde{\Omega}_h^{N-1}$, we can write
$F^{h,N}(x_1, \ldots, x_N) = \frac{1}{N} \sum_{i=1}^N F^{h,N}(x_1,
  \ldots, x_N) = \frac{1}{N} \sum_{i=1}^N F^{h,N} (x_i,
  \mathbf{x}_{-i})$.
  By means of this decomposition, the proximal gradient descent
  \eqref{eq:prox_grad_descent_omega_N} can be decomposed into the
  following agent-wise update, for $i \in \until{N}$:
\begin{align*}
  x_i^{+} = \arg \min_{z \in \overline{\Omega}_h} \frac{1}{2\tau} |x_i -
  z|^2 + F^{h,N} (z, \mathbf{x}_{-i}^{+}).
\end{align*} where $\overline{\Omega}_h$ is the closure
of $\tilde{\Omega}_h$.
Note that the above scheme requires $\mathbf{x}_{-i}^{+}$.  In other
words, to implement the above algorithm, every agent $i$, at time $k$,
requires the positions of the other agents at a future time $k+1$,
posing a hurdle for implementation.  To avoid this problem, we
consider the following proximal descent scheme:
\begin{align}
  x_i^{+} = \arg \min_{z \in \overline{\Omega}_h} \frac{1}{2\tau} |x_i -
  z|^2 + F^{h,N}(z, \mathbf{x}_{-i}),
	\label{eq:multi_agent_proximal_grad}
\end{align} 
for every~$i \in \until{N}$.
It follows from Lemma~\ref{lemma:strong_convexity_prox_grad_omega_N}
that the objective function in~\eqref{eq:multi_agent_proximal_grad} is
also strongly convex, and thereby has a unique minimizer.
We now present the following result on the convergence of
\eqref{eq:multi_agent_proximal_grad} to the local minimizers of
$F^{h,N}$:
\begin{theorem}[\bf 
\emph{Convergence of \eqref{eq:multi_agent_proximal_grad} to critical points of $F^{h,N}$}]
 Let $F^{h,N}$ be $\alpha$-smooth and satisfy
  Assumption~\ref{ass:F^h,N_boundary_condns}. For $\tau <
\frac{2}{3\alpha}$, the sequence $\lbrace \mathbf{x}(k) \rbrace_{k \in
  \mathbb{N}}$ generated by the update scheme
\eqref{eq:multi_agent_proximal_grad} converges to a critical point
$\mathbf{x}^*$ of $F^{h,N}$ that is not a local maximizer, for all
initial conditions $\mathbf{x}(0) \in \overline{\Omega}^N_h$.
Moreover, if the critical
point $\mathbf{x}^* \in \Delta_{\delta}$ for some $\delta > 0$ and $h
\in (0, \bar{h}_{\delta}]$, then $\mathbf{x}^*$ is a local minimizer.
\label{thm:convergence_multi_agent_local_min}
\end{theorem}
\begin{proof}
  We first consider the objective function in
  \eqref{eq:multi_agent_proximal_grad}, $J_i(z) = \frac{1}{2\tau} |x_i
  - z|^2 + F^{h,N}(z, \mathbf{x}_{-i})$, with $z\in
    \overline{\Omega}_h$. The inner product of the gradient of $J_i$ on
   $ z \in \partial \overline{\Omega}_h$ with the outward normal
  $\tilde{\mathbf{n}}$ to $\partial \overline{\Omega}_h$,
	is given by:
\begin{align*}
  \nabla J_i (z) \cdot \tilde{\mathbf{n}} (z)
  &= \frac{1}{\tau} (z - x_i) \cdot \tilde{\mathbf{n}}(z) + 
  \partial_1  F^{h,N}(z, \mathbf{x}_{-i}) \cdot \tilde{\mathbf{n}}(z) \\
  &= \frac{1}{\tau} (z - x_i) \cdot \tilde{\mathbf{n}}(z) \geq 0,
\end{align*}
with the inequality being strict when $x_i \notin \partial
\overline{\Omega}_h$.  This implies that the $x_i^{+} \in \partial
\overline{\Omega}_h$ cannot be a minimizer if $x_i \notin \partial
\overline{\Omega}_h$, and if $x_i \in \partial \overline{\Omega}_h$,
we will have $x_i^{+} = x_i$. In both cases, the minimizer $x_i^{+}$
is also a critical point of the function $J_i$.  This allows us to
express~\eqref{eq:multi_agent_proximal_grad} equivalently by:
\begin{align}
\label{eq:multi_agent_prox_update_diff}
x^{+}_i = x_i - \tau \partial_1 F^{h,N}(x_i^{+}, \mathbf{x}_{-i}).
\end{align}
We note that in the limit $\tau \rightarrow 0$, we get a gradient flow
that can be shown to converge to a critical point of $F^{h,N}$.  We
therefore hope that this property is preserved over a neighborhood of
$\tau = 0$. In what follows, we establish that this is indeed the case
and provide a sufficient strict upper bound on $\tau$ for which the
property is preserved.
From $\alpha$-smoothness of $F^{h,N}$, we get:

{\small 
\begin{align*}
  &\left| F^{h,N}(\mathbf{x}^{+}) - F^{h,N}(\mathbf{x}) - \sum_{i=1}^N
    \left \langle \partial_1 F^{h,N}(x_i, \mathbf{x}_{-i}), x_i^{+} -
      x_i \right \rangle \right| \\
  &\leq \frac{\alpha}{2} \| \mathbf{x}^{+} - \mathbf{x} \|^2.
\end{align*}
}
We can rewrite the above as:

{\small
\begin{align*}
  &\left| F^{h,N}(\mathbf{x}^{+}) - F^{h,N}(\mathbf{x}) - \sum_{i=1}^N \left \langle \partial_1 F^{h,N}(x_i^{+}, \mathbf{x}_{-i}), x_i^{+} - x_i \right \rangle \right. \\
  & \left. - \sum_{i=1}^N \left \langle \partial_1 F^{h,N}(x_i, \mathbf{x}_{-i}) -  \partial_1 F^{h,N}(x_i^{+}, \mathbf{x}_{-i}) , x_i^{+} - x_i \right \rangle	\right| \\
  & \leq \frac{\alpha}{2} \| \mathbf{x}^{+} - \mathbf{x} \|^2.
\end{align*}
}
\begin{sloppypar}
By~\eqref{eq:multi_agent_prox_update_diff}, we now have $-
\sum_{i=1}^N \left \langle \partial_1 F^{h,N}(x_i^{+},
  \mathbf{x}_{-i}), x_i^{+} - x_i \right \rangle = \frac{1}{\tau} \|
\mathbf{x}^{+} - \mathbf{x} \|^2$ and by the $\alpha$-smoothness of
$F^{h,N}$:
\begin{align*}
  &\left| \sum_{i=1}^N \left \langle \partial_1 F^{h,N}(x_i, \mathbf{x}_{-i}) -  \partial_1 F^{h,N}(x_i^{+}, \mathbf{x}_{-i}) , x_i^{+} - x_i \right \rangle \right| \\
  &\leq \alpha \| \mathbf{x}^{+} - \mathbf{x} \|^2.
\end{align*}
\end{sloppypar}
From the above inequalities, we therefore obtain:
\begin{align*}
  F^{h,N}(\mathbf{x}^{+}) \leq F^{h,N}(\mathbf{x}) - \left(
    \frac{1}{\tau} - \frac{3\alpha}{2} \right) \| \mathbf{x}^{+} -
  \mathbf{x} \|^2.
\end{align*}
Thus, for $\tau < \frac{2}{3\alpha}$, when every agent follows the
update \eqref{eq:multi_agent_proximal_grad}, we get a descent in
$F^{h,N}$, and $\mathbf{x}^{+}$ belongs to the $F^{h,N}$-sublevel set
of $\mathbf{x}$.
We can express the above inequality for any time instant $k \in
\mathbb{N}$ as:
\begin{align*}
\begin{aligned}
  F^{h,N}(\mathbf{x}(k+1)) \leq &F^{h,N}(\mathbf{x}(k)) \\
			&\quad - \left( \frac{1}{\tau} - \frac{3\alpha}{2} \right) \|
  \mathbf{x}(k+1) - \mathbf{x}(k) \|^2.
\end{aligned}
\end{align*}
Summing over $k = 0, \ldots, K-1$, we obtain:
\begin{align*}
  F^{h,N}(\mathbf{x}(K)) \leq &F^{h,N}(\mathbf{x}(0)) \\
  & \quad - \left( \frac{1}{\tau} - \frac{3\alpha}{2} \right) \sum_{k=1}^{K}
  \| \mathbf{x}(k) - \mathbf{x}(k-1) \|^2,
\end{align*}
and it follows that:
\begin{align*}
  &\sum_{k=1}^{K} \| \mathbf{x}(k) - \mathbf{x}(k-1) \|^2 \\
  &\leq \left(\frac{1}{ \frac{1}{\tau} - \frac{3\alpha}{2}} \right)
  \left( F^{h,N}(\mathbf{x}(0)) - F^{h,N}(\mathbf{x}(K)) \right).
\end{align*}
Since the sequence $\lbrace \mathbf{x}(k) \rbrace_{k \in
  \mathbb{N}}$ belongs to the $F^{h,N}$-sublevel set of
$\mathbf{x}(0)$ (for all $\mathbf{x}(0) \in \overline{\Omega}_h^N$),
which is a subset of $\overline{\Omega}_h^N$ (compact), it is
precompact. By the boundedness above, in the limit $K \rightarrow
\infty$, we get $\lim_{K \rightarrow \infty} \| \mathbf{x}(K) -
\mathbf{x}(K-1) \|^2 = 0$. 

{Since $\overline{\Omega}_h$
  is compact, there is a convergent subsequence
  $\{\mathbf{x}(k_\ell)\}$ to a point
  $\overline{\mathbf{x}}\in\overline{\Omega}_h^N$. Given $\mathbf{x}$,
  define the mapping
  \begin{align*}
    G_{\mathbf{x}}^{h,N}(\mathbf{z}) = \left( \frac{1}{\tau} -
      \frac{3\alpha}{2}\right) \|\mathbf{x} - \mathbf{z} \|^2 +
    F^{h,N}(\mathbf{z}), \quad \mathbf{z} \in \overline{\Omega}_h^N.
  \end{align*}
  Let $\overline{\mathbf{x}}^+ $ be the next iteration
  of~\eqref{eq:multi_agent_prox_update_diff} from
  $\overline{\mathbf{x}}$. Then, from the above,
  $G_{\overline{\mathbf{x}}}^{h,N}(\overline{\mathbf{x}}^+) \le
  F^{h,N}(\overline{\mathbf{x}}) =
  G_{\overline{\mathbf{x}}}^{h,N}(\overline{\mathbf{x}})$. Due to the
  fact that $\mathbf{x}(k_\ell)$ converges to $\overline{\mathbf{x}}$,
  we also have that
  $G_{\overline{\mathbf{x}}}^{h,N}(\overline{\mathbf{x}}) =
  F^{h,N}(\overline{\mathbf{x}}) \le
  G_{\mathbf{x}(k_\ell)}^{h,N}(\mathbf{x}(k_\ell +1))$, for all
  $\ell$. Following similar steps as in the proof of
  Theorem~\ref{thm:conv_prox_recursion_wasserstein}, one can find a
  constant $M$ such that $|G_{\overline{\mathbf{x}}}^{h,N}(\mathbf{z})
  - G_{\mathbf{x}(k_\ell)}^{h,N}(\mathbf{z})| \le M
  \|\overline{\mathbf{x}} - \mathbf{x}(k_\ell)\|$ for all $\mathbf{z}
  \in \overline{\Omega}_h^N$. This implies that
  $|G_{\overline{\mathbf{x}}}^{h,N}(\overline{\mathbf{x}}^+ ) -
  G_{\mathbf{x}(k_\ell)}^{h,N}(\mathbf{x}(k_\ell +1)) | \le \epsilon
  $, for all $\ell \ge \ell_0$. It is easy to see that
  $G_{\overline{\mathbf{x}}}^{h,N}(\overline{\mathbf{x}}^+) \le
  G_{\overline{\mathbf{x}}}^{h,N}(\overline{\mathbf{x}}) \le
  G_{\mathbf{x}(k_\ell)}^{h,N}(\mathbf{x}(k_\ell +1))$ holds, and thus
  $G_{\overline{\mathbf{x}}}^{h,N}(\overline{\mathbf{x}}^+) =
  F^{h,N}(\overline{\mathbf{x}})$, which can only happen when
  $\overline{\mathbf{x}}^+ = \overline{\mathbf{x}}$.  In other words,
  $\overline{\mathbf{x}} $ is a fixed point
  of~\eqref{eq:multi_agent_prox_update_diff}, and we thereby get:
\begin{align*}
  \partial_{1}F^{h,N}(\overline{x}_{i}, \overline{\mathbf{x}}_{-i}) = 0, ~~~ \forall ~i
  \in \lbrace 1, \ldots, N \rbrace,
\end{align*}
and $\nabla F^{h,N}(\overline{\mathbf{x}}) = 0$.  From here, the point
$\overline{\mathbf{x}}$ cannot be a local maximizer since $\lbrace
F^{h,N}(\mathbf{x}(k_\ell)) \rbrace_{k \in \mathbb{N}}$ is decreasing
and lower-bounded by $F^{h,N}(\overline{\mathbf{x}})$ and,
consequently, every neighborhood of $\overline{\mathbf{x}}$ contains
at least one point with a higher value of $F^{h,N}$. Note that this
conclusion applies for every accumulation point of the entire sequence
$\{\mathbf{x}(k)\}_{k \in \mathbb{N}}$.}  

Finally, suppose that an accumulation point $\overline{\mathbf{x}}$
satisfies $\overline{\mathbf{x}} \in \Delta_{\delta}$, for some
$\delta > 0$ and $h \in (0, \bar{h}_{\delta}]$. From Lemmas
\ref{lemma:comparison_cyclically_monotone} and
\ref{lemma:neighborhood_cyclically_monotone}, we conclude that there
exists an open ball $B(\overline{\mathbf{x}}) \subset \Omega^N$ such
that for all $\mathbf{x} \in B(\overline{\mathbf{x}})$, we have
$F^{h,N}(\mathbf{x}) \geq F^{h,N}(\overline{\mathbf{x}})$, which
implies that $\overline{\mathbf{x}}$ must be a local minimizer. \oprocend
\end{proof}
Theorem~\ref{thm:convergence_multi_agent_local_min} establishes the
convergence of~\eqref{eq:multi_agent_proximal_grad} to critical points
of the function $F^{h,N}$, which are not necessarily local
maximizers. This is a weaker result than
Theorem~\ref{thm:target_dynamics_euclidean}, which established
convergence of the transport scheme~\eqref{eq:target_dyn_euclidean} to
the global minimizer $\mu^\star$ of $F$. The guarantee is weakened
after the discretization of $F$, which is involved in defining the
multi-agent transport scheme (the convergence results for $F$ employ
the convexity properties of $F$, which are lost by $F^{h,N}$.)
However, we can still hope to achieve the convergence to the global
minimizer in the limit of particle and time discretizations, thereby
guaranteeing best performance asymptotically.  In the section that
follows, we evaluate this possibility.
%
%
%
\begin{remark}[\bf \emph{Distributed implementation}]
Furthermore, we note that the choice of the coverage objective
    function determines whether the resulting
    algorithm can be implemented in a distributed manner.  
    In particular, this depends on whether the local
    objective function for the agents (alternatively, the derivative
    of the coverage objective function) can be computed with purely
    local information by the agents as defined by a proximity graph.
    In the specific case of the Lloyd algorithm, we recall that the
    coverage objective function does possess this ``localizability''
    property naturally\footnote{When assuming that agents are able to
      communicate over the Delaunay graph. We also recall that, even
      in this case, this may not lead to a distributed computation
      over the $r$-disk graph.}. However, in general this may not be
    the case and we note here that in the absence of such a
    localizability property (according to a desired graph), the
    transport algorithm would either i) have to be augmented with an
    algorithm for local objective function computation for distributed
    implementation, ii) be constrained with local computations at the
    expense of some performance cost.
\end{remark}
\subsection{Continuous-time and many-particle limits}
\label{sec:multi_agent_N_infty_limit}
We now present a discussion of the continuous-time and many-particle limits for the
multi-agent transport scheme~\eqref{eq:multi_agent_proximal_grad},
retrieving \eqref{eq:target_dyn_euclidean} from
\eqref{eq:multi_agent_proximal_grad} as $N \rightarrow \infty$ and $h
\rightarrow 0$ limit. We know from
Theorem~\ref{thm:target_dynamics_euclidean} that transport of a
probability measure $\mu_0$ by \eqref{eq:target_dyn_euclidean}, which
is identical to the following:
\begin{align}
	\begin{aligned}
          x^{+} = \arg \min_{z \in \Omega} ~&\frac{1}{2\tau} |x - z|^2 + \varphi(z), \\
          & x \sim \mu.
	\end{aligned}
	\label{eq:multi-agent_limit_N_infty}
\end{align}
with $\varphi \equiv \frac{\delta F}{\delta \nu}_{\big| \mu}$,
is guaranteed to converge to the global minimizer $\mu^*$ of $F$.
The following lemma establishes the convergence of~\eqref{eq:multi_agent_proximal_grad} 
to~\eqref{eq:multi-agent_limit_N_infty} in the limit~$N \rightarrow \infty$:
\begin{lemma}[\bf \emph{Convergence of update scheme}]
	\label{lemma:transport_many_particle}  
	Let $\Omega$ and $F$ satisfy the regularity
          conditions of Assumption~\ref{ass:regularity}.
    The scheme~\eqref{eq:multi_agent_proximal_grad} converges 
    in distribution to~\eqref{eq:multi-agent_limit_N_infty} 
    in the limit $N \rightarrow \infty$.
\end{lemma}
We refer the reader to the Appendix for a proof of the above lemma.
Informally, we see that as $\tau \rightarrow 0$ in
\eqref{eq:multi-agent_limit_N_infty}, we have that $x^{+} \rightarrow
x$ and we let $\mathbf{v} (x) = \lim_{\tau \rightarrow 0} \frac{x^{+}
  - x}{\tau} = - \nabla \varphi(x)$.  We can thus expect the
solutions to \eqref{eq:multi-agent_limit_N_infty} converge to the
solution of the gradient flow under the vector field $\mathbf{v} = -
\nabla \varphi$.  We now show, in a weak sense, that the above
reasoning holds.
We observe that the vector field $\mathbf{v} = - \nabla \varphi$
satisfies a zero-flux boundary condition $\mathbf{v} \cdot \mathbf{n}
= \nabla \varphi \cdot \mathbf{n} = 0$ on $\partial \Omega$ owing to
the definition of the functional $F$.

We refer the reader to the Appendix for a 
detailed treatment of:
(a) the convergence of the scheme~\eqref{eq:multi_agent_proximal_grad}
to~\eqref{eq:multi-agent_limit_N_infty} as $N \rightarrow \infty$, 
(b) the convergence of~\eqref{eq:multi-agent_limit_N_infty} to
the gradient flow under the vector field $\mathbf{v} = - \nabla \varphi$,
and (c) the asymptotic stability of the gradient flow and its convergence to
the global minimizer~$\mu^*$.

Furthermore, we can naturally identify four modeling regimes 
distinguished by (i) length scale (macroscopic vs microscopic), and 
(ii) time scale (discrete-time vs continuous-time). The macroscopic
model serves to capture the $N \rightarrow \infty$ limit of the 
microscopic model, while the continuous-time model captures 
the~$h \rightarrow 0$ limit of the discrete-time model.
Table~\ref{table:micro_macro} summarizes the models of transport in the
various modeling regimes.

\begin{table*}[t]
\centering 
\renewcommand{\arraystretch}{1.6}
\begin{tabular}{|c| c| c|}
\hline 
Modeling regimes & Microscopic & Macroscopic \\
\hline
Discrete-time
& $\begin{aligned}
	\\
	x^{+} &= \arg \min_{z \in \Omega}~\frac{1}{2\tau} \left| x - z \right|^2 + \varphi(z) \\
	&\qquad x \sim \mu \qquad \varphi = \left. \frac{\delta F}{\delta \nu} \right|_{\nu = \mu}
	\\[1ex]
\end{aligned}$
& $\begin{aligned}
	\mu^{+} = \arg &\min_{\nu \in \mathcal{P}(\Omega)}~\frac{1}{2\tau} W_2^2(\mu, \nu) + \mathbb{E}_\nu \left[ \varphi \right] \\
\end{aligned}$
\\	
\hline 
Continuous-time
& $\begin{aligned} 
	\\
	&\qquad \qquad \dot{x}(t) = - \nabla \varphi_t(x(t)) \\
	&\qquad x(0) \sim \mu_0 \qquad \varphi_t = \left. \frac{\delta F}{\delta \nu} \right|_{\nu = \mu_t}
	\\[1ex]
\end{aligned}$ 
& $\begin{aligned}
	\frac{d \mu_t}{dt} = \nabla \cdot \left( \mu_t \nabla \varphi_t \right)
\end{aligned}$ \\
\hline 
\end{tabular}
\caption{Models of transport} 
\label{table:micro_macro}
\end{table*}

%
\section{Multi-agent coverage control algorithms}
\label{sec:case_study_coverage_control} 
In this section, we aim to place well-known multi-agent coverage
control algorithms in the literature~\cite{JC-SM-TK-FB:02-tra,
  JC:08-tac} within the multiscale theoretical framework established
in the previous sections, in an effort to understand the macroscopic
behavior of the coverage algorithms.  To do this, we first relate the
corresponding coverage objective functions used in both formulations
and then apply our results to analyze their behavior in the limit $N
\rightarrow \infty$. We begin with a widely-used aggregate objective
function for coverage control of multi-agent systems, the multi-center
distortion function, and then obtain its functional counterpart in the
space of probability measures. The multi-center distortion function
$\mathcal{H}_f: \Omega^N \rightarrow
\realnonnegative$~\cite{JC-SM-TK-FB:02-tra} is given by:
\begin{align}
  \mathcal{H}_f (\mathbf{x}) = \int_{\Omega} \min_{i \in \lbrace 1,
    \ldots, N \rbrace } f( |x - x_i| ) d\mu^*(x).
	\label{eq:loc_opt_agg_obj_func}
\end{align}
where $f:\realnonnegative \rightarrow \realnonnegative$ is a
strictly convex and non-decreasing function and $\mu^\star(x) = \rho^\star (x) \dvol$,
with $\rho^\star$ a target density in $\Omega$.
The Voronoi partition of~$\Omega$, $\lbrace \mathcal{V}_i
\rbrace_{i=1}^N$, generated by $\mathbf{x} \in \Omega^N$ facilitates
the analysis of $\mathcal{H}_f$ and is defined
is 
as follows:
\begin{align*}
  \mathcal{V}_i = \left \lbrace x \in \Omega \; \left| \; |x - x_i|
      \leq |x - x_j| ~\forall j \in \until{N} \right. \right \rbrace,
  \, \forall \, i.
\end{align*} 
The following proposition establishes the relationship between
$\mathcal{H}_f$ and the optimal transport cost $C_f$ in
\eqref{eq:OT_cost_f}:
\begin{proposition}[\bf \emph{Optimal transport formulation of
    coverage objective}\footnote{Refer to the Appendix
    for the proof.}]
\label{prop:aggregate_objective_OT_cost}
The aggregate objective function $\mathcal{H}_f$ as defined
in~\eqref{eq:loc_opt_agg_obj_func}, satisfies:
\begin{align*}
  \mathcal{H}_f(\mathbf{x}) &= \min_{\substack{\mathbf{w} \in
      \Delta^{N-1} }} C_f \left( \sum_{i=1}^N w_i \delta_{x_i}~,~
    \mu^\star \right).
\end{align*}
where $\Delta^{N-1} = \{ \mathbf{w} \in \realnonnegative^N \; \left|
  \; \sum_{i=1}^N w_i = 1 \right. \}$ is the $(N-1)$-simplex.
Furthermore, the minimizing weights $\mathbf{w}^\star = (w_1^\star,
\ldots, w_N^\star)$ are given by $w_i^\star =
\mu^\star(\mathcal{V}_i)$, where $\lbrace \mathcal{V}_i
\rbrace_{i=1}^N$ is the Voronoi partition of~$\Omega$.
\end{proposition}
%
%
The following corollary applies Proposition~\ref{prop:aggregate_objective_OT_cost}
to the special case of $f(x) = x^2$:
\begin{corollary}[\bf \emph{$L^2$-Wasserstein distance as aggregate objective function}]
  Applying Proposition~\ref{prop:aggregate_objective_OT_cost} with a
  quadratic cost $f(x) = x^2$ (and the corresponding aggregate
  objective function $\mathcal{H}_2$), we have:
\begin{align*}
  \mathcal{H}_2(\mathbf{x}) = 
  W_2^2 \left( \sum_{i=1}^N \mu^\star(\mathcal{V}_i) \delta_{x_i},
    \mu^\star \right).
\end{align*}
\end{corollary}
We now investigate the properties of the aggregate objective function
$\mathcal{H}_f$ in the limit $N \rightarrow \infty$.  

\begin{lemma} \label{lemma:conv_agg_obj_func}
	Let $\mu^\star\in \mathcal{P}(\Omega)$ be
    an absolutely continuous measure defining $\mathcal{H}_f$.
    Let $x_i \sim_{i.i.d} \mu$, for $i \in \until{N}$, where $\mu \in
    \mathcal{P}(\Omega)$ is any absolutely continuous probability
    measure such that $\supp(\mu) \supseteq \supp(\mu^\star)$.  It
    holds almost surely that $\lim_{N \rightarrow \infty} \mathcal{H}_f(\mathbf{x})
    = 0$. 
\end{lemma}
The previous result holds for any configuration of the points
$\lbrace x_i \rbrace_{i=1}^N$ as long as they are sampled from a
distribution whose support contains that of $\mu^\star$. Note that
this is consistent with what happens in the discrete particle case, in
the coverage control problem. In this case, critical point
configurations are given by the so-called centroidal Voronoi
configurations~\cite{JC-SM-TK-FB:02-tra}.  However, as the number of
agents goes to infinity, any configuration of points asymptotically
become centroids of their Voronoi regions. Thus, those positions
correspond to local optimizers of the discrete coverage control
problem. In this way, while the empirical measure $\frac{1}{N}
\sum_{i=1}^N \delta_{x_i}$ corresponding to the points $\lbrace x_i
\rbrace_{i=1}^N$ samples from $\mu$ converges uniformly almost surely
to $\mu$ (Glivenko-Cantelli theorem), the quantization energy
$\mathcal{H}_f$, converges to zero, which does not really reflect the
discrepancy between the measures $\mu$ and $\mu^\star$.  Thus, the
functional $\mathcal{H}_f$ suffers from this deficiency as a candidate
aggregate function for coverage control in the large scale limit.

Consider instead the following aggregate objective function:
\begin{align}
\label{eq:balanced_coverage_objective_function}
\bar{\mathcal{H}}_f(\mathbf{x}) = C_f \left( \frac{1}{N} \sum_{i=1}^N
  \delta_{x_i}~,~ \mu^\star \right).
\end{align}
This performance metric has been used before in the so-called area
(weight)-constrained coverage control problem~\cite{JC:08-tac} (the
weights $w_i = 1/N$ are balanced in the case of
\eqref{eq:balanced_coverage_objective_function}).

\begin{lemma} \label{lemma:conv_agg_obj_func_OT}
  Let $\mu^\star\in \mathcal{P}(\Omega)$ be an absolutely continuous
  measure and let $\bar{\mathcal{H}}_f$ be defined as
  in~\eqref{eq:balanced_coverage_objective_function}.  Let $x_i
  \sim_{i.i.d} \mu$, for $i \in \until{N}$, where $\mu \in
  \mathcal{P}(\Omega)$ is any absolutely continuous probability
  measure.  It holds almost surely that $\lim_{N \rightarrow \infty}
  \bar{\mathcal{H}}_f(\mathbf{x}) = C_f(\mu,\mu^\star)$.
\end{lemma}

Similarly to~\eqref{eq:prox_grad_descent_omega_N}, we can formulate a
multi-agent proximal descent algorithm on the aggregate objective
function $\bar{\mathcal{H}}_f$, with $f(x) = x^2$, as follows, for
every $i\in \until{N}$:
\begin{align}
  x_i^{+} = \arg \min_{z \in \Omega} \frac{1}{2\tau} |x_i - z|^2 +
  \bar{\mathcal{H}}_f(z, \mathbf{x}_{-i}).
  \label{eq:multi_agent_Voronoi_lloyd_descent}
\end{align} 
Note that this is a proximal formulation of the load-balancing variant
of the Lloyd's algorithm in~\cite{JC:08-tac}.
\begin{theorem}[\bf \emph{Convergence to generalized centroidal  Voronoi configuration
and~$\mu^\star$}]
The Lloyd proximal descent
\eqref{eq:multi_agent_Voronoi_lloyd_descent}, with $f(x) = x^2$,
converges to a local minimizer of $\bar{\mathcal{H}}_f$.
Furthermore, as $N \rightarrow \infty$, 
the proximal descent scheme \eqref{eq:multi_agent_Voronoi_lloyd_descent} 
converges to:
\begin{align}
	x^{+} = \arg \min_{z \in \Omega} \frac{1}{2\tau} |x-z|^2 + \phi_{\mu \rightarrow \mu^*}(z),
	\label{eq:proximal_coverage_scheme}
\end{align}
with $x \sim \mu$ and $\phi_{\mu \rightarrow \mu^*} = \left. \frac{\delta W_2^2(\nu,
\mu^\star)}{\delta \nu} \right|_{\mu}$, the Kantorovich potential for
optimal transport from $\mu$ to $\mu^\star$.
The sequence $\lbrace \mu_k \rbrace_{k \in \mathbb{N}}$
obtained as the transport of an absolutely continuous probability
measure~$\mu_0 \in \mathcal{P}(\Omega)$
by~\eqref{eq:proximal_coverage_scheme} with~$\tau \in (0, \bar{\tau})$
(for some~$\bar{\tau} > 0$), with $x_0 \sim \mu_0$,
converges weakly to $\mu^\star$ as $k \rightarrow \infty$. 
\end{theorem}
\begin{proof}
  Let $\widehat{\mu}^{h,N}_{\mathbf{x}}$ be defined as
  in~\eqref{eq:density_kernel} with a kernel satisfying
  Assumption~\ref{ass:kernel_properties}.  We see that
  $C_f(\widehat{\mu}^{h,N}_{\mathbf{x}}, \mu^\star)$ as a function
  of~$\mathbf{x}$ is $\alpha$-smooth for some $\alpha > 0$ (from
  Proposition~\ref{prop:l_smooth_C_f} in
  Appendix~\ref{app:agg_obj_func}
  and an application of
  Lemma~\ref{lemma:alpha_smoothness_F^{h,N}}).  Further, we note that
  $\bar{\mathcal{H}}_f(\mathbf{x}) = \lim_{h \rightarrow 0}
  C_f(\widehat{\mu}^{h,N}_{\mathbf{x}}, \mu^\star)$ and the
  $\alpha$-smoothness property carries over to the limit,
  as well as the comparison Lemma~\ref{lemma:comparison_cyclically_monotone} for
    $\bar{\mathcal{H}}_f(\mathbf{x})$.
  The convergence of~\eqref{eq:multi_agent_Voronoi_lloyd_descent} with
  $f(x) = x^2$ to a local minimizer of $\bar{\mathcal{H}}_f$ then
  follows from a similar version of Theorem~\ref{thm:convergence_multi_agent_local_min} 
  applied to $\bar{\mathcal{H}}_f(\mathbf{x})$. It is easy to see
    that these local minima correspond to generalized centroidal
    Voronoi configurations as in~\cite{JC:08-tac}.
    
Following a similar reasoning as in
Lemma~\ref{lemma:transport_many_particle}
for $F=C_f$ and $F^{h,N} = C_f^{h,N}$, we have that, as $N \rightarrow \infty$, 
the proximal descent scheme \eqref{eq:multi_agent_Voronoi_lloyd_descent} converges to
\eqref{eq:proximal_coverage_scheme}.

From Theorem~\ref{thm:target_dynamics_euclidean}, it follows that
\eqref{eq:proximal_coverage_scheme}
corresponds to the following transport
in~$\mathcal{P}(\Omega)$:
\begin{align*}
	\mu_{k+1} = \arg \min_{\nu \in \mathcal{P}(\Omega)}~ \frac{1}{2\tau} W_2^2(\mu_k, \nu) + \mathbb{E}_{\nu} \left[ \phi_{\mu_k \rightarrow \mu^*}  \right].
\end{align*}
We have that $W_2^2(\cdot, \mu^\star)$ is
strictly (generalized) geodesically convex, 
$l$-smooth w.r.t. the reference measure~$\mu^*$
and its Fr{\'e}chet derivative~$\nabla \phi_{\mu \rightarrow \mu^*}$ is Lipschitz continuous with Lipschitz constant~$\lambda$
(by an application of Propositions~\ref{prop:strictly_convex_C_f},~\ref{prop:l_smooth_C_f} 
and~\ref{prop:lipschitz_frechet_derivative} in Appendix~\ref{app:agg_obj_func}).
Furthermore, for the critical point of the objective function above,
we have:
\begin{align*}
	\frac{1}{\tau} \nabla \phi_{\mu_{k+1} \rightarrow \mu_k} + \nabla \phi_{\mu_k \rightarrow \mu^*} = 0,
\end{align*}
from which it follows that $\mu_{k+1}$ lies on the geodesic from~$\mu_k$ to~$\mu^*$.
From the above and Theorem~\ref{thm:proximal_scheme_convergence_target},
we get that there exists a $\bar{\tau} > 0$ such that for any $\tau \in (0, \bar{\tau})$,
the sequence $\lbrace \mu_k \rbrace_{k \in \mathbb{N}}$ obtained as the transport
of an absolutely continuous probability measure~$\mu_0 \in
\mathcal{P}(\Omega)$ by~\eqref{eq:proximal_coverage_scheme}
converges weakly to $\mu^\star$. \oprocend
\end{proof}
\noindent It is known that the generalized Lloyd's algorithm
results in convergence to generalized centroidal  Voronoi configurations
\cite{JC:08-tac}, where the generators $\lbrace x_1, \ldots, x_N
\rbrace$ of the generalized Voronoi partition are also the centroids
of their respective generalized Voronoi cells.  The generalized centroidal
 Voronoi configuration is, however, not unique, and this
relates to the fact that the convergence is to the local minimizers of
$\bar{\mathcal{H}}_f$, which is typically nonconvex.

\subsection{Numerical experiments}
We now present results from numerical experiments for the coverage
control algorithm~\eqref{eq:multi_agent_Voronoi_lloyd_descent} for the
objective function $\bar{\mathcal{H}}_f$, with $f(x) = x^2$.  We first
sample i.i.d. from a multimodal Gaussian distribution and normalize
the histogram of the samples over a discretization of the spatial
domain to obtain a (quantized) target distribution over the domain.
We then implement the coverage control
algorithm~\eqref{eq:multi_agent_Voronoi_lloyd_descent} for various
sizes~$N$ of the multi-agent system, from random initializations of
the agent positions.
We present the following: (i) The steady state distribution of agents
(in Figure~\ref{fig:agent_dist}). We observe that the distribution of the agents more closely
approximates the target distribution as the size~$N$ of the system increases. 
(ii) The value of the coverage objective function as a function of time (in
Figure~\ref{fig:error_vs_T_N}), for various sizes~$N$ of the
multi-agent system. We observe that the steady state value decreases 
with the size~$N$ of the system, in accordance with our theoretical results.

\begin{figure*}
\begin{center}
        \includegraphics[width=0.24\textwidth]{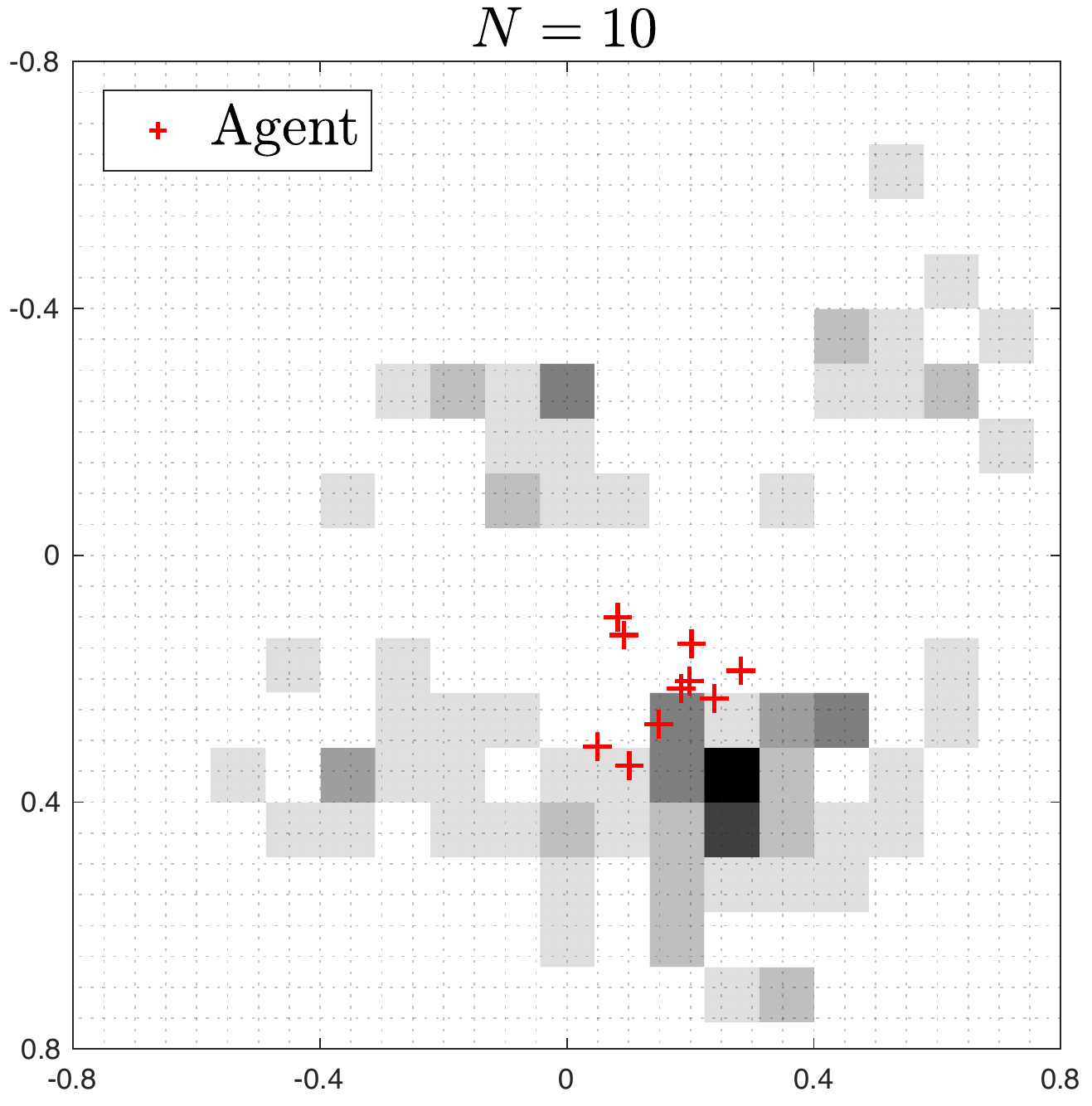}
                \includegraphics[width=0.24\textwidth]{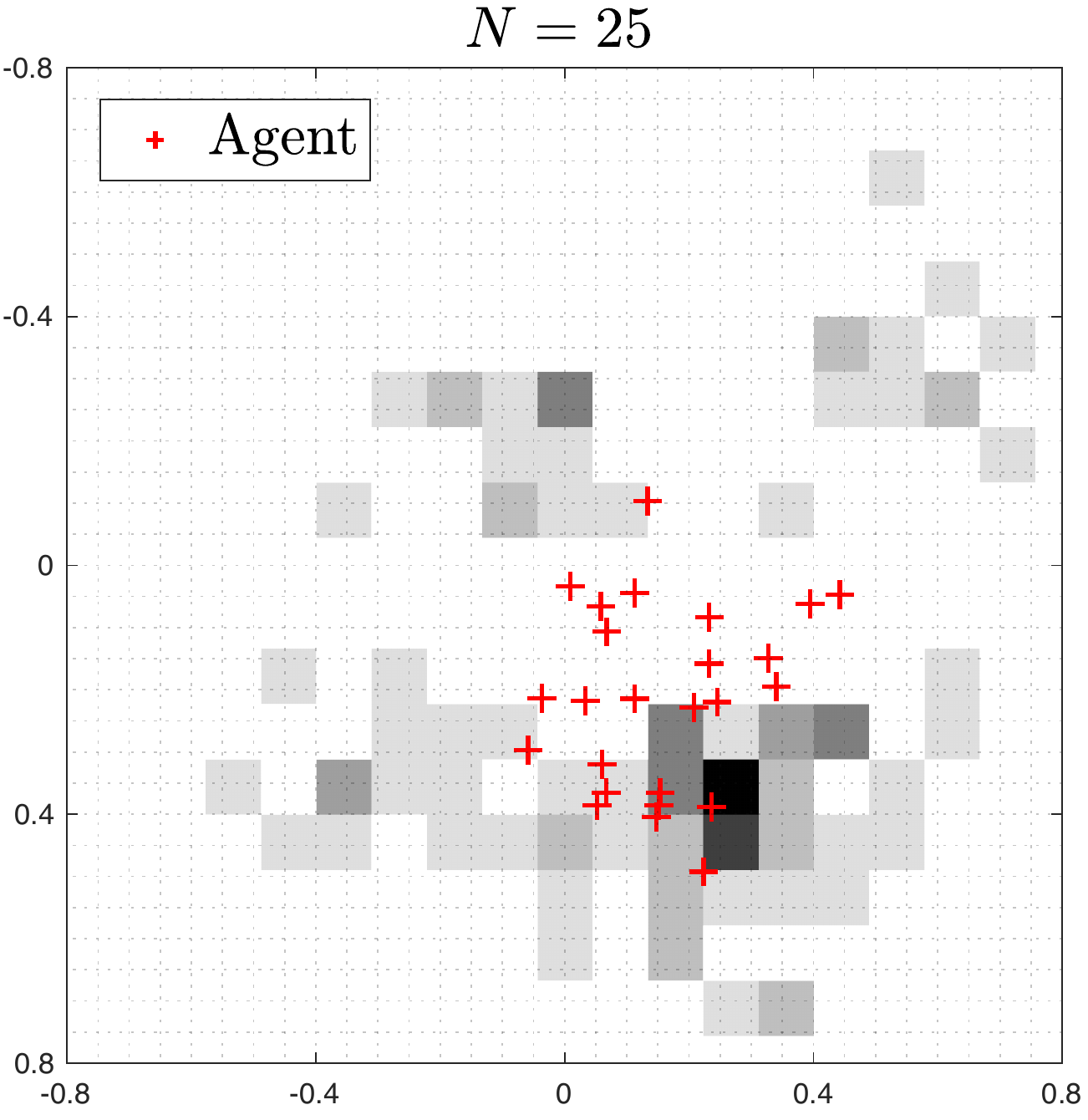}
                        \includegraphics[width=0.24\textwidth]{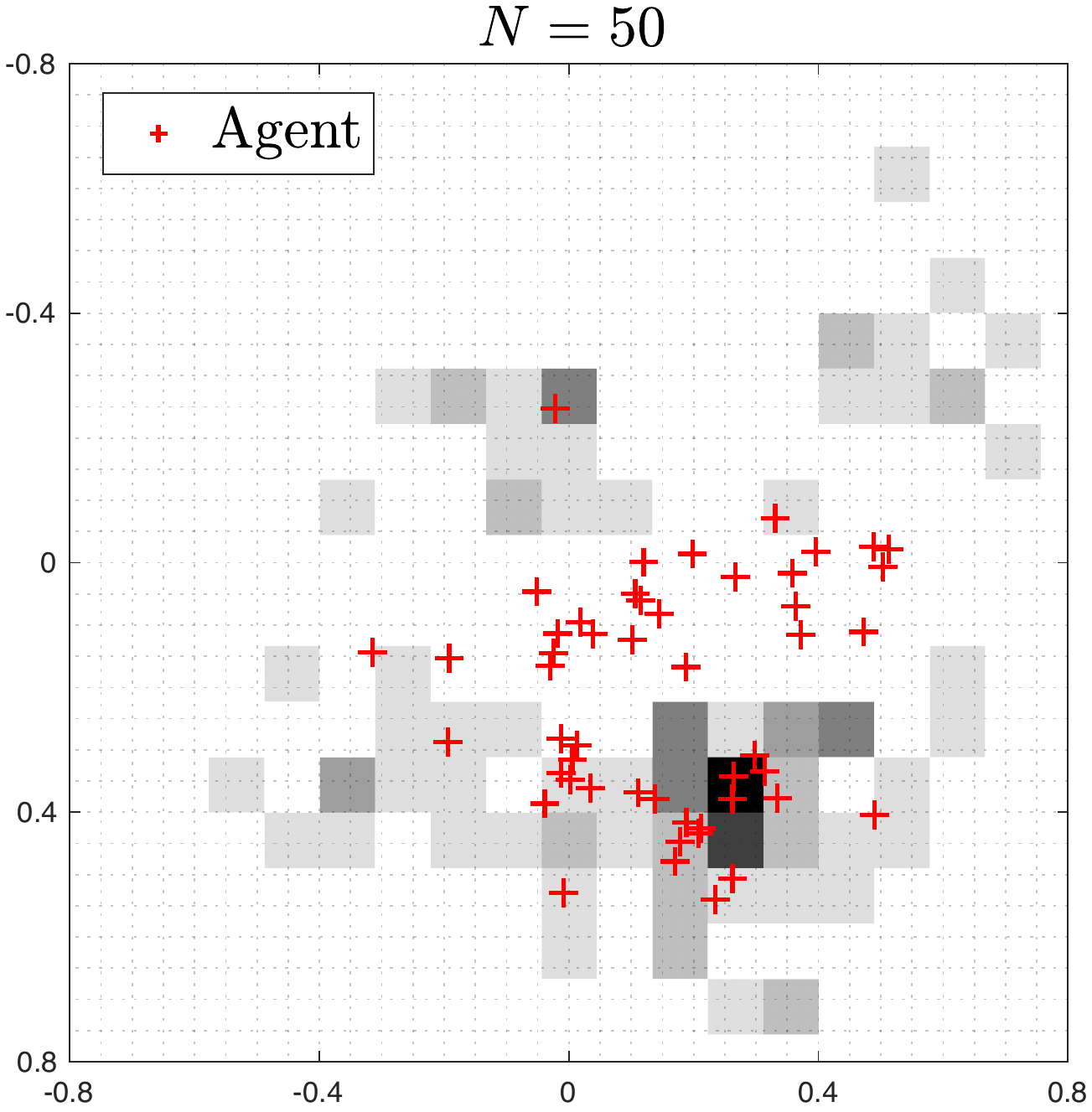}
                                \includegraphics[width=0.24\textwidth]{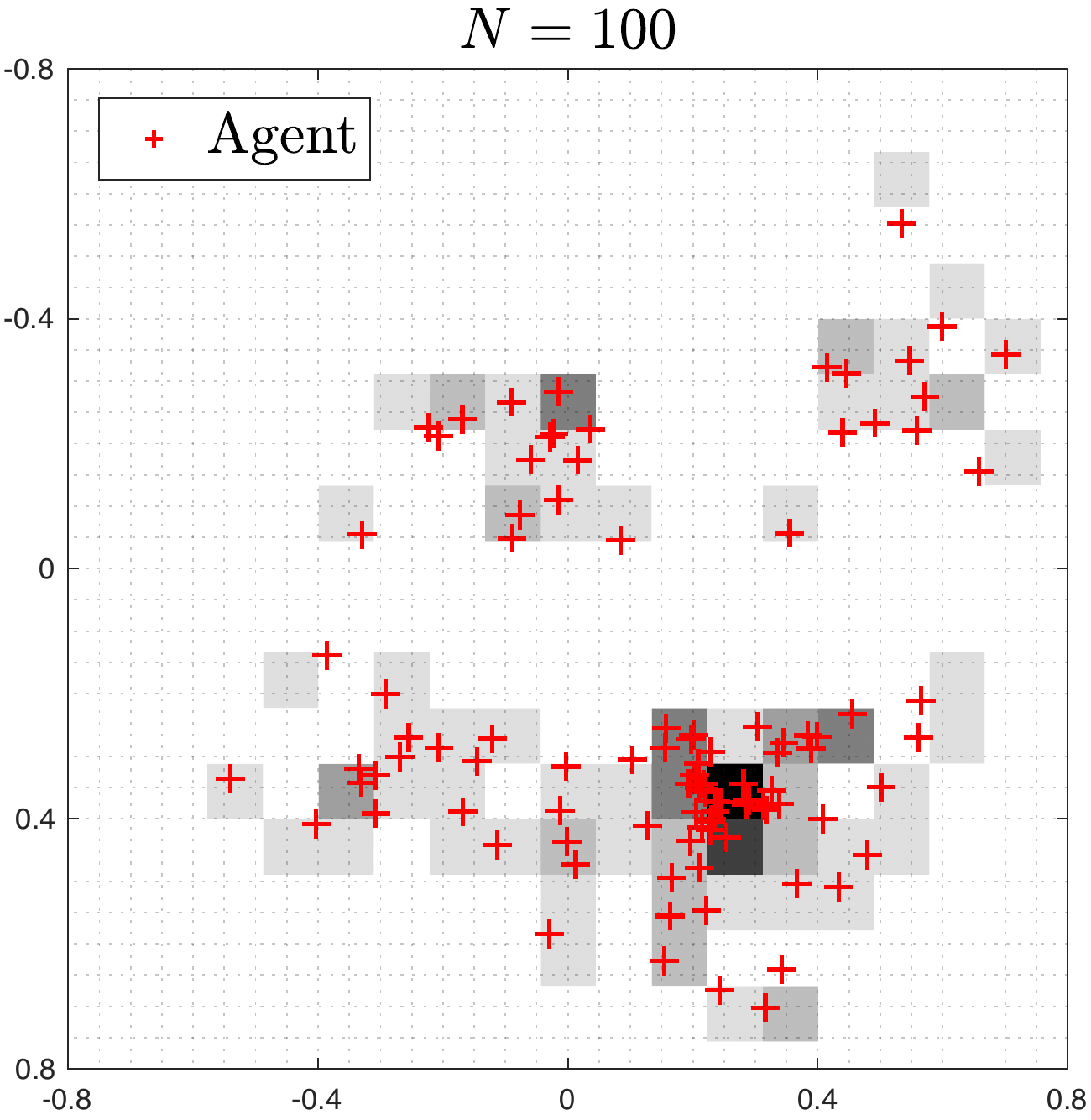}
\end{center}
\caption{The figure shows the steady state distribution of the agents
  implementing the coverage
  algorithm~\eqref{eq:multi_agent_Voronoi_lloyd_descent} with the
  target distribution depicted in grayscale, for $N = 10, 25, 50,
  100$. We observe that the distribution of the agents more closely
  approximates the target distribution as the size~$N$ of the system
  increases.}
		\label{fig:agent_dist}
\end{figure*}
\begin{figure}
        \includegraphics[width=0.48\textwidth]{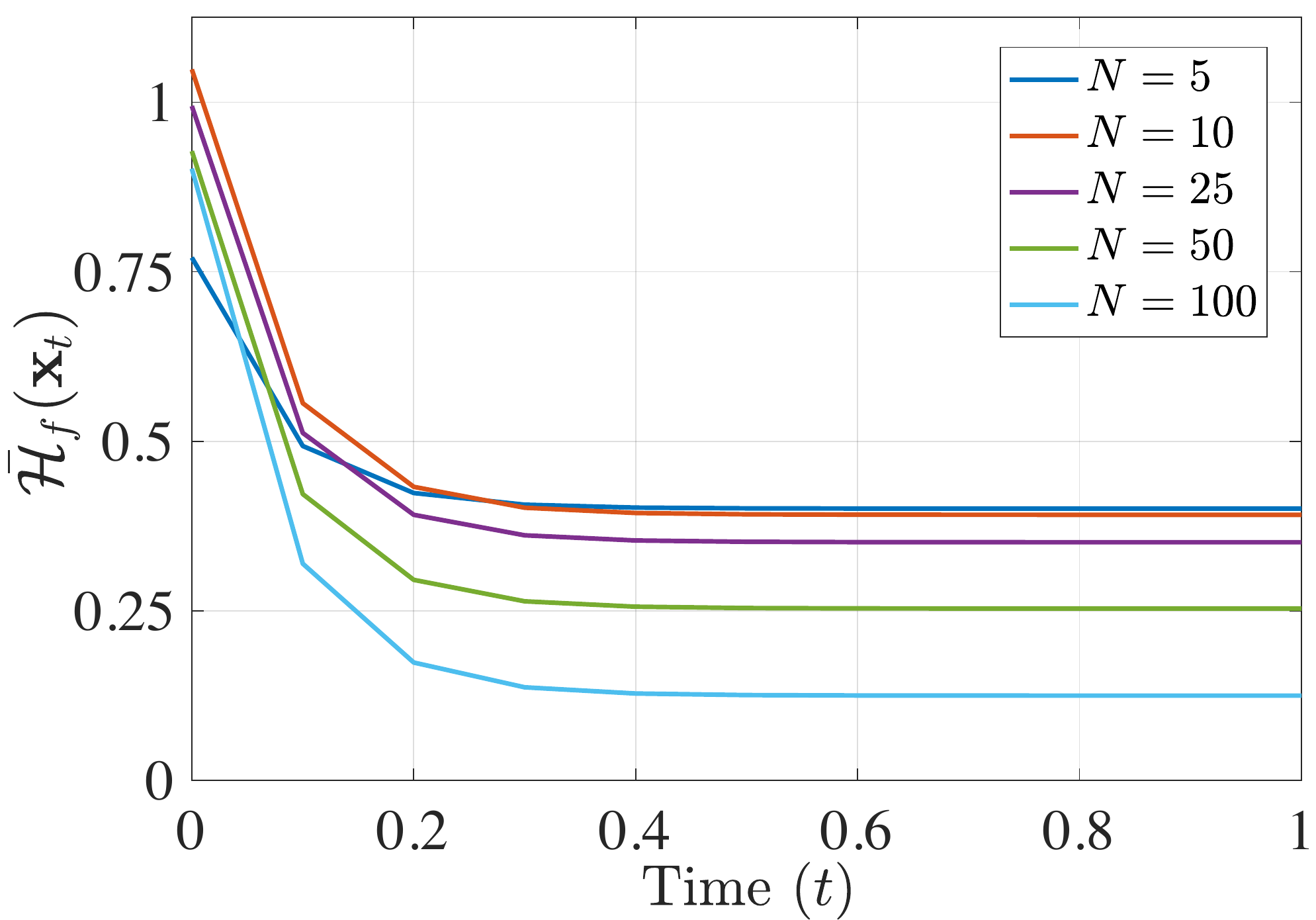}
        \caption{The figure is a representative plot of the value of the aggregate objective function 
        				$\bar{\mathcal{H}}_f(\mathbf{x}_t)$ (with $f(x) = x^2$) vs. time~$t$
        				for various sizes~$N$ of the multi-agent system and random initializations
        				of agent positions. We observe that the steady state value decreases 
        				with the size~$N$ of the system, in accordance with our theoretical results.}
		\label{fig:error_vs_T_N}
\end{figure}


%
\section{Conclusion}
\label{sec:conclusions}
In this paper, we have introduced a multiscale framework for
the analysis and design of 
multi-agent coverage algorithms that begins with 
a macroscopic specification of the target coverage behavior 
to derive provably-correct microscopic, agent-level algorithms 
that achieve the target macroscopic specification. 
Our class of macroscopic proximal descent schemes exploit 
convexity properties of coverage objective functionals to steer
the macroscopic configuration, 
which are then translated into agent-level algorithms via a variational
discretization. 
We uncover the relationship with previously studied coverage
algorithms, and obtain insights into the large-scale behavior of 
these algorithms. 
Future work will consider the extension to a constrained optimization
framework to include such constraints as sensing limitations, 
dynamic and collision-avoidance constraints.
We have assumed in this paper that the underlying spatial domain $\Omega$ is convex. 
This assumption results in the convexity of the space of probability measures. 
The convexity of the space of probability measures (which is the search space in our proximal descent schemes) 
is necessary for the minimization problems we obtain to be convex. 
The presence of obstacles in the mission space is likely to violate this assumption 
and lead to non-convexity of the corresponding minimization problems. 
Overcoming this difficulty is a consideration for future work.

\bibliographystyle{plain}
\bibliography{alias,SMD-add,SM,JC,FB}

\begin{appendices}
 \section{Additional preliminaries}
 We present here the mathematical preliminaries on
 convergence of measures, the $L^2$-Wasserstein space and
 smoothness and convexity notions for functions defined on the 
 $L^2$-Wasserstein space. 

 \subsection{Convexity of functions}
 \label{subsec:functionconvexity}
 Recall that a set $\Omega \subseteq \real^d$ is \textit{convex} if for
 $x,y \in \Omega$, we have $(1 - t) x + t y \in \Omega$ for all $t \in
 [0,1]$.
 A function $f : \Omega \rightarrow \real$ is \textit{convex} if
 $\Omega$ is convex and $f((1 - t) x + t y) \leq (1 - t) f(x) + t f(y)$
 for all $t \in [0,1]$.
 A function $f : \Omega \rightarrow \real$ is \textit{$m$-strongly
   convex} if $\Omega$ is convex and $f((1 - t) x + t y) \leq (1 - t)
 f(x) + t f(y) - \frac{1}{2} m t(1-t) \| x - y \|^2$ for all $t \in
 [0,1]$ and some $m >0$.
 \subsection{The space of probability measures and its 
   topology}
 Let $\Omega = \bar{D}$, with $D \subset \real^d$ an open, bounded set
 in the $d$-dimensional Euclidean space $\real^d$.  Let
 $\mathcal{B}(\Omega)$ be the Borel $\sigma$-algebra in $\Omega$, which
 is the collection of measurable sets w.r.t.~Borel measures.  The space
 of probability measures, $\mathcal{P}(\Omega)$, is the collection of
 functions $\mu:\mathcal{B}(\Omega) \rightarrow [0,1]$ satisfying the
 following properties: (a) $\mu \left( \emptyset \right) = 0$, (b) $\mu
 \left( \Omega \right) = 1$, and (c) (sub-additivity) $\mu\left(
   \cup_{i \in \mathbb{N}} A_i \right) = \sum_{i} \mu(A_i)$, for a
 countable family of pairwise disjoint sets $\{A_i\,|\, A_i \in
 \mathcal{B}(\Omega)\}$.
 %
 %
 We denote by $\mathcal{P}^r(\Omega) \subset \mathcal{P}(\Omega)$ the
 space of atomless probability measures, where a measure $\mu \in
 \mathcal{P}(\Omega)$ is said to be \emph{atomless} if for any $A \in
 \mathcal{B}(\Omega)$ with $\mu(A) > 0$, there exists $B \in
 \mathcal{B}(\Omega)$, $B \subset A$, such that $\mu(A) > \mu(B) >
 0$. It follows that for an atomless measure $\mu$, we will have
 $\mu(\lbrace x \rbrace) = 0$ for all $x \in \Omega$. We consider this
 a notion of regularity of probability measures, and hence the use of
 the superscript $r$ in $\mathcal{P}^r(\Omega)$. We refer to~\cite{PB:13} 
 for other basic definitions in measure theory. Finally, we recall the following:

 \begin{definition}[\bf \emph{Pushforward measure}]
   Let $\Omega, \Theta $ be (Borel) measurable spaces, a measurable
   mapping~$\mathcal{T} : \Omega \rightarrow \Theta$ and consider a measure~$\mu
   \in \mathcal{P}(\Omega)$. The \emph{pushforward measure}~$\nu =
   \mathcal{T}_{\#} \mu \in \mathcal{P}(\Theta)$ of~$\mu$ is defined as
   $\nu(B) = \mathcal{T}_{\#} \mu (B) = \mu (\mathcal{T}^{-1}(B))$, for
   all Borel measurable $B \subseteq \Theta$.
 \end{definition}

 \subsection{Weak convergence of measures}
 The results of this manuscript rely on the notions of weak convergence
 in $\mathcal{P}(\Omega)$, the topology of weak convergence, its
 metrizability, and the compactness of sets of
 $\mathcal{P}(\Omega)$. We recall them here and refer the reader to~\cite{PB:13}
 for more information.
 \begin{definition}[\bf \emph{Weak convergence}]
   Let $\Omega \subseteq \real^d$, and $\mathcal{P}(\Omega)$ be its set
   of probability measures. A sequence $\lbrace \mu_k \rbrace_{k \in
     \mathbb{N}} \subseteq \mathcal{P}(\Omega)$ \emph{converges weakly}
   to $\mu \in \mathcal{P}(\Omega)$ if for any bounded and continuous
   function $f$ on $\Omega$, $\lim_{k \rightarrow \infty}
   \int_{\Omega} f d\mu_k = \int_{\Omega} f d\mu$.
 \end{definition}
 Equivalently, in the definition above, the sequence $\lbrace \mu_k
 \rbrace_{k \in \mathbb{N}}$ in $\mathcal{P}(\Omega)$ is said to
 \emph{converge to $\mu$ in $\mathcal{P}(\Omega)$ equipped with the
   topology of weak convergence}.  The
 space of probability measures $\mathcal{P}(\Omega)$ equipped with the
 topology of weak convergence is \emph{metrizable}~\cite{PB:13}. In
 other words, there exists a metric on $\mathcal{P}(\Omega)$ such that
 the topology of weak convergence is obtained as the topology induced
 by the metric. One such metric is the Wasserstein distance, see
 Section~\ref{sec:W2}.
 We now state Prokhorov's theorem~\cite{PB:13} on the equivalence between
 tightness and precompactness of a collection of probability
 measures over a separable and complete metric (Polish) space.
 \begin{lemma}[\bf \emph{Prokhorov's theorem}]
 \label{lemma:Prokhorov}
 Let $\Omega$ be a complete metric space, and let $\mathcal{K}
 \subseteq \mathcal{P}(\Omega)$. 
 The closure of $\mathcal{K}$ w.r.t.~the topology of weak convergence
 in $\mathcal{P}(\Omega)$ is \emph{compact} if and only if
 $\mathcal{K}$ is tight. That is, $\mathcal{K}$ is \emph{tight} if
 for any $\epsilon > 0$ there exists a compact $K_{\epsilon} \subseteq
 \Omega$ such that $\mu(K_{\epsilon}) > 1-\epsilon$, for all $\mu \in
 \mathcal{K}$.
 \end{lemma}
 \begin{corollary}[\bf \emph{Compactness of $\mathcal{P}(\Omega)$}]
 \label{corollary:compactness_P_Omega}
 Let $\Omega \subseteq \real^d$ a compact set. Then, the closure of
 $\mathcal{P}(\Omega)$ w.r.t.~the topology of weak convergence in
 $\mathcal{P}(\Omega)$ is compact. This follows from
 Prokhorov's theorem in Lemma~\ref{lemma:Prokhorov}, since
 $\mathcal{P}(\Omega)$ is tight: for any $\epsilon > 0$, we choose
 $\Omega$ itself as the compact set and have $\mu(\Omega) = 1 >
 1-\epsilon$ for any $\mu \in \mathcal{P}(\Omega)$. Moreover, since
 $\mathcal{P}(\Omega)$ is also closed w.r.t.~the topology of weak
 convergence, it is therefore compact.
 \end{corollary}
 \subsection{The $L^2$-Wasserstein distance}\label{sec:W2}
 The $L^2$-Wasserstein distance between two probability 
 measures $\mu, \nu \in \mathcal{P}(\Omega)$
 is given by:
 \begin{align}
   W_2^2(\mu, \nu) = \min_{\pi \in \Pi(\mu, \nu)} \int_{\Omega \times
     \Omega} |x-y|^2 ~d\pi(x,y),
 	\label{eq:Kantorovich_Formulation_W2}
 \end{align}
 where $\Pi(\mu, \nu)$ is the space of joint probability measures over
 $\Omega \times \Omega$ with marginals $\mu$ and $\nu$. The definition
 of $L^2$-Wasserstein distance in~\eqref{eq:Kantorovich_Formulation_W2}
 follows from the so-called Kantorovich formulation of optimal
 transport. An alternative formulation of this problem, called the
 Monge formulation of optimal transport, is given below:
 \begin{align}
   W_2^2(\mu, \nu) = \min_{\substack{T: \Omega \rightarrow \Omega \\
       T_{\#}\mu = \nu }} \int_{\Omega} |x - T(x)|^2 ~d\mu(x).
   \label{eq:Monge_Formulation_W2}
 \end{align}
 In the Monge formulation~\eqref{eq:Monge_Formulation_W2}, the
 minimization is carried out over the space of maps $T: \Omega
 \rightarrow \Omega$ for which the probability measure $\nu$ is
 obtained as the pushforward of $\mu$.  This can be viewed as a
 deterministic formulation of optimal transport, where the transport is
 carried out by a map, whereas the Kantorovich formulation
 \eqref{eq:Kantorovich_Formulation_W2} can be seen as a problem
 relaxation, where the transport plan is described by a joint
 probability measure $\pi$ over $\Omega \times \Omega$, with $\mu$ and
 $\nu$ as its marginals.  It is to be noted that the Monge formulation
 does not always admit a solution, while the Kantorovich problem
 does.
 Roughly speaking, the Kantorovich formulation is the ``minimal''
 extension of the Monge formulation, as both problems attain the same
 infimum~\cite{FS:15}. Further, the two formulations
 \eqref{eq:Kantorovich_Formulation_W2} and
 \eqref{eq:Monge_Formulation_W2} are equivalent under certain
 conditions~\cite{FS:15}.
 %
 The space of probability measures
   $\mathcal{P}(\Omega)$ endowed with the $L^2$-Wasserstein distance
   $W_2$ will equivalently be referred to as the $L^2$-Wasserstein
   space $(\mathcal{P}(\Omega), W_2)$ over $\Omega$. The following
 lemma, which follows from Theorem~6.9 in~\cite{CV:08}, establishes the
 equivalence between convergence in the sense of the topology of weak
 convergence and in the $L^2$-Wasserstein metric.
 \begin{lemma}[\bf \emph{Convergence in $(\mathcal{P}(\Omega), W_2)$}]
 \label{lemma:weak_vs_wasserstein_convergence}
 For compact $\Omega \subset \real^d$, the
 $L^2$-Wasserstein distance $W_2$ metrizes the weak convergence in
 $\mathcal{P}(\Omega)$. That is, a sequence of measures $\lbrace \mu_k
 \rbrace_{k \in \mathbb{N}}$ in $\mathcal{P}(\Omega)$ converges weakly
 to $\mu \in \mathcal{P}(\Omega)$ if and only if $\lim_{k \rightarrow
   \infty} W_2(\mu_k, \mu) = 0$.
 \end{lemma}
\section{Fr\'echet differentials of functionals on atomless measures} 
Let~$\mu_0, \mu_1 \in \mathcal{P}^{\rm a}(\Omega)$ be
atomless probability measures, and let $T_{\mu_0 \rightarrow \mu_1}$ 
be the optimal transport map from~$\mu_0$ to~$\mu_1$.
Furthermore, for $\epsilon \in [0,1]$, let: 
\begin{align} \label{eq:gen_geodesic}
	\mu_\epsilon = ((1-\epsilon) \id + \epsilon T_{\mu_0 \rightarrow \mu_1})_{\#} \mu_0.
\end{align}

We now begin by introducing the notion of first variation of a
functional on $\mathcal{P}(\Omega)$ as follows:
\begin{definition}[\bf \emph{First variation of a functional on $\mathcal{P}(\Omega)$}]
  Let $F: \mathcal{P}(\Omega) \rightarrow \real$ and $\mu_0 \in \mathcal{P}^{\rm a}(\Omega)$.
  Suppose that there exists a unique $\varphi$ such that for any $\mu_1 \in \mathcal{P}^{\rm a}(\Omega)$
  and $\lbrace \mu_{\epsilon} \rbrace_{\epsilon \in [0,1]}$ as defined in~\eqref{eq:gen_geodesic}, 
  the following holds: 
    \begin{align*}
    		\left. \frac{d}{d\epsilon}
    F(\mu_{\epsilon}) \right|_{\epsilon = 0} = \lim_{\epsilon
    \rightarrow 0^{+}} \frac{1}{\epsilon} \int_{\Omega} \varphi \left(d\mu_{\epsilon} - d\mu_0 \right).
    \end{align*}
   Then $\varphi$ is the first variation of $F$ evaluated at $\mu_0$,
   denoted as $\varphi = \frac{\delta F}{\delta \mu}(\mu_0)$.
\end{definition}
For functionals for which the first variation exists as in the above
definition, we can introduce the notion of Fr{\'e}chet
derivative on the $L^2$-Wasserstein
space~$(\mathcal{P}(\Omega), W_2)$:
\begin{definition}[\bf \emph{Derivative of a functional on $(\mathcal{P}(\Omega), W_2)$}]
  A functional~$F : \mathcal{P}(\Omega) \rightarrow \real$ is
  Fr{\'e}chet differentiable with derivative~$\xi \in L^2((\Omega, \mu); \real^d)$
  at an atomless measure $\mu_0 \in \mathcal{P}^{\rm a}(\Omega)$ 
  if for any $\mu_1 \in \mathcal{P}^{\rm a}(\Omega)$
  and $\lbrace \mu_{\epsilon}\rbrace_{\epsilon \in [0,1]}$ as defined in~\eqref{eq:gen_geodesic},
  the following holds:
\begin{align*}
  \lim_{\epsilon \rightarrow 0^{+}} \frac{F(\mu_{\epsilon}) - F(\mu_0) 
  		- \int_{\Omega} \left\langle \xi,   T_{\mu_0 \rightarrow \mu_\epsilon} - \id	\right\rangle d\mu_0}{W_2(\mu_0, \mu_1)} = 0,
\end{align*}
where~$\xi = \nabla \varphi$ and $\varphi = \frac{\delta F}{\delta
  \tilde{\mu}} (\mu_0)$.
\end{definition}
Furthermore, we define the directional derivative of~$F$ at~$\mu_0$
along a tangent vector field~$\mathbf{v} \in L^2((\Omega,\mu);\real^d)$ as:
\begin{align*}
  D_{\mathbf{v}}F(\mu_0) = \int_{\Omega}
  \left \langle \xi, \mathbf{v} \right \rangle d\mu,
\end{align*}
where $\xi$ is the Fr\'echet derivative of $F$ evaluated at~$\mu_0$.  
\section{Results on regularity of functionals}
The following lemma can be verfied for strongly geodesically 
convex functionals as in Definition~\ref{defn:strong_geodesic_convexity}:
\begin{lemma}[\bf \emph{Strongly geodesically convex functionals}]
\label{lemma:strongly_convex_functional}
Let $F: \mathcal{P}(\Omega) \rightarrow \real$ be an
$m$-strongly (generalized) geodesically convex functional on $(\mathcal{P}(\Omega),W_2)$
w.r.t base measure~$\theta \in \mathcal{P}^{\rm a}(\Omega)$.
Let $\mu , \nu \in \mathcal{P}^{\rm a}(\Omega)$ be
atomless probability measures and let $\xi_{\mu}$, 
$\xi_{\nu}$ and $\xi_{\theta}$ be the Fr{\'e}chet derivatives of~$F$ 
evaluated at~$\mu$,~$\nu$ and~$\theta$ respectively.
The following holds:
\begin{align*}
  &\int_{\Omega} \left \langle \xi_{\nu}(T_{\theta
      \rightarrow \nu}) - \xi_{\mu}(T_{\theta \rightarrow \mu}) , T_{\theta
      \rightarrow \nu} - T_{\theta \rightarrow \mu} \right \rangle
  d\theta \\ &\geq m \int_{\Omega} \left| T_{\theta \rightarrow \nu} - T_{\theta \rightarrow \mu} \right|^2 d\theta,
\end{align*}
where $T_{\theta \rightarrow \mu}: \Omega \rightarrow \Omega$ and
$T_{\theta \rightarrow \nu}: \Omega \rightarrow \Omega$ are optimal
transport maps from $\theta$ to $\mu$ and from $\theta$ to $\nu$
respectively.
\end{lemma}
Similarly, the following lemma can be verfied for $l$-smooth functionals
as defined in~\ref{defn:l_smooth_functional}:
\begin{lemma}[\bf \emph{$l$-smooth functionals}]
\label{lemma:l_smooth_functional}
Let $F: \mathcal{P}(\Omega) \rightarrow \real$ be an
$l$-smooth functional on $(\mathcal{P}(\Omega),W_2)$
w.r.t. a base measure~$\theta \in \mathcal{P}^{\rm a}(\Omega)$.
Let $\mu , \nu \in \mathcal{P}^{\rm a}(\Omega)$ be
atomless probability measures and let $\xi_{\mu}$, 
$\xi_{\nu}$ and $\xi_{\theta}$ be the Fr{\'e}chet derivatives of~$F$ 
evaluated at~$\mu$,~$\nu$ and~$\theta$ respectively.
The following holds:
\begin{align*}
	&\left| \int_{\Omega} \left \langle \xi_{\nu}( T_{\theta \rightarrow \nu} ) - \xi_{\mu}( T_{\theta \rightarrow \mu} ),
      T_{\theta \rightarrow \nu} - T_{\theta \rightarrow \mu} \right \rangle d\theta \right|  \\
    &\leq l \int_{\Omega} \left| T_{\theta \rightarrow \nu} - T_{\theta \rightarrow \mu} \right|^2 d\theta.
\end{align*}
\end{lemma}
\section{Supporting results for
  Theorem~\ref{thm:conv_prox_recursion_wasserstein}}

\begin{lemma}[\bf \emph{Compactness and convexity of sublevel sets}]
\label{lemma:compactness_sublevel_F}
Let $F$ satisfy the regularity conditions of
Assumption~\ref{ass:regularity}.
Then, the $F$-sublevel set of any absolutely continuous probability
measure $\mu \in \mathcal{P}(\Omega)$ is compact and geodesically
convex in the $L^2$-Wasserstein space $(\mathcal{P}(\Omega), W_2)$.
\end{lemma}
\begin{proof}
  For any $\mu \in \mathcal{P}(\Omega)$, the sublevel set
  $\mathcal{S}({\mu}) = \lbrace \nu \in \mathcal{P}(\Omega) | F(\nu)
  \leq F(\mu) \rbrace$ is closed in $(\mathcal{P}(\Omega), W_2)$,
  since $F$ is continuous and $\mathcal{P}(\Omega)$ is closed and
  compact. 
  This implies that $\mathcal{S}({\mu})$ is
  also compact since it is a closed subset of a compact set.

  It holds 
  that $(\mathcal{P}(\Omega), W_2)$ is geodesically convex, and
  consider, for any $\nu_0, \nu_1 \in \mathcal{S}({\mu})$, and $\nu_t
  \in \mathcal{P}(\Omega)$, for $t \in [0,1]$, the generalized
  geodesic between $\nu_0$ to $\nu_1$ with $\mu$ as the reference
  measure\footnote{From~\cite[Theorem~1.17]{FS:15} 
  it follows that unique optimal transport maps from $\mu$ to $\nu_0$
  and $\mu$ to $\nu_1$ exist, since $\mu$ is absolutely continuous,
  and therefore so does a unique generalized geodesic in
  $(\mathcal{P}(\Omega), W_2)$ between $\nu_0$ and $\nu_1$ as in
  Definition~\ref{defn:generalized_displacement_interpolation}.}.  From
  the (generalized) geodesic convexity of $F$ we have that $F(\nu_t)
  \leq (1-t) F(\nu_0) + t F(\nu_1) \leq F(\mu)$ (since $F(\nu_0) \leq
  F(\mu)$ and $F(\nu_1) \leq F(\mu)$ by definition of
  $\mathcal{S}({\mu})$).  This implies that $\nu_t \in
  \mathcal{S}({\mu})$ for any $t \in [0,1]$, from which we infer the
  geodesic convexity of $\mathcal{S}({\mu})$.
%

\end{proof}

\begin{lemma}[\bf \emph{Strong convexity of objective functional}]
\label{lemma:strongly_convex_PGD}
Let $F$ satisfy the regularity conditions of Assumption~\ref{ass:regularity}.
For any absolutely continuous probability measure $\mu \in
\mathcal{P}(\Omega)$, the functional $G(\nu) = \frac{1}{2\tau}
W_2^2(\mu, \nu) + F(\nu)$ is $\left( \frac{1}{\tau} - l
\right)$-strongly (generalized) geodesically convex (in the sense of
Definition~\ref{defn:strong_geodesic_convexity} w.r.t. reference measure~$\mu$) 
over $\supscr{\mathcal{P}}{a}(\Omega)$ for $0< \tau < 1/l$.
\end{lemma}

\begin{proof}
Since $F$ is $l$-smooth w.r.t. any (atomless) base measure, 
applying Lemma~\ref{lemma:l_smooth_functional}
for two atomless measures~$\nu_1$ and~$\nu_2$, we get:
\begin{align}
\begin{aligned}
 & \left| \int_{\Omega} \left \langle \xi_2 (T_{\mu \rightarrow \nu_2}) - \xi_1(T_{\mu \rightarrow \nu_1}), 
  									T_{\mu \rightarrow \nu_2} - T_{\mu \rightarrow \nu_1} \right \rangle d\mu \right|  \\
  				&\leq l \int_{\Omega} \left| T_{\mu \rightarrow \nu_2} - T_{\mu \rightarrow \nu_1} \right|^2  d\mu,
\end{aligned}
\label{eq:l_smooth_lemma_strongly_convex}
\end{align}
where~$\xi_1$ and~$\xi_2$ are the Fr{\'e}chet derivatives of $F$
evaluated at~$\nu_1$ and~$\nu_2$, respectively, and~$T_{\mu \rightarrow \nu_1}$
and $T_{\mu \rightarrow \nu_2}$ are the optimal transport maps from $\mu$ to $\nu_1$
and~$\nu_2$, respectively.
Let $\eta_i = \left. \nabla \left( \frac{\delta G}{\delta \nu} \right)
\right|_{\nu_i}$, for $i = 1,2$, and let $\phi_i = \frac{1}{2}
\left. \frac{\delta W_2^2(\mu, \nu)}{\delta \nu} \right|_{\nu_i}$ be
the so-called Kantorovich potential for the transport
from $\nu_1$ to $\mu$, for $i = 1,2$.  We now have:
\begin{align*}
  &\int_{\Omega} \left \langle \eta_2 (T_{\mu \rightarrow \nu_2}) - \eta_1 (T_{\mu \rightarrow \nu_1}), T_{\mu \rightarrow \nu_2} - T_{\mu \rightarrow \nu_1} \right \rangle d\mu \\
  &= \int_{\Omega} \left \langle \frac{1}{\tau} \nabla \phi_2 (T_{\mu \rightarrow \nu_2}) - \frac{1}{\tau} \nabla \phi_1 (T_{\mu \rightarrow \nu_1}) \right. \\
  						&\qquad \left. - \xi_1 (T_{\mu \rightarrow \nu_1}) + \xi_2 (T_{\mu \rightarrow \nu_2}), 
  										T_{\mu \rightarrow \nu_2} - T_{\mu \rightarrow \nu_1} \right \rangle d\mu \\
  &= \frac{1}{\tau} \int_{\Omega} \left \langle \nabla \phi_2(T_{\mu \rightarrow \nu_2}) - \nabla \phi_1 (T_{\mu \rightarrow \nu_1}), 
  										T_{\mu \rightarrow \nu_2} - T_{\mu \rightarrow \nu_1} \right \rangle d\mu \\
  & \qquad + \int_{\Omega} \left \langle \xi_2 (T_{\mu \rightarrow \nu_2}) - \xi_1 (T_{\mu \rightarrow \nu_1}), T_{\mu \rightarrow \nu_2} - T_{\mu \rightarrow \nu_1} \right \rangle d\mu  \\
  &\geq \left( \frac{1}{\tau} - l \right) \int_{\Omega} \left| T_{\mu \rightarrow \nu_2} - T_{\mu \rightarrow \nu_1}  \right|^2 d\mu ,
\end{align*} where the penultimate inequality above follows
from~\eqref{eq:l_smooth_lemma_strongly_convex}. 
\begin{sloppypar}
  We have also used the fact that $\int_{\Omega} \left \langle \nabla
    \phi_2(T_{\mu \rightarrow \nu_2}) - \nabla \phi_1(T_{\mu \rightarrow \nu_1}), T_{\mu \rightarrow \nu_2} - T_{\mu
      \rightarrow \nu_1} \right \rangle d\mu = \int_{\Omega} \left|
    T_{\mu \rightarrow \nu_2} - T_{\mu \rightarrow \nu_1} \right|^2
  d\mu$ (this follows from
    an application of~\cite[Theorem 1.17]{FS:15}).
 Since $\tau < \frac{1}{l}$, we
get that the functional $G$ is strongly convex with
parameter $\frac{1}{\tau} - l$. 
\end{sloppypar}
\end{proof}

\section{Proof of Proposition~\ref{prop:aggregate_objective_OT_cost}} \label{app:agg_obj_OT_proof}
  Let~$\widehat{\mu}^N_{\mathbf{x},\mathbf{w}} = \sum_{i=1}^N w_i
  \delta_{x_i}$ be a weighted discrete probability measure
  corresponding to~$\left \lbrace x_i \right \rbrace_{i=1}^N$ with
  weights~$\lbrace w_i \rbrace_{i=1}^N$, such that $w_i \in [0,1]$ and
  $\sum_{i=1}^N w_i = 1$.  The optimal transport cost
  between~$\widehat{\mu}^N_{\mathbf{x},\mathbf{w}}$ and~$\mu^\star$ is
  given by:
\begin{align*}
  C_f(\widehat{\mu}^N_{\mathbf{x},\mathbf{w}}, \mu^\star) =
  \inf_{\substack{T: \Omega \rightarrow \Omega \\ T_{\#}\mu^\star =
      \widehat{\mu}^N_{\mathbf{x}} }}~ \int_{\Omega} f(| x - T(x)
  |)~d\mu^\star(x),
\end{align*}
where the infimum is over the set of maps~$T$ that
pushforward~$\mu^\star$ to~$\widehat{\mu}^N_{\mathbf{x},\mathbf{w}}$
(since~$\widehat{\mu}^N_{\mathbf{x},\mathbf{w}}$ has finite support,
pushforward maps exist only from~$\mu^\star$
to~$\widehat{\mu}^N_{\mathbf{x},\mathbf{w}}$ and not the other way
around). Transport maps~$T : \Omega \rightarrow \lbrace x_i
\rbrace_{i=1}^N$ partition~$\Omega$ into~$N$ regions $\left \lbrace
  \mathcal{W}_i \right \rbrace_{i=1}^N$,  where
$\mathcal{W}_i = \{ x \in \Omega \,|\, T(x) = x_i \}$, of
mass~$\mu^\star(\mathcal{W}_i) = w_i$. Let~$T^\star : \Omega
\rightarrow \lbrace x_i \rbrace_{i=1}^N$ be the optimal transport map
from $\mu^\star$ to~$\widehat{\mu}^N_{\mathbf{x},\mathbf{w}}$, which
allows us to write:
\begin{align}	
          &C_f(\widehat{\mu}^N_{\mathbf{x},\mathbf{w}}, \mu^\star)
           =  \inf_{\substack{T: \Omega \rightarrow \Omega \nonumber \\ T_{\#}\mu^\star = \widehat{\mu}^N_{\mathbf{x},\mathbf{w}} }}~ \int_{\Omega} f(| x - T(x) |)~d\mu^\star(x)  \\
           &
          =	\int_{\Omega} f(| x - T^\star(x) |)~d\mu^\star(x) 
          = \sum_{i=1}^N \int_{\mathcal{W}_i^\star} f(|x - x_i|)
          ~d\mu^\star(x)
          \nonumber  \\
           & \qquad \geq  \sum_{i=1}^N \int_{\mathcal{V}_i} f(|x - x_i|)
          ~d\mu^\star(x) .
	\label{eq:wasserstein_geq_inequality}
      \end{align} The above inequality, which holds for any choice of
      $\mathbf{w}$, follows from the fact that $f$ is non-decreasing, and
      the definition of the Voronoi partition $\{\mathcal{V}_i\}_{i =
        1}^N$.   As $f$ is non-decreasing:
\begin{align*}
  \int_{\Omega} \hspace*{-0.5ex}\min_{i \in \lbrace 1, \ldots, N \rbrace} f(| x - x_i
  |)d\mu^\star(x) = \hspace*{-0.3ex}\sum_{i=1}^N \int_{\mathcal{V}_i} f(| x - x_i
  |) d\mu^\star(x),
\end{align*}
where~$\lbrace \mathcal{V}_i \rbrace_{i=1}^N$ 
is the Voronoi partition of~$\Omega$. We now define a map~$T_{\mathcal{V}} :
\Omega \rightarrow \Omega$ such that $T_{\mathcal{V}}(x) = x_i$ for~$x
\in \mathcal{V}_i$, with~$T_{\mathcal{V}}(\Omega) = \lbrace x_1,
\ldots, x_N \rbrace$, for which the following holds:
\begin{align*}
  \int_{\Omega} &f(|x - T_{\mathcal{V}}(x)|)~d\mu^\star(x) 
  = \sum_{i=1}^N \int_{\mathcal{V}_i} f(| x - x_i |)~d\mu^\star(x) \\
  &= \int_{\Omega} \min_{i \in \lbrace 1, \ldots, N \rbrace} f(| x - x_i |)~d\mu^\star(x).
\end{align*}
From~\eqref{eq:wasserstein_geq_inequality} and the above, we therefore
get:
\begin{align*}
  \int_{\Omega} f(|x - T_{\mathcal{V}}(x)|)~d\mu^\star(x) \leq
  C_f(\widehat{\mu}^N_{\mathbf{x},\mathbf{w}}, \mu^\star).
\end{align*}
%
For the particular choice of the weights~$w_i^\star = \mu^\star(\mathcal{V}_i)$, 
such that~$\widehat{\mu}^N_{\mathbf{x},\mu^\star(\mathcal{V})} = \sum_{i=1}^N \mu^\star(\mathcal{V}_i) \delta_{x_i}$,
we also get the inequality:
%
\begin{align*}
  & C_f \left(\sum_{i=1}^N \mu^\star(\mathcal{V}_i) \delta_{x_i} , \mu^\star \right)  \\
  &= \inf_{\substack{T: \Omega \rightarrow \Omega \\  T_{\#}\mu^\star = \sum_{i=1}^N \mu^\star(\mathcal{V}_i) \delta_{x_i} }}~ \int_{\Omega} f(| x - T(x) |)~d\mu^\star(x)  \\
  &\leq \int_{\Omega} f(|x - T_{\mathcal{V}}(x)|)~d\mu^\star(x),
\end{align*}
and we therefore get:
\begin{align*}
  C_f \left( \sum_{i=1}^N \mu^\star(\mathcal{V}_i) \delta_{x_i} , \mu^\star
  \right) = \int_{\Omega} \min_{ i \in \lbrace 1, \ldots, N \rbrace }
  f(| x - x_i |)~d\mu^\star(x),
\end{align*}
which establishes that:
\begin{align*}
  \min_{\substack{\mathbf{w} \in \realnonnegative^N}} C_f \left( \sum_{i=1}^N w_i \delta_{x_i}, \mu^\star \right) 
  &= \int_{\Omega} \min_{ i \in \lbrace 1, \ldots, N \rbrace } f(| x - x_i |)~d\mu^\star(x) \\
  &= \mathcal{H}_f(\mathbf{x}),
\end{align*}
with the minimizing weights $w_i^\star = \mu^\star(\mathcal{V}_i)$.

\section{Aggregate objective functions} \label{app:agg_obj_func}
\begin{proposition}[\bf \emph{Strict geodesic convexity of $C_f(\cdot, \mu^*)$}]
\label{prop:strictly_convex_C_f}
  Fix $\mu^* \in \mathcal{P}(\Omega)$ (absolutely continuous) as the
  reference measure and let $\mu_0, \mu_1 \in \mathcal{P}(\Omega)$.
  Let $T_{\mu^* \rightarrow \mu_0}$ and $T_{\mu^* \rightarrow \mu_1}$
  be optimal transport maps from $\mu^*$ to $\mu_0$ and $\mu^*$ to
  $\mu_1$ respectively, corresponding to the optimal transport cost
  $C_f$, and let $T_t = (1-t) T_{\mu^* \rightarrow \mu_0} + t T_{\mu^*
    \rightarrow \mu_1}$ for $t \in [0,1]$. For $\mu_t = {T_t}_{\#}
  \mu^*$, we have:
\begin{align*}
  C_f(\mu_t, \mu^* ) < (1-t) C_f(\mu_0, \mu^*) + t C_f(\mu_1, \mu^*).
\end{align*}
\end{proposition}
\begin{proof}
We have:
\begin{align*}
  &C_f(\mu_t, \mu^* ) \leq \int_{\Omega} f(|T_t(x) - x|) d\mu^*(x) \\
  &= \int_{\Omega} f\left( \left| (1-t) T_{\mu^* \rightarrow \mu_0}(x) + t T_{\mu^* \rightarrow \mu_1}(x) - x \right| \right) d\mu^*(x) \\
  &= \int_{\Omega} f\left( \left| (1-t) \left[T_{\mu^* \rightarrow \mu_0}(x) - x \right] \right. \right. \\
      & \left. \left. \qquad \qquad + t \left[T_{\mu^* \rightarrow \mu_1}(x) - x \right] \right| \right) d\mu^*(x) \\
  &\leq \int_{\Omega} f\left( (1-t) \left| T_{\mu^* \rightarrow \mu_0}(x) - x \right|  \right. \\
   & \left. \qquad \qquad + t \left|T_{\mu^* \rightarrow \mu_1}(x) - x \right| \right) d\mu^*(x),
\end{align*}
where the last inequality is a consequence of the fact that $f$ is
non-decreasing.  Further, if $f$ is strictly convex in $\Omega$, we
have:
\begin{align*}
 & C_f(\mu_t, \mu^* ) \\
 &< \int_{\Omega} \left[ (1-t) f\left( \left| T_{\mu^* \rightarrow \mu_0}(x) - x \right|\right) \right.  \\
 						& \left. \qquad \qquad + t f\left( \left|T_{\mu^* \rightarrow \mu_1}(x) - x \right| \right) \right] d\mu^*(x) \\
  &= (1-t)\int_{\Omega}  f\left( \left| T_{\mu^* \rightarrow \mu_0}(x) - x \right|\right) d\mu^*(x) \\
  						&\qquad \qquad + t\int_{\Omega}f\left( \left|T_{\mu^* \rightarrow \mu_1}(x) - x \right| \right) d\mu^*(x) \\
  &= (1-t) C_f(\mu_0, \mu^*) + t C_f(\mu_1, \mu^*).
\end{align*} \oprocend
\end{proof}
We now establish the following result: 
\begin{proposition}[\bf \emph{$l$-smoothness of $C_f(\cdot, \mu^*)$}]
\label{prop:l_smooth_C_f}
Let the Fr{\'e}chet derivative of the functional $F(\mu) = C_f(\mu,
\mu^*)$ at $\mu \in \mathcal{P}^r(\Omega)$ be denoted as
$\xi_{\mu}$.
The derivative~$\xi_{\mu}$ satisfies:
\begin{small}
\begin{align*}
 & \left| \int_{\Omega} \left \langle \xi_{\mu_2} (T_{\mu^* \rightarrow \mu_2})  - \xi_{\mu_1} (T_{\mu^* \rightarrow \mu_1}) ,
      T_{\mu^* \rightarrow \mu_2} - T_{\mu^* \rightarrow \mu_1} \right \rangle d\mu^* \right| \\
 & \leq l \int_{\Omega} \left| T_{\mu^* \rightarrow \mu_2} - T_{\mu^* \rightarrow \mu_1} \right|^2 d\mu^*,
\end{align*}
\end{small}
\noindent
where $T_{\mu^* \rightarrow \mu_1}$ and $T_{\mu^* \rightarrow \mu_2}$ 
are the optimal transport maps (w.r.t.~$f$) 
from $\mu^*$ to $\mu_1$ and~$\mu_2$, respectively.
\end{proposition}
\begin{proof}
  Let $\phi_{\mu} = \frac{\delta C_f(\mu, \mu^*)}{\delta \mu}$ be the
  Kantorovich potential for the optimal transport from $\mu$ to
  $\mu^*$.  We now have the following relation~\cite[Theorem~1.17]{FS:15}:
\begin{align*}
	T_{\mu \rightarrow \mu^*} = \id - \left( \nabla h \right)^{-1} (\nabla \phi_{\mu}),
\end{align*}
where the function $h : \real^d \rightarrow \real$ is such that
$h(\mathbf{v}) = f(|\mathbf{v}|)$. It follows from the 
$l$-smoothness of~$f$ that the function~$h$ is also
$l$-smooth. From the above and $l$-smoothness 
of $h$, (with $\xi_\mu = \nabla \phi_\mu$) we get:
\begin{align*}
  &\left| \int_{\Omega} \left \langle \xi_{\mu_2}(T_{\mu^* \rightarrow \mu_2}) - \xi_{\mu_1}(T_{\mu^* \rightarrow \mu_1}) ,
      T_{\mu^* \rightarrow \mu_2} - T_{\mu^* \rightarrow \mu_1} \right \rangle  d\mu^* \right| \\
  &= \left| \int_{\Omega} \left \langle \nabla h \left( T_{\mu^* \rightarrow \mu_2} - \id \right) \right. \right. \\
       & \qquad \qquad \left. \left. - \nabla h \left(  T_{\mu^* \rightarrow \mu_1} - \id  \right) ,  T_{\mu^* \rightarrow \mu_2} -  T_{\mu^* \rightarrow \mu_1} \right \rangle d\mu^* \right| \\
  &\leq \int_{\Omega}\left| \left \langle \nabla h \left( T_{\mu^* \rightarrow \mu_2} - \id \right) \right. \right. \\
       & \qquad \qquad \left. \left. - \nabla h \left(  T_{\mu^* \rightarrow \mu_1} - \id  \right) ,  T_{\mu^* \rightarrow \mu_2} -  T_{\mu^* \rightarrow \mu_1} \right \rangle \right| d\mu^* \\
  &\leq l \int_{\Omega} \left| T_{\mu^* \rightarrow \mu_2} -  T_{\mu^* \rightarrow \mu_1}  \right|^2 d\mu^*.
\end{align*}  \oprocend
\end{proof}
\begin{proposition}[\bf \emph{Lipschitz continuous Fr{\'e}chet derivative of~$C_f(\cdot, \mu^*)$}]
\label{prop:lipschitz_frechet_derivative}
Let the Fr{\'e}chet derivative of the functional $F(\mu) = C_f(\mu,
\mu^*)$ at $\mu \in \mathcal{P}^r(\Omega)$ be denoted as $\xi_{\mu}$.
For any~$x, x' \in \Omega$, the derivative~$\xi_{\mu}$ satisfies:
\begin{align*}
	\left\| \xi_\mu(x) - \xi_\mu(x') \right\| \leq \lambda \left\| x - x' \right\|.
\end{align*}
\end{proposition}
\begin{proof}
From Proposition~\ref{prop:l_smooth_C_f}, we note that
$\xi_\mu = \nabla h \left( \id - T_{\mu \rightarrow \mu^*} \right)$
on~$\Omega$, where $\nabla h$ is Lipschitz continuous (with constant~$l$).
Therefore, for any~$x, x' \in \Omega$ we have:
\begin{align*}
	&\left\| \xi_\mu(x) - \xi_\mu(x') \right\| \\
	&= \left\| \nabla h \left( x - T_{\mu \rightarrow \mu^*}(x) \right) - \nabla h \left( x' - T_{\mu \rightarrow \mu^*}(x') \right) \right\| \\
	&\leq l \left\| x - x' \right\| + l \left\| T_{\mu \rightarrow \mu^*}(x) - T_{\mu \rightarrow \mu^*}(x') \right\| \\
	&\leq \underbrace{l \left( 1 + \lip(T_{\mu \rightarrow \mu^*}) \right)}_{\lambda} \left\| x - x' \right\|,
\end{align*}
where the final inequality follows from the fact that~$T_{\mu \rightarrow \mu^*}$
is a Lipschitz continuous map on~$\Omega$ (with Lipschitz constant $\lip(T_{\mu \rightarrow \mu^*})$)~\cite[Chapter~1.3]{FS:15}.
\end{proof}

\section{Proofs of Lemmas} \label{app:lemma_proofs}

\subsection{Proof of Lemma~\ref{lemma:conv_h,N}}
  We first recall that $F^{h,N}(\mathbf{x}) =
  F(\widehat{\mu}^{h,N}_{\mathbf{x}})$.  By the Glivenko-Cantelli
  Theorem~\cite{VSV:58} and
  Assumption~\ref{ass:kernel_properties}-(iv), we have:
  \begin{align*}
    \lim_{\substack{h \rightarrow 0, \\ N \rightarrow \infty}} \sup_{f
      \in \mathcal{M}} \left \lbrace
      \mathbb{E}_{\widehat{\mu}^{h,N}_{\mathbf{x}}} [f] -
      \mathbb{E}_{\mu} [f] \right \rbrace = 0, ~~a.s.
  \end{align*}   
  We denote the above as $\widehat{\mu}^{h,N}_{\mathbf{x}}
  \rightarrow_{\mathrm{u.a.s}}~\mu$, i.e.,
  $\widehat{\mu}^{h,N}_{\mathbf{x}}$ converges uniformly almost surely
  to~$\mu$ as $h \rightarrow 0$ and $N \rightarrow \infty$.  Note that
  this implies the (almost sure) weak convergence of
  $\{\widehat{\mu}^{h,N}_{\mathbf{x}}\}$ to $\mu$.  Therefore, by
  continuity of $F$ in the topology of weak convergence (which follows
  from the fact that~$F$ is Fre\'chet differentiable in the
  $L^2$-Wasserstein space), we have $\lim_{\substack{h \rightarrow 0,
      N \rightarrow \infty}} F^{h,N}(\mathbf{x}) = \lim_{\substack{h
      \rightarrow 0, N \rightarrow \infty}}
  F(\widehat{\mu}^{h,N}_{\mathbf{x}}) = F(\lim_{\substack{h
      \rightarrow 0, N \rightarrow \infty}}
  \widehat{\mu}^{h,N}_{\mathbf{x}}) = F(\mu)$, almost surely.

\subsection{Proof of Lemma~\ref{lemma:der_F^h,N_vs_F}}
  Let $\mathbf{x}(t) = (x_1(t), \ldots, x_N(t))$ be a curve in
  $\tilde{\Omega}_h^N$ parametrized by $t \in \real$, with
  $\dot{\mathbf{x}}(0) = \mathbf{v} = (\mathbf{v}_1, \ldots,
  \mathbf{v}_N)$, where $\mathbf{v}_i \in \real^d$ for all $i \in
  \lbrace 1, \ldots, N \rbrace$.  As $F^{h,N}$ is differentiable,
  partial derivatives exist and we can write:
\begin{align*}
  \frac{d}{dt} F^{h,N}(\mathbf{x}(0)) = \sum_{i=1}^N \left
    \langle \partial_i F^{h,N} (\mathbf{x}(0)) , \mathbf{v}_i \right
  \rangle.
\end{align*}
Since $F^{h,N}(\mathbf{x}) = F(\widehat{\mu}^{h,N}_{\mathbf{x}})$,
using the Fr{\'e}chet derivative of~$F$, we can write:
\begin{align*}
  \frac{d}{dt} F^{h,N}(\mathbf{x}(0))
  &= \frac{1}{N} \sum_{i=1}^N \int_{\Omega} \left \langle \nabla
      \varphi^{h,N}_{\mathbf{x}(0)}, \mathbf{v}_i \right \rangle~d\widehat{\mu}^h_{x_i(0)} \\
  &= \frac{1}{N} \sum_{i=1}^N \left \langle \int_{\Omega} \nabla \varphi^{h,N}_{\mathbf{x}(0)}
    ~d\widehat{\mu}^h_{x_i(0)} , \mathbf{v}_i \right \rangle.
\end{align*}
This holds for all $\mathbf{v} =
(\mathbf{v}_1, \ldots, \mathbf{v}_N)$ and  $\mathbf{x}(0) \in
\tilde{\Omega}_h^N$, thus, by uniqueness of the partial derivatives, it
holds that:
\begin{align*}
  \partial_i F^{h,N} (\mathbf{x}) = \frac{1}{N} \int_{\Omega} \nabla
  \varphi^{h,N}_{\mathbf{x}} ~d\widehat{\mu}^h_{x_i(0)},
\end{align*}  
where $\partial_i$ denotes the derivative w.r.t.~the $\supscr{i}{th}$
argument, and we consider any $\mathbf{x}(0) \in
\tilde{\Omega}^N_h$. From the previous expression:
\begin{align*}
  \partial_1 F^{h,N}(z, \eta) &= \frac{1}{N} \int_{\Omega} \nabla
  \varphi^{h,N}_{\mathbf{x}}~d\widehat{\mu}^h_{z}  \\
  &= \frac{1}{N} \int_{\supp \left(\widehat{\mu}^h_{z} \right)} \nabla
  \varphi^{h,N}_{\mathbf{x}}~d\widehat{\mu}^h_{z},
\end{align*}
where $z \in \tilde{\Omega}_h$, $\eta \in \tilde{\Omega}_h^{N-1}$,
$d\widehat{\mu}^h_{z} = \rho^{h}_z~\dvol$ with $\rho^{h}_z (x) =
K(x-z,h)$, and $\varphi^{h,N}_{\mathbf{x}}= \frac{\delta F}{\delta \nu}
\left. \right|_{\widehat{\mu}^{h,N}_{\mathbf{x}}}$, and the result follows.

\subsection{Proof of Lemma~\ref{lemma:neighborhood_cyclically_monotone}}
  For $\mathbf{x} \in \Delta_{\delta} \subset \Omega^N$, let
  $\mathbf{y} \in \mathring{\Omega}^N$ such that for all $i \in
  \lbrace 1, \ldots, N \rbrace$, we have $y_i \in B_{\delta/2} (x_i)$,
  where $B_{\delta/2} (x_i)$ is the open $\delta/2$-ball centered at
  $x_i \in \Omega$.  Now for any $j \in \lbrace 1, \ldots, N \rbrace$
  with $j \neq i$, we have $| y_i - x_j | = | y_i - x_i + x_i - x_j |
  \geq |x_i - x_j| - |y_i - x_i| > \delta - \delta/2 > \delta/2$,
  since $|x_i - x_j| > \delta$ as $\mathbf{x} \in \Delta_{\delta}$ and
  $|y_i - x_i| < \delta/2$.  Thus, among all (non-identity)
  permutations $\sigma$, we have:
\begin{align*}
  \frac{1}{N} \sum_{i=1}^N | x_i - y_{\sigma(i)} |^2 >
  \frac{\delta^2}{4} > \frac{1}{N} \sum_{i=1}^N | x_i - y_i |^2.
\end{align*}
Thus, we infer that $\mathbf{y} \in \Gamma_{\mathbf{x}}$  for an
arbitrary $\mathbf{y} \in \Omega^N \cap \Pi_{i=1}^N B_{\delta/2}
(x_i)$, and the result follows.

\subsection{Proof of Lemma~\ref{lemma:relaxation_atomless_measures}}
 The proof applies a generalization of Brenier's Theorem in~\cite{RM:95}. We consider
  convex functions $\chi_i : \Omega \rightarrow \real$, for $i \in
  \lbrace 1, \ldots, N \rbrace$ defined by:
\begin{align*}
  \chi_i(z) = \frac{1}{2} \left| z + y_i - x_i \right|^2.
\end{align*}
We note that the gradient of $\chi_i$, $\nabla \chi_i(z) = z + y_i -
x_i$ defines a map that transports the measure $\widehat{\mu}^h_{x_i}$
to $\widehat{\mu}^h_{y_i}$ simply by translation. In
  addition, this mapping defines a measure with cyclically monotone
  support and marginals $\widehat{\mu}^{h,N}_{\mathbf{x}} $ and
  $\widehat{\mu}^{h,N}_{\mathbf{y}} $. By the generalization of
Brenier's Theorem~\cite{RM:95} (c.f.~Theorem~12 and extensions on uniqueness) 
%
%
  a measure
  that has cyclic monotone support is both unique and optimal in the
  Monge-Kantorovich sense. Thus it coincides with the measure defined
  by the $\chi_i$ and the statement of the lemma follows.

\subsection{Proof of Lemma~\ref{lemma:alpha_smoothness_F^{h,N}}}
  From $l$-smoothness of~$F$, we have that the function 
  $\varphi = \left. \frac{\delta F}{\delta \nu} \right|_{\mu}$ 
  is continuously differentiable on $\Omega$ for all $\mu$.
  We note that for $x,y \in \tilde{\Omega}_h$,
  $\widehat{\mu}^{h}_{y}(z) = \widehat{\mu}^{h}_{x}(z + (x - y))$
  for all $z \in \supp \left( \widehat{\mu}^{h}_{y} \right)$. For
    any $\mathbf{x} \in \tilde{\Omega}^N_h$, we use $(x_i,
    \mathbf{x}_{-i}) \in \tilde{\Omega}_h\times
    \tilde{\Omega}^{N-1}_h$ to denote the vector with its first entry
    equal to the $\supscr{i}{th}$ component of $\mathbf{x}$ and all
    others equal to the remaining $N-1$ components of $\mathbf{x}$. We now have:
\begin{footnotesize}
\begin{align*}
  &\left| \left \langle \nabla F^{h,N}(\mathbf{y}) - \nabla F^{h,N}(\mathbf{x}), \mathbf{y} - \mathbf{x} \right \rangle \right| \\
  & = \left| \sum_{i=1}^N \left \langle \partial_1 F^{h,N}(y_i, \mathbf{y}_{-i}) - \partial_1 F^{h,N}(x_i, \mathbf{x}_{-i}), y_i - x_i \right \rangle \right| \\
  &= \left|  \frac{1}{N} \sum_{i=1}^N \left \langle \int_{\Omega} \nabla \varphi^{h,N}_{\mathbf{y}}(z) d\widehat{\mu}^{h}_{y_i}(z) - \int_{\Omega} \nabla \varphi^{h,N}_{\mathbf{x}}(z) d\widehat{\mu}^{h}_{x_i}(z),  \right. \right. \\
  & \qquad \qquad \qquad \left. \left. y_i - x_i \right \rangle \right| \\
  &= \left| \frac{1}{N} \sum_{i=1}^N \left \langle  \int_{\Omega} \left[ \nabla \varphi^{h,N}_{\mathbf{y}}(z + (y_i - x_i)) - \nabla \varphi^{h,N}_{\mathbf{x}}(z) \right] d\widehat{\mu}^{h}_{x_i}(z), \right. \right. \\
  & \qquad \qquad \qquad \left. \left. y_i - x_i \right \rangle \right| \\
  &= \left| \frac{1}{N} \sum_{i=1}^N  \int_{\Omega} \left \langle \nabla \varphi^{h,N}_{\mathbf{y}}(z + (y_i - x_i)) - \nabla \varphi^{h,N}_{\mathbf{x}}(z) ,  \right. \right. \\ 
  & \qquad \qquad \qquad \left. \left.	 y_i - x_i \right \rangle  d\widehat{\mu}^{h}_{x_i}(z) \right| \\
  &\leq l W_2^2(\widehat{\mu}^{h,N}_{\mathbf{x}},\widehat{\mu}^{h,N}_{\mathbf{y}}) \\
  &= \frac{l}{N} \| \mathbf{y} - \mathbf{x} \|^2,
\end{align*}
\end{footnotesize}
\noindent where the penultimate inequality results from the $l$-smoothness of
$F$.  Moreover, the final inequality results from the fact
that
$W_2(\widehat{\mu}^{h,N}_{\mathbf{x}},\widehat{\mu}^{h,N}_{\mathbf{y}})
\leq \| \mathbf{y} - \mathbf{x} \|$.

\subsection{Proof of Lemma~\ref{lemma:comparison_cyclically_monotone}}
  For $\mathbf{x} \in \Delta_{\delta}$ and $\mathbf{y} \in
  \Gamma_{\mathbf{x}}$, using the geodesic convexity of the functional
  $F$ and Lemma~\ref{lemma:first_order_convexity} with
  $\widehat{\mu}^{h,N}_{\mathbf{x}}$ as the reference measure, it
  follows that:
\begin{align*}
  &F^{h,N}(\mathbf{y}) = F(\widehat{\mu}^{h,N}_{\mathbf{y}}) \\
  &\geq F(\widehat{\mu}^{h,N}_{\mathbf{x}}) + \int_{\Omega} \left
    \langle \nabla \varphi^{h,N}_{\mathbf{x}} ,
    T_{\widehat{\mu}^{h,N}_{\mathbf{x}} \rightarrow
      \widehat{\mu}^{h,N}_{\mathbf{y}}}
    - \id \right \rangle d\widehat{\mu}^{h,N}_{\mathbf{x}} \\
  &= F(\widehat{\mu}^{h,N}_{\mathbf{x}}) + { \frac{1}{N} \sum_{i=1}^N
    \int_{\Omega} \left \langle \nabla \varphi^{h,N}_{\mathbf{x}} ,
      T_{\widehat{\mu}^{h,N}_{\mathbf{x}} \rightarrow \widehat{\mu}^{h,N}_{\mathbf{y}}} - \id \right \rangle d\widehat{\mu}^{h}_{x_i} }\\
  &= F(\widehat{\mu}^{h,N}_{\mathbf{x}}) + \frac{1}{N} \sum_{i=1}^N
  \int_{\supp(\widehat{\mu}^h_{x_i})} \left
    \langle \nabla \varphi^{h,N}_{\mathbf{x}} , y_i - x_i \right \rangle d\widehat{\mu}^{h}_{x_i} \\
  &= F(\widehat{\mu}^{h,N}_{\mathbf{x}}) + \frac{1}{N} \sum_{i=1}^N
  \left \langle \int_{\supp(\widehat{\mu}^h_{x_i})}
    \nabla \varphi^{h,N}_{\mathbf{x}} d\widehat{\mu}^{h}_{x_i} ~,~ y_i - x_i \right \rangle \\
  &= F^{h,N}(\mathbf{x}) + \sum_{i=1}^N \left \langle \partial_1
    F^{h,N}(x_i, \mathbf{x}_{-i}) , y_i - x_i \right \rangle,
\end{align*}
thereby establishing the claim.

\subsection{Proof of Lemma~\ref{lemma:strong_convexity_prox_grad_omega_N}}
  From Lemma~\ref{lemma:alpha_smoothness_F^{h,N}} on
  $\alpha$-smoothness of $F^{h,N}$, we have:
\begin{align*}
  \left| \left \langle \nabla F^{h,N}(\mathbf{y}) - \nabla
      F^{h,N}(\mathbf{x}) , \mathbf{y} - \mathbf{x} \right \rangle
  \right| \leq \alpha \| \mathbf{y} - \mathbf{x} \|^2.
\end{align*}
With $G^{h,N}_{\mathbf{x}}(\mathbf{z}) = \frac{1}{2\tau} \| \mathbf{x}
- \mathbf{z} \|^2 + F^{h,N}(\mathbf{z})$, we have:
{\small 
\begin{align*}
  &\left \langle \nabla G^{h,N}_{\mathbf{x}}(\mathbf{z}_1) - \nabla G^{h,N}_{\mathbf{x}}(\mathbf{z}_2), \mathbf{z}_1 - \mathbf{z}_2 \right \rangle \\
  &=  \left \langle \frac{1}{\tau} (\mathbf{z}_1 - \mathbf{z}_2) + \nabla F^{h,N}(\mathbf{z}_1) - \nabla F^{h,N}(\mathbf{z}_2) ~,~ \mathbf{z}_1 - \mathbf{z}_2  \right \rangle \\
  &= \frac{1}{\tau} \| \mathbf{z}_1 - \mathbf{z}_2 \|^2 +  \left \langle \nabla F^{h,N}(\mathbf{z}_1) - \nabla F^{h,N}(\mathbf{z}_2) ~,~ \mathbf{z}_1 - \mathbf{z}_2  \right \rangle \\
  &\geq \frac{1}{\tau} \| \mathbf{z}_1 - \mathbf{z}_2 \|^2 - \alpha \| \mathbf{z}_1 - \mathbf{z}_2 \|^2 \\
  &= \left( \frac{1}{\tau} - \alpha \right) \| \mathbf{z}_1 -
  \mathbf{z}_2 \|^2,
\end{align*}
}
thereby establishing the claim.

\subsection{Proof of Lemma~\ref{lemma:transport_many_particle}}
Let $\mathbf{x} = (x_1,\dots,x_N) \in
  \overline{\Omega}_h^N$ be a vector, and define the
  measures 
  $\otimes_{\substack{j=1, \\ j \neq i}}^{N}~ \delta_{x_j}
  \in \mathcal{P}(\Omega^{N-1})$, for $i \in \until{N}$, where
  $\otimes_{\substack{j=1, \\ j \neq i}}^{N} ~\delta_{x_j}$ is the
  product measure describing the independent coupling between the
  discrete measures $\delta_{x_i}$.  Observe that, given $\mathbf{x}$,
  we can obtain $x_i$ as a sample of $ \frac{1}{N} \sum_{i=1}^N
  \delta_{x_i}$ and $\mathbf{x}_{-i}$ as a sample of
  $\otimes_{\substack{j=1, \\ j \neq i}}^{N} ~\delta_{x_j}$.
  Thus, we rewrite~\eqref{eq:multi_agent_proximal_grad}
  as:
\begin{align*}
  &x^{+} = \arg \min_{z \in \overline{\Omega}_h} \frac{1}{2\tau} |x -
  z|^2 + F^{h,N}(z, \xi), \\
  & x \sim  \frac{1}{N} \sum_{i=1}^N \delta_{x_i}, ~~~~ \xi \sim
  \otimes_{\substack{j=1, \\ j \neq i}}^{N}~ \delta_{x_j}, \quad
   \forall \, i \in \until{n}.
\end{align*}
From the arguments in the proof of
Theorem~\ref{thm:convergence_multi_agent_local_min}, the above update
scheme can be expressed equivalently as:
\begin{align*}
  &x^{+} = x - \tau \partial_1 F^{h,N}(x^{+}, \xi), \\
  & x \sim \frac{1}{N} \sum_{i=1}^N \delta_{x_i}, ~~~~ \xi \sim
  \otimes_{\substack{j=1, \\ j \neq i}}^{N}~ \delta_{x_j}, \quad
  \forall \, i \in \until{n}.
\end{align*}
For $x_i \sim_{i.i.d} \mu$, we know that $\lim_{N \rightarrow \infty}
\frac{1}{N} \sum_{i=1}^N \delta_{x_i} = \mu$ uniformly, almost surely.
In addition, since $F$ is Fr{\'e}chet differentiable over the
  compact $\mathcal{P}(\Omega)$ and this differential is Lipschitz continuous
  it follows that $\nabla \left( \frac{\delta F}{\delta \nu} \right)$
  is bounded in $\mathcal{P}(\Omega)$. Since $\nabla
  \varphi^{h,N} = \nabla \left( \frac{\delta F}{ \delta \nu} \right)
  (\widehat{\mu}^{h,N}) \le K$, and $\int_\Omega K
  d\mu_z^h = \int_\Omega K \rho_z^h \dvol < C$ by
  Assumption~\ref{ass:kernel_properties}, 
  we can exchange integral and limits to derive
%

{\small
\begin{align*}
  \lim_{h \rightarrow 0} \lim_{N \rightarrow \infty} \partial_1
  F^{h,N}(z, \xi)
  &= \lim_{h \rightarrow 0} \lim_{N \rightarrow \infty} \int_{\Omega}
  \nabla \varphi^{h,N} d\mu^{h}_{z} 
 = \nabla \varphi (z),
\end{align*}
}
with $\xi \sim \otimes_{\substack{j=1, \\ j \neq i}}^{N}~\delta_{x_j}$ and $\varphi = \frac{\delta F}{\delta \nu}
\left. \right|_{\mu}$. This follows from
  $\lim_{\substack{h\rightarrow 0\\
    N \rightarrow \infty}}
  \frac{\delta F}{\delta \nu}(\widehat{\mu}^{h,N}) =
  \frac{\delta F}{\delta \nu} (\mu)$ and $\lim_{h \rightarrow 0} \mu^h_z = \delta_z$.
Thus:
\begin{align*}
	x^{+} &= x - \tau \nabla \varphi (x^{+}), 
	\qquad x \sim \mu.
\end{align*}
or equivalently:
\begin{align*}
          x^{+} = \arg \min_{z \in \Omega} ~&\frac{1}{2\tau} |x - z|^2
          + \varphi(z), 
\qquad x \sim \mu.
\end{align*}

\subsection{Proof of Lemma~\ref{lemma:conv_agg_obj_func}}
  From the Glivenko-Cantelli theorem, it follows that, as $N
  \rightarrow \infty$, the limit $\sum_{i=1}^N
  \mu^\star(\mathcal{V}_i) \delta_{x_i} \rightarrow \mu^\star$ holds
  almost surely, in the weak sense (from the expectation w.r.t.~$\sum_{i=1}^N
  \mu^\star(\mathcal{V}_i) \delta_{x_i}$ of any simple function).
  Thus, by the continuity of $C_f$:
{\small 
  \begin{equation*}
    \lim_{N \rightarrow \infty} \mathcal{H}_f(\mathbf{x})
    = \lim_{N \rightarrow \infty} C_f
    \left( \sum_{i=1}^N \mu^\star(\mathcal{V}_i) \delta_{x_i}, \mu^\star
    \right) = 0.
  \end{equation*}
}

\subsection{Proof of Lemma~\ref{lemma:conv_agg_obj_func_OT}}
  This can be seen from the following:

{\small
\begin{align*}
  &\bar{\mathcal{H}}_f(\mathbf{x}) = C_f \left( \frac{1}{N} \sum_{i=1}^N \delta_{x_i}~,~ \mu^\star \right) \\
  &= \min_{\substack{T: \Omega \rightarrow \lbrace x_i \rbrace_{i=1}^N
      \\ T_{\#}\mu^\star = \frac{1}{N} \sum_{i=1}^N \delta_{x_i}}}
  \int_{\Omega} 	f(|x - T(x)|) d\mu^\star(x) \\
  &= \min_{\substack{T: \Omega \rightarrow \lbrace x_i
      \rbrace_{i=1}^N\\
      \mu^\star( T^{-1}(\lbrace x_i \rbrace) ) = \frac{1}{N} }}
  \int_{\Omega} f(|x - T(x)|) d\mu^\star(x) 
.
\end{align*}
}
Similar to $\mathcal{H}_f(\mathbf{x})$, the functional
$\bar{\mathcal{H}}_f$ can be expressed as the sum of integrals over
certain space partition. However, this case involves a generalized
Voronoi partition $\left \lbrace \mathcal{W}_i \right
\rbrace_{i=1}^N$: 

\begin{align*}
  \mathcal{W}_i = \left \lbrace x \in \Omega \left| f(|x - x_i|) -
      \omega_i \leq f(|x- x_j|) - \omega_j \right. \right \rbrace,
\end{align*} 
where $\lbrace \omega_1, \ldots, \omega_N \rbrace$ are chosen such
that $\mu^\star(\mathcal{W}_i) = 1/N$ for all $i \in \lbrace 1,
\ldots, N \rbrace$.  We refer the reader to~\cite{JC:08-tac} for a
detailed treatment.  We can now write:
\begin{align*}
  \bar{\mathcal{H}}_f (\mathbf{x}) = \sum_{i=1}^N \int_{\mathcal{W}_i}
  f(|x-x_i|)~d\mu^\star(x).
\end{align*}
Now, by letting $x_i \sim_{i.i.d} \mu$, where $\mu \in
\mathcal{P}(\Omega)$ is any absolutely continuous probability measure,
in the limit $N \rightarrow \infty$, we have $\frac{1}{N} \sum_{i=1}^N
\delta_{x_i}$ converging uniformly almost surely to $\mu$. In this
way, by the continuity of $C_f$, we have:
\begin{align*}
  \lim_{N \rightarrow \infty} \bar{\mathcal{H}}_f (\mathbf{x}) = C_f
  \left( \mu, \mu^\star \right), ~~a.s.
\end{align*}

\section{On the continuous-time and many-particle limits}
\label{app:proof_prop_transport_cts_time_many_particle_continuity_equation}
We establish the model of transport in the continuous-time and many-particle limits
via the following proposition:
\begin{proposition}[\bf \emph{Model of transport in the continuous
    time and many-particle limits}]
	\label{prop:transport_cts_time_many_particle_continuity_equation}  
	Let $\Omega$ and $F$ satisfy the assumptions of
    Theorem~\ref{thm:conv_prox_recursion_wasserstein}.
    The following hold: \\
    {\bf \emph{(i) Convergence of update scheme:}}
    	The scheme~\eqref{eq:multi_agent_proximal_grad} converges 
    	in distribution to~\eqref{eq:multi-agent_limit_N_infty} 
    	in the limit $N \rightarrow \infty$. \\
{\bf \emph{(ii) Gradient flow:}} For every decreasing sequence $\lbrace \tau_n
\rbrace_{n \in \mathbb{N}}$ satisfying $\tau_0 < \frac{1}{l}$ and
$\lim_{n \rightarrow \infty} \tau_n = 0$, the sequence of solutions
$\lbrace x^n \rbrace_{n \in \mathbb{N}}$
to \eqref{eq:multi-agent_limit_N_infty} with
  corresponding $\{ \tau_n\}_{n \in \mathbb{N}}$ contains a
convergent subsequence, and the limit is a weak solution to the
gradient flow:
\begin{align}
	\partial_t X^t(x) = - \nabla \varphi_t(X^t(x)),
	\label{eq:cts_time_particle_grad_flow}
\end{align}
with $X^0(x) = x$, $\mu(t) = X^t_{\#} \mu_0$ and 
$\varphi_t = \left. \frac{\delta F}{\delta \nu} \right|_{\mu(t)}$. \\
{\bf \emph{(iii) Continuity equation:}}  Let~$T > 0$ and $\mathbf{v} \in
        L^{\infty}([0,T] \times \Lip(\Omega)^d)$, 
        and $\dot{x}_i(t) = \mathbf{v}(t, x_i(t))$ for any
        $t \in [0,T]$ and $i \in \mathbb{N}$, with $x_i(0) \sim_{i.i.d} \mu_0$.  
        Then, for $\mathbf{x}^N = (x_1, \ldots, x_N)$ for any $N \in \mathbb{N}$,
        the sequence $\lbrace \mathbf{x}^N \rbrace_{N \in \mathbb{N}}$ 
        converges in a distributional sense to a solution~$\mu$ of the continuity
        equation:
\begin{align}
  \frac{\partial \mu}{\partial t} + \nabla \cdot \left( \mu \mathbf{v}
  \right) = 0, \qquad \mu(0) = \mu_0.
	\label{eq:continuity}
\end{align}
\end{proposition}
\begin{proof}
\textit{(i)} Let $\mathbf{x} = (x_1,\dots,x_N) \in
  \overline{\Omega}_h^N$ be a vector, and define the
  measures 
  $\xi_{-i} =\otimes_{\substack{i=1, \\ x_i \neq x}}^{N}~ \delta_{x_i}
  \in \mathcal{P}(\Omega^{N-1})$, for $i \in \until{N}$, where
  $\otimes_{\substack{i=1, \\ x_i \neq x}}^{N} ~\delta_{x_i}$ is the
  product measure describing the independent coupling between the
  discrete measures $\delta_{x_i}$.  Observe that, given $\mathbf{x}$,
  we can obtain $x_i$ as a sample of $ \frac{1}{N} \sum_{i=1}^N
  \delta_{x_i}$ and $\mathbf{x}_{-i}$ as a sample of
  $\otimes_{\substack{i=1, \\ x_i \neq x}}^{N} ~\delta_{x_i}$.
  Thus, we rewrite~\eqref{eq:multi_agent_proximal_grad}
  as:
\begin{align*}
  &x^{+} = \arg \min_{z \in \overline{\Omega}_h} \frac{1}{2\tau} |x -
  z|^2 + F^{h,N}(z, \xi_{-i}), \\
  & x \sim  \frac{1}{N} \sum_{i=1}^N \delta_{x_i}, ~~~~ \xi \sim
  \otimes_{\substack{i=1, \\ x_i \neq x}}^{N}~ \delta_{x_i}, \quad
  \forall \, i \in \until{n}.
\end{align*}
From the arguments in the proof of
Theorem~\ref{thm:convergence_multi_agent_local_min}, the above update
scheme can be expressed equivalently as:
\begin{align*}
  &x^{+} = x - \tau \partial_1 F^{h,N}(x^{+}, \xi_{-i}), \\
  & x \sim \frac{1}{N} \sum_{i=1}^N \delta_{x_i}, ~~~~ \xi \sim
  \otimes_{\substack{i=1, \\ x_i \neq x}}^{N}~ \delta_{x_i}, \quad
  \forall \, i \in \until{n}.
\end{align*}
%
For $x_i \sim_{i.i.d} \mu$, we know that $\lim_{N \rightarrow \infty}
\frac{1}{N} \sum_{i=1}^N \delta_{x_i} = \mu$ uniformly, almost surely.
 In addition, if $F$ is Fr\'echet differentiable over the
  compact $\mathcal{P}(\Omega)$ and this differential is continuous
 then, $\frac{\delta F}{\delta \nu}$
  is bounded in $\mathcal{P}(\Omega)$. Since $\nabla
  \varphi^{h,N}_{(x_i,\xi_{-i})} = \frac{\delta F}{ \delta \nu}
  (\widehat{\mu}^{h,N}_{(x_i,\xi_{-i})}) \le K$, and $\int_\Omega K
  d\mu_z^h = \int_\Omega K \rho_z^h \dvol < C$ by
  Assumption~\ref{ass:kernel_properties}, 
  we can exchange integral and limits
 to derive

{\small
\begin{align*}
  \lim_{h \rightarrow 0} \lim_{N \rightarrow \infty} \partial_1
  F^{h,N}(z, \xi_{-i})
  &= \lim_{h \rightarrow 0} \lim_{N \rightarrow \infty} \int_{\Omega}
  \nabla \varphi^{h,N}_{(z,\xi_{-i})} d\mu^{h}_{z} 
 = \nabla \varphi (z),
\end{align*}
}
with $\xi \sim \otimes_{\substack{i=1, \\ x_i \neq
    x}}^{N}~\delta_{x_i}$ and $\varphi = \frac{\delta F}{\delta \nu}
\left. \right|_{\mu}$. This follows from
  $\lim_{\substack{h\rightarrow 0\\
    N \rightarrow \infty}}
  \frac{\delta F}{\delta \nu}(\widehat{\mu}^{h,N}_{(z,\xi_{-i})}) =
  \frac{\delta F}{\delta \nu} (\mu)$ and $\lim_{h \rightarrow 0} \mu^h_z = \delta_z$.
Thus:
\begin{align*}
	x^{+} &= x - \tau \nabla \varphi (x^{+}), 
	\qquad x \sim \mu.
\end{align*}
or equivalently:
\begin{align}
          x^{+} = \arg \min_{z \in \Omega} ~&\frac{1}{2\tau} |x - z|^2
          + \varphi(z), 
\qquad x \sim \mu.
\end{align}

\textit{(ii)} Let $X : [0,T] \times \Omega \rightarrow \Omega$ be the
flow corresponding to $\mathbf{v} \in L^{\infty}([0,T] \times
\Lip(\Omega)^d)$, 
such that:
\begin{align*}
	\partial_t X^t(x) = \mathbf{v}(t, X^t(x)),
\end{align*}
with $X^0(x) = x$, and let $\mu(t) = X^t_{\#} \mu_0$ be the
pushforward of $\mu_0$ by the flow at time $t \in [0,T]$.
Dropping the superscript~$N$ from~$\mathbf{x}^N$ for conciseness, and
recalling from~\eqref{eq:density_kernel},
$\widehat{\mu}^{h,N}_{\mathbf{x}}(x) = \frac{1}{N} \sum_{i=1}^N
K(x-x_i, h)$. Let~$\rho^{h,N}_{\mathbf{x}}$ be the corresponding
density function.  Now, with $\widehat{\mu}^N_{\mathbf{x}(t)} =
\frac{1}{N} \sum_{i=1}^N
\delta_{x_i(t)}$, 
we can write:
\begin{align*}
  \frac{\partial \rho^{h,N}_{\mathbf{x}}}{\partial t} (t,x) &= -
  \frac{1}{N} \sum_{i=1}^N \nabla_x K (x-x_i(t),h) \cdot \mathbf{v}(t,
  x_i)\\
  &= - \int_{\Omega} \mathbf{v}(t,z) \cdot \nabla_x K (x-z,h)
  ~d\widehat{\mu}^N_{\mathbf{x}(t)}(z).
\end{align*}
%
%
%
We note that, by the Glivenko-Cantelli Theorem, the measure
$\widehat{\mu}^N_{\mathbf{x(0)}}$ converges uniformly almost surely to
$\mu_0$ as $N \rightarrow \infty$. This implies that for every~$t$,
$\widehat{\mu}^N_{\mathbf{x}(t)}$ converges uniformly almost surely to
the pushforward $\mu(t) = X^t_{\#} \mu_0$ as $N \rightarrow
\infty$. Therefore, by the  dominated convergence theorem
and Assumption~\ref{ass:kernel_properties}-(4), we have:
\begin{align}
	\begin{aligned}
          &\lim_{h \rightarrow 0} \lim_{N \rightarrow \infty}
          \rho^{h,N}_{\mathbf{x}}(t,x) = \lim_{h \rightarrow 0}
          \lim_{N \rightarrow \infty} \int_{\Omega} K(x-z,h)~
          d\widehat{\mu}^N_{\mathbf{x}(t)}(z) \\
          &=_{a.s.} \lim_{h \rightarrow 0} \int_{\Omega} K(x-z,h)~
          d\mu(t,z) \\ &= \lim_{h \rightarrow 0} \int_{\Omega} \rho(t,z)
          K(x-z,h) ~\dvol(z) = \rho(t,x).
	\end{aligned}
		\label{eq:limit_density}
\end{align}
From the above, we get that for a smooth test function 
$\zeta \in C^{\infty}([0,T] \times \Omega)$ such that $\zeta(0) = 0 = \zeta(T)$,
and again by the dominated convergence theorem and
Assumption~\ref{ass:kernel_properties}-(4):

{\small 
\begin{align}
  &\int_0^T \int_{\Omega} \frac{\partial \zeta}{\partial t}(t,x)
  \rho(t,x) \dvol(x) \hspace*{-0.2ex}\diff \hspace*{-0.4ex} t \nonumber \\
  &=_{a.s.} \lim_{h \rightarrow 0} \lim_{N \rightarrow \infty}
  \int_{0}^T \hspace*{-1.5ex}\int_{\Omega} \frac{\partial
    \zeta}{\partial t}(t,x) \rho_{\mathbf{x}}^{h,N}(t,x) \dvol(x)
  \hspace*{-0.2ex}\diff \hspace*{-0.4ex}t \nonumber
  \\
  &= \lim_{h \rightarrow 0} \lim_{N \rightarrow \infty} - \int_0^T
  \hspace*{-1.5ex} \int_{\Omega} \zeta(t,x) \frac{\partial
    \rho_{\mathbf{x}}^{h,N}}{\partial t}(t,x) \dvol(x)
  \hspace*{-0.2ex}\diff \hspace*{-0.4ex}t \nonumber\\
  &= \lim_{\substack{h \rightarrow 0\\
      N \rightarrow \infty}} \int_{0}^T \hspace*{-1.5ex} \int_{\Omega}
  \hspace*{-1ex}\zeta(t, x)\hspace*{-1ex} \int_{\Omega}\hspace*{-1ex}
  \mathbf{v}(t,z) \cdot \nabla_x K (x-z,h)
  d\widehat{\mu}^N_{\mathbf{x}(t)}(z) \dvol(x) \hspace*{-0.2ex}\diff
  \hspace*{-0.4ex}t \nonumber
\\
  &= \hspace*{-0.4ex}-\hspace*{-1ex}\lim_{\substack{h \rightarrow 0\\N
      \rightarrow \infty}} \hspace*{-1ex}\int_0^T\hspace*{-1.5ex}
  \int_{\Omega^2} 
  \hspace*{-1.5ex}\rho^{h,N}_{\mathbf{x}}(t,z) \nabla \zeta(t,x)
  \cdot
  \mathbf{v}(t,z) K(x-z,h) \hspace*{-0.2ex} \dvol(z,x)
  \hspace*{-0.2ex}\diff \hspace*{-0.4ex}t \nonumber \\
  &= - \int_{0}^T \int_{\Omega} \nabla \zeta(t,x) \cdot \mathbf{v}(t,z)
  \rho(t,x) \dvol(x)\diff \hspace*{-0.4ex}t.
\label{eq:continuity_distributional}
\end{align} 
}

The above is the distributional equivalent of the
continuity equation~\eqref{eq:continuity}. 
We now note from \eqref{eq:multi-agent_limit_N_infty} that $\frac{1}{2\tau}
|x^{+} - x|^2 \leq \varphi(x) - \varphi(x^+)$ (where $\tau <
\frac{1}{l}$, $x \sim \mu$ and $\varphi = \left. \frac{\delta
    F}{\delta \nu} \right|_{\mu}$).  Let $\lbrace \tau_{n} \rbrace_{n
  \in \mathbb{N}}$ be a decreasing sequence such that $\tau_0 <
\frac{1}{l}$ and $\lim_{n \rightarrow \infty} \tau_n = 0$. Let
$\lbrace x^{n}(k) \rbrace_{k,n\in \mathbb{N}}$ be the sequence of
solutions to \eqref{eq:multi-agent_limit_N_infty}, for each $\tau_n$,
starting from the same initial condition $x_0$.  We note that $x^n(k)
\in \Omega$ for all $n, k \in \mathbb{N}$.  We now define continuous
curves $\bar{x}^n: \realnonnegative \rightarrow \Omega$ 
 such that $\bar{x}^n(t) = \left( 1 + \lfloor \frac{t}{\tau_n} \rfloor - t
\right) x^n(\lfloor \frac{t}{\tau_n} \rfloor) + \left( t - \lfloor
  \frac{t}{\tau_n} \rfloor \right) x^n(\lfloor \frac{t}{\tau_n}\rfloor
+ 1)$.  From the compactness of $\Omega$, we get that the sequence
$\lbrace \bar{x}^{n} \rbrace$ is uniformly bounded.  Moreover, we have
that for $0 \leq k' \leq k$, and a fixed $n\in \mathbb{N}$, that:

{\small 
\begin{align*}
  &\left| x^n \left( k \right) - x^n \left( k' \right) \right|
  \leq \sum_{m=k'+1}^k \left| x^n(m) - x^n(m-1) \right|  \\
  &\leq  \left( \sum_{m=k'+1}^k \left| x^n(m) - x^n(m-1) \right|^2 \right)^{1/2} \left( k - k' \right)^{1/2} \\
  &= \sqrt{2\tau_n}\left( \sum_{m=k'+1}^k \frac{1}{2\tau_n}  \left| x^n(m) - x^n(m-1) \right|^2 \right)^{1/2} \left( k- k'  \right)^{1/2} \\
  &\leq \sqrt{2 \tau_0} \left( \sum_{m \in \mathbb{N}} \frac{1}{2\tau_n}  \left| x^n(m) - x^n(m-1) \right|^2 \right)^{1/2} \left( k - k'   \right)^{1/2} \\
  &\leq \sqrt{2 \tau_0} \left( \varphi_0 (x_0) - \lim_{m \rightarrow
      \infty} \varphi_{m}(x^n(m)) \right)^{1/2} \left( k - k'
  \right)^{1/2} ,
\end{align*}
}where $\varphi_m = \left. \frac{\delta F}{\delta \nu}
\right|_{\mu(m)}$, $\mu(m) = {T_{m-1}}_{\#} \ldots {T_0}_{\#}\mu_0$,
for all $m \ge 0$, and $T_{k} = \left( \id + \tau_n \nabla
  \varphi_{k-1} \right)^{-1}$ for $k \ge 1$.  From
Theorem~\ref{thm:target_dynamics_euclidean}, it follows that $\lim_{m
  \rightarrow \infty} \varphi_{m} = \left. \frac{\delta F}{\delta \mu}
\right|_{\mu^*} = C$, a constant function. We therefore have:
{\small 
\begin{align}
  \left| x^n \left( k \right) - x^n \left( k' \right) \right| \leq
  \sqrt{2\tau_0} \left( \varphi_0 (x_0) - C \right)^{1/2} \left( k -
    k' \right)^{1/2}.
	\label{eq:norm_bound_discrete_sol}
\end{align}
}
It now follows for $0 \leq t' \leq t$ that:
{\small
\begin{align*}
  &\left| \bar{x}^n(t) - \bar{x}^n(t') \right| \\ &= \left|
    \bar{x}^n(t) - x^n \left(\left\lfloor \frac{t}{\tau_n}
      \right\rfloor \right) + x^n\left(\left\lfloor \frac{t}{\tau_n}
      \right\rfloor \right) - x^n\left(\left\lfloor
        \frac{t'}{\tau_n}\right\rfloor + 1
    \right) \right. \\
  & \qquad \left. \qquad \qquad +~~ x^n\left(\left\lfloor
        \frac{t'}{\tau_n}\right\rfloor + 1  \right) - \bar{x}^n(t')  \right| \\
  &\leq \left| \bar{x}^n(t) - x^n \left(\left\lfloor \frac{t}{\tau_n}
      \right\rfloor \right) \right| + \left| x^n\left(\left\lfloor
        \frac{t}{\tau_n} \right\rfloor \right) - x^n\left(\left\lfloor
        \frac{t'}{\tau_n}\right\rfloor + 1 \right)
  \right| \\ & \qquad  \qquad \qquad    +~~  \left| x^n\left(\left\lfloor \frac{t'}{\tau_n} \right\rfloor + 1 \right) - \bar{x}^n(t')  \right| \\
  &\leq \left| \bar{x}^n(t) - x^n \left(\left\lfloor \frac{t}{\tau_n}
      \right\rfloor \right) \right|  + \sum_{m = \left\lfloor
      \frac{t'}{\tau_n} \right\rfloor + 1}^{\left\lfloor
      \frac{t'}{\tau_n}\right\rfloor - 1} \left| x^n(m+1) - x^n(m)
  \right| \\   & \qquad  \qquad \qquad    +~~  \left| x^n\left(\left\lfloor \frac{t'}{\tau_n} \right\rfloor + 1 \right) - \bar{x}^n(t')  \right|  \\
  &\leq \sqrt{2\tau_0} \left( \varphi_0 (x_0) - C \right)^{1/2} \left(
    t - t' \right),
\end{align*}
} where the final inequality follows from the definition of
$\bar{x}^n(t)$ and \eqref{eq:norm_bound_discrete_sol}.  The above
inequality holds for any $n \in \mathbb{N}$, and it thereby follows
that the family $\lbrace \bar{x}^n (t)\rbrace$ is uniformly
equicontinuous.  Therefore, from the Arzel{\'a}-Ascoli
theorem~\cite{WR:64}, we have that $\lbrace \bar{x}^n(t) \rbrace$
contains a uniformly convergent subsequence, and let the limit be the
curve $\lbrace x(t) \rbrace_{t \in \realnonnegative}$. 
Moreover, by isolating the uniformly convergent
subsequence and using a smooth test function $\zeta \in
C^{\infty}([0,T])$, we have:
\begin{align*}
  \int_{[0,T]} &\frac{d\zeta}{dt} x(t) \diff \hspace*{-0.4ex}t =
  \lim_{n \rightarrow \infty} \int_{[0,T]} \frac{d\zeta}{dt}
  \bar{x}^n(t) \diff
  \hspace*{-0.4ex}t \\
  & = \lim_{n \rightarrow \infty} \int_{[0,T]} \left( \frac{\zeta(t + \tau_n) - \zeta(t)}{\tau_n} \right) \bar{x}^n(t)  \diff \hspace*{-0.4ex}t \\
  &= \lim_{n \rightarrow \infty} \int_{[\tau_n,T]} \zeta(t) \left(
    \frac{\bar{x}^n(t-\tau_n) - \bar{x}^n(t)}{\tau_n} \right) \diff
  \hspace*{-0.4ex}t\\ & = \int_{[0,T]} \zeta(t) \nabla \varphi_t(x(t))
  \diff \hspace*{-0.4ex}t,
\end{align*}
where the final equality follows from
\eqref{eq:multi-agent_limit_N_infty}. This is the weak form of the
gradient flow~\eqref{eq:cts_time_particle_grad_flow}.
\end{proof}

\end{appendices}

\end{document}